\documentclass[pagesize]{scrartcl}

\usepackage{mystyle}

\title{A Discretization Scheme for BSDEs with Random Time Horizon}
\author{
    Frank T. Seifried\thanks{University of Trier, Department IV -- Mathematics, Universit\"{a}tsring 19, 54296 Trier, Germany. M.\ W\"{u}rschmidt gratefully acknowledges financial support from the German Research Foundation (DFG) within the Research Training Group 2126: Algorithmic Optimization.} 
    \and 
    Maximilian W\"{u}rschmidt\footnotemark[1]
}
\date{}

\begin{document}
\allowdisplaybreaks
\maketitle

\begin{abstract}
    We analyze a natural extension of the backward Euler approximation for a class of BSDEs with Lipschitz generators and random (unbounded) time horizons. We derive strong error bounds in terms of the underlying stepsize; the distance between the continuous terminal time and a discrete-time approximation; the distance between the terminal condition and a respective approximation; and an integrated distance depending on an approximation of the time component of the generator - all are scaled by the exponential of the maximal terminal time. As application we consider decoupled FBSDEs on bounded domains. We use an Euler-Maruyama scheme to approximate the diffusion and further refine our error bounds to only depend on the distance of the exit times. 
    

\end{abstract}

\section{Introduction}

Backward stochastic differential equations with random (unbounded) terminal time condition were first investigated in the seminal works \cite{PardouxRandomTimeBSDE, Peng1991}, and further developed in \cite{briand2003lp, BriandHuHomogenization, confortola2008quadratic, lin2020second, papapantoleon2018existence, papapantoleon2023stability, royer2004bsdes, toldo2006stability}.
In this note we analyze a time-discrete approximation scheme which naturally extends the backward implicit Euler method on bounded time horizons. We establish error bounds in terms of the underlying stepsize and externally given discrete-time guesses of the BSDE terminal condition, terminal time and time-dependency on the generator. 

%
To the best of our knowledge in literature a discrete-time scheme for a BSDE with random unbounded time horizon is not investigated. There are several works concerned with the strong rate of approximation for constant time horizons, see e.g. \cite{bender2007forward, bender2008time, bouchard2008discrete, bouchard2004discrete, gobetLabart2007, gobet2024numerical, zhang2004numericalScheme}. Furthermore we refer to \cite{Bouchard_strong_BSDE_on_domain} where the convergence rate is established for BSDEs with random (bounded) time horizon. 
In addition and for completeness we mention \cite[Section 3.2]{matoussi2024_ErgodicDeep} where the authors 
pursue an extension of \cite{Bouchard_strong_BSDE_on_domain} to the case of a random unbounded time horizon. Their analysis proceeds via an iterative scheme while certain challenges specific to the unbounded case do not appear to be the focus of their discussion.

%
Our main results are error bounds for a general BSDE approximation scheme. The framework within we analyze the random unbounded time horizon reformulates the iteratively defined scheme (from \cite{Bouchard_strong_BSDE_on_domain} as well as \cite{bouchard2004discrete, zhang2004numericalScheme}) into a fixed-point problem in a $L^2$ sequence space. For a family of equidistant discrete-time grids we analyze the deviation of our discrete-time fixed-point scheme (satisfying \eqref{def:FixpointScheme}) from the continuous time solution of the BSDE \eqref{def:BSDE_general}. We establish two types of error bounds:
\begin{itemize}
    \item Theorem~\ref{thm:BSDE_MainResult_no_SUPREMUM} is a bound of the maximal $L^2$-distance considering each discrete-time individually. 
    Exclusively using path properties of the BSDE we derive Corollary~\ref{corollary:BSDE_Rate_Classic} and Corollary~\ref{corollary:BoundedMartingaleIntegrand} as pendant to the error bounds from the bounded setting, see e.g. \cite{bender2008time, bouchard2004discrete, Bouchard_strong_BSDE_on_domain, zhang2004numericalScheme}. 
    \item Theorem~\ref{thm:BSDE_MainResult} is a bound of the $L^2$-distance of the maximal deviation, i.e. the expected distance taking into account the entire discrete-time grid at once. Assuming the solution of the BSDE satisfies Kolmogorov's tightness criterion for sufficiently large moments we derive Theorem~\ref{thm:Kolmogorov_BSDE_Error_Bound}: A bound of maximal distance between the continuous solution and the constantly interpolated discrete-time scheme taking into account $[0,\infty)$ under a single expectation. 
\end{itemize}
All error bounds consist of a rate in terms of the stepsize and $L^p$-distances of the terminal conditions, the terminal times and a difference that arises from a discrete-time approximation of the BSDE generator; all expressions are scaled by an exponential (with positive scalar) of the maximal terminal time. Since the discrete terminal time is depending on the stepsize we require existence of such exponential moments. In the case of random unbounded time horizons some constants require additional attention to ensure they are uniform boundedness for all admissible stepsizes. This is highlighted in Corollary~\ref{corollary:MainResults_AfterHoelder} where merely error bounds with a rate remain.

%
We apply our main results to a decoupled FBSDE on a bounded domain where the forward diffusion is approximated by a classical Euler-Maruyama scheme. An illustrative presentation (of Theorem~\ref{thm:BSDE_MainResult}) is Theorem~\ref{thm:FBSDE_ErrorBound} where the general error bound is refined to the rate of convergence $\tfrac{1}{8}$ and additional error terms consisting of $L^p$-distances between continuous and discrete exit times. 
To our knowledge \cite{BGG_forward_exit} is the only work that considers general (unbounded) exit time convergence rates. While we present this application under the assumption of existence of the necessary exponential moments, we also provide a sufficient possibility to enforce this. The latter methodology gives rise to a systematic error that remains to be fully understood; its analysis is deferred to future work beyond the scope of this paper.

%
Our discrete-time BSDE approximation scheme is a natural extension of its random bounded time horizon counterpart, see \cite{Bouchard_strong_BSDE_on_domain}. The fundamental difference is that the unbounded time horizon transforms the scheme into a countable infinite sequence whereas a bound on the time horizons dictates a maximum number of discrete-times. Despite this we accomplish the same regularity properties and particularly our proposed discrete-time approximation is \textit{adapted}. With this in mind we emphasize that our analysis provides a theoretical presentation of intrinsic limitations and serves as surrogate tool in the convergence analysis of potential numerical algorithms which for instance rely on forward simulation circumventing the need to use conditional expectations. Such an approach is e.g. carried out in \cite{gao2023convergence, han2020convergence, Knochenhauer2021ConvergenceRates} where the error bounds from auxiliary schemes established in \cite{bender2008time, gobetLabart2007, zhang2004numericalScheme} are used.

%
We confine our analysis to BSDEs with generators that are not depending on the martingale integrand. The technical reasoning for this is twofold: On the one hand it remains unknown whether existence of the scheme can be achieved while preserving the required regularity to establish the error bounds. On the other hand (assuming the scheme is well-defined) a discrete-time counterpart of a \textit{Zhang-$L^2$-regularity} error term (see \cite[Section 3]{zhang2004numericalScheme}) arises when attempting to establish Proposition~\ref{prop:AdaptedGronwall}.

%
Our framework remains widely applicable despite the restriction to generators which are not depending on the martingale integrand. For instance in multilevel Monte-Carlo approximations for exit times, see e.g. \cite{giles2018multilevel, gobet2010stopped, higham2013mean}. In utility optimization with random time horizon, see e.g. \cite{aurand2023epstein, jeanblanc2015utility}. Concerning counterparty risk, see e.g. \cite{crepey2015bsdes}. In applications pertaining semilinear second-order elliptic PDEs. Completely deterministic applications such as solving incompressible Navier-Stokes  equations, see e.g. \cite{lejay2020forward}. Furthermore with a direct link to the probabilistic setting including insurance mathematics, see e.g. \cite{frey2022deep, kremsner2020deep} or stock repurchase pricing, see e.g. \cite{hamdouche2022policy}.

%
\textbf{Outline: } In Section~\ref{sec:BSDE_ErrorBoundMain} we introduce the discrete-time approximation scheme for BSDEs with random terminal time. We state the required assumptions and present our main result Theorem~\ref{thm:BSDE_MainResult_no_SUPREMUM} and Theorem~\ref{thm:BSDE_MainResult} which are error bounds regarding the discrete-time grid, together with the immediate Corollary~\ref{corollary:MainResults_AfterHoelder} which highlights the stepsize-uniform nature of our analysis. 
In Section~\ref{sec:BSDE_StrongErrorBound} we present the error bounds regarding deviations on the continuous time line while considering a constant interpolation of the discrete-time approximation in Corollaries~\ref{corollary:BSDE_Rate_Classic} and~\ref{corollary:BoundedMartingaleIntegrand} as classic strong error bounds. Moreover as main result of Section~\ref{sec:BSDE_StrongErrorBound} we present Theorem~\ref{thm:Kolmogorov_BSDE_Error_Bound} the uniform-in-time strong error bound where the supremum considered under the expectation.
The proof of our main results (of Section~\ref{sec:BSDE_ErrorBoundMain}) rely on a continuous and adapted interpolation of the discrete-time scheme. The main component is presented in Section~\ref{sec:proof_AdaptedGronwall}. The proofs rely on integrability conditions which are derived in Sections~\ref{sec:Existence}-\ref{Sec:RegularityInterpolation}. To be precise: Section~\ref{sec:Existence} establishes unique existence of our proposed scheme and in Section~\ref{sec:Majorant_DiscreteTime} we derive a majorant for our scheme which yields the required integrability. This is based on a pathwise discrete Gronwall lemma which is presented in the Supplements, see Appendix~\ref{sec:generalSupplementsAndPathwiseGronwall}.
In Section~\ref{Sec:RegularityInterpolation} this integrability is extended for an adapted interpolation on which the main proofs are based. 
The assumption on uniform integrability of exponential exits times is discussed in Section~\ref{sec.ExponentialMoments} where with stand-alone proofs we present a sufficient possibility to ensure existence of the necessary moments. 
The application to the Markovian setting can be found in Section~\ref{sec:Application_EM} where we present the illustrative result Theorem~\ref{thm:FBSDE_ErrorBound} as a special case of Theorem~\ref{thm:MarkovianFBSDE_StrongBound}; both are concerned with the uniform-in-time strong error bound. Furthermore we derive Theorem~\ref{thm:MarkovianFBSDE_ClassicBound} as pendant to the usually considered classic strong error bound. 
The corresponding proof is based on a typical adapted interpolation. The corresponding intermediate results are presented in Section~\ref{sec:Hoelder_Bounds_Euler}. Those results are based on the Hölder regularity properties of the Euler-Maruyama interpolation. Auxiliary results are presented in the supplements. We refer to Appendix~\ref{sec:HoelderCoeff} for a result on maximal deviation between two stopping times for a general process satisfying Kolmogorov's criterion and to Appendix~\ref{sec:StrongEulerBounds_BGG} for classical strong bounds of the Euler-Maruyama interpolation.

\section{BSDE Discretization Scheme and Discrete-Time Error Bounds}\label{sec:BSDE_ErrorBoundMain}
Throughout the paper let $(\Omega, \F, \mathfrak{F}, \prob)$ denote a filtered probability space and suppose the filtration is generated by a $d\in\nat$ dimensional Brownian motion $W$ augmented by all $\prob$-nullsets. 
For a terminal time $\tau$ satisfying $\exp(m \tau)\in L^1(\F_{\tau})$ for some $m>0$ and terminal condition $\xi\in L^p (\F_{\tau})$ for some $p \geq 4$ we consider BSDEs 
\begin{equation}\label{def:BSDE_general}
	Y_t = \xi + \int_{t\wedge\tau}^{\tau} f(s,Y_s) \de s - \int_{t\wedge\tau}^{\tau} Z_s \de W_s \qquad t \geq 0
\end{equation}
For a given stepsize $h>0$ let $\unboundedTimeGrid\defined\{nh\ |\ n \in\nat_0\}$ denote the equidistant discretization of $[0,\infty)$. 
Moreover let $\StoppDiscrete$ denote a finite $\unboundedTimeGrid$-valued stopping time; corresponding to a $\unboundedTimeGrid$-valued approximation of the continuous terminal time $\tau$. 
For a general discrete-time terminal condition $\bar{\xi}\in L^p (\F_{\StoppDiscrete})$ and a discrete-time approximation of the generator $f$ of the BSDE \eqref{def:BSDE_general} denoted by $\fd{h}\colon \unboundedTimeGrid\times\R\to \R$ we define the approximating sequence\footnote{
    We emphasize that this scheme is not to be understood iteratively but as a sequence composed of the terminal condition and a solution to a fixed-point problem. For details about the existence see Section~\ref{sec:Existence}. 
}
\begin{equation}\label{def:FixpointScheme}
	\Y_{t} \defined \1_{\{t \geq \StoppDiscrete\}} \bar{\xi} 
	+ \1_{\{t < \StoppDiscrete\}} \big( \E_{t} [\Y_{t+h}] + h \fd{\unboundedTimeGrid} (t, \Y_{t})\big) \quad t\in\unboundedTimeGrid 
\end{equation}

Our main results Theorem~\ref{thm:BSDE_MainResult_no_SUPREMUM} and Theorem~\ref{thm:BSDE_MainResult} are postulated under the assumption of unique existence of a solution $(Y,Z)$ satisfying standard regularity conditions. Subsequently we impose the relevant assumptions for our proofs;
\begin{bsde_assumption}\label{bsde:assumption}
    Let $f\colon [0,\infty)\times\Omega\times\R \to \R$ be Lipschitz in $y$ with Lipschitz constant $\lipschitz{f} > 0$. Moreover assume that $t\mapsto f(t,Y_t)$ is progressive and satisfies 
    $\int_0^{\tau} f(s,0)^p \de s \in L^1$. \close
\end{bsde_assumption}

\begin{stepsize_condition}\label{stepsize:assumption}
    Assume the family of equidistant discrete-time grids $(\unboundedTimeGrid_{h})_{h\in(0,h_0]}$ has maximal stepsize 
    $        h_0 < \min \{ \lipschitz{f},\ \tfrac{1}{12\lipschitz{f}} \}$. \close
\end{stepsize_condition}

Next is the decisive condition to ensure existence of a solution to \eqref{def:FixpointScheme} satisfying the required regularity properties.
The principal steps where this assumption is invoked are the characterization of a specific $L^2$-sequence space (see Proposition~\ref{prop:Characterization}) and the \textit{stepsize uniform} $\mathcal{S}^q$-integrability (see Proposition~\ref{prop:FixPointScheme_Lq_Integrable}).
\footnote{
    Assumption~\hyperref[exp:assumption]{(R-q)} is due to the consideration of (unbounded) random time horizons.
    In Section~\ref{sec.ExponentialMoments} we discuss the moment condition for the Brownian motion. 
    In Section~\ref{sec:Application_EM} we apply the bounds to Markovian BSDEs. See Remark~\ref{remark:Temporal_CutOff} for the discussion of sufficiency at cost of potentially losing expressivity of the error bound.    
}
\begin{exp_assumption}\label{exp:assumption}
    Let $q \geq 2$ be a given parameter. 
    Assume that there is $\rho>0$ such that
    \begin{equation}
        \rho > 4q\ \lipschitz{f} 
        \quad\text{and}\quad \exp(\rho\tau) \in L^1 
        \quad\text{and}\quad \sup_{h\in (0,h_0]}\exp(\rho\StoppDiscrete) \in L^1 \closeEqn
    \end{equation}
\end{exp_assumption}

Furthermore we impose the discrete-time approximations of the BSDE generator and terminal time to posses the same regularity conditions; and that the respective families (collecting all admissible stepsizes) are bounded in $L^p$.\footnote{
    In \cite[Theorem 2]{bender2007forward} a similar assumption is imposed. 
}
\begin{discrete_generator_assumption}\label{discrete_f:assumption}
    %
    The discrete-time approximation $\fd{\unboundedTimeGrid} \colon \unboundedTimeGrid\times\Omega\times\R\to \R$ is Lipschitz continuous in $y$ with the same $\lipschitz{f}$ as $f$ (see
    Assumption \hyperref[bsde:assumption]{(F)}) 
    , for each $t\in\unboundedTimeGrid$ and $\F_t$-measurable $A$ it holds that $\fd{\unboundedTimeGrid}(t,A)$ is $\F_t$-measurable and
    \begin{equation}
        \E \Big[ \sup_{h\in(0,h_0]} \sup_{t\in\unboundedTimeGrid\cap [0, \StoppMax] } \big| \fd{\unboundedTimeGrid}( t,0) \big|^{p} \Big] < \infty \closeEqn 
    \end{equation}
\end{discrete_generator_assumption}

\begin{discrete_terminal_assumption}\label{discrete_xi:assumption}
    The discrete-time terminal condition $\bar{\xi}$ is $\F_{\StoppDiscrete}$-measurable and satisfies
    \begin{equation}
        \E \Big[ \sup_{h\in(0,h_0]} |\bar{\xi}|^p \Big] < \infty \closeEqn
    \end{equation}
\end{discrete_terminal_assumption}

With this we state our first main result: A uniform bound of the distances between the solution of the BSDE and our proposed scheme along the discrete-time grid. Such quantities typically are analyzed in order to establish error bounds and convergence results, see e.g. \cite{bender2008time, Bouchard_strong_BSDE_on_domain, bouchard2004discrete, zhang2004numericalScheme}. 
\begin{theorem}[Classic Bound]\label{thm:BSDE_MainResult_no_SUPREMUM}
    Assume that the BSDE \eqref{def:BSDE_general} has a solution\footnote{
		We use the notation $ \mathcal{S}^p \defined \{ Y\ |\ \E [ \sup_{ s\geq 0 } |Y_s|^p ] < \infty \} $ and $\mathcal{H}^p \defined \{ Z \ |\ \E \big[ (\int_{0}^{\infty} \|Z\|^2\de s)^\frac{p}{2} \big] < \infty \}$.
	} 
    $(Y,Z)\in \mathcal{S}^4\times\mathcal{H}^4$.
    Assume \hyperref[stepsize:assumption]{(h)}, \hyperref[exp:assumption]{(R-q)} with $q=4$,  \hyperref[discrete_f:assumption]{($\unboundedTimeGrid$-F)} with $p>8$ and \hyperref[discrete_xi:assumption]{($\unboundedTimeGrid$-T)} with $p=8$. 
    Then there is $\theGronwallV \in L^2$ (not depending on $h$; see \eqref{def:GronwallVanisher}) and a constant $C>0$ such that 
    \begin{equation}
        \sup_{t\in\unboundedTimeGrid} \E \big[ | Y_t - \Y_t |^2 \big] \leq C h
        + \E \bigg[ \exp\big( 3 \lipschitz{f} (\tau\vee\StoppDiscrete)\big) \Big( \big|\xi-\bar{\xi}\big|^2 + 2 \theGronwallV |\tau-\StoppDiscrete| + \tfrac{3}{2\lipschitz{f}}  \int_{0}^{\tau} \Delta F_s^2 \de s \Big)\! \bigg]
    \end{equation}
    where $\Delta F_s \defined f(s,Y_s) - \fd{\unboundedTimeGrid} (\phi(s),Y_s)$ for $s\geq 0$.\footnote{
        We use $\phi\colon [0,\infty)\to\unboundedTimeGrid, s\mapsto\argmax\{t\in\unboundedTimeGrid\ | \ t\leq s \}$ to denote the discrete time directly before $s\geq 0$.
    } 
    \close
\end{theorem}
\begin{proof}
    Let $t\in\unboundedTimeGrid$. From Proposition~\ref{prop:AdaptedGronwall} we obtain a pathwise bound of the distance between $Y$ and $\Y$ along $\unboundedTimeGrid$.\footnote{
        The result is an adapted bound of the difference of the process $Y$ and a continuous, adapted interpolation of $\Y$. The bound consists of a closed martingale; holds on $[0,\infty)$ and is derived using a stochastic Gronwall lemma. For ease of presentation here we apply Proposition~\ref{prop:AdaptedGronwall} with the choice $\alpha_1=\tfrac{1}{3}$ and $ \alpha_2=\tfrac{2}{3}\lipschitz{f}$.
    }
    Specifically there are $\theGronwallR, \theGronwallV \in L^2$ (not depending on $h$; see \eqref{def:theGronwallR} and \eqref{def:GronwallVanisher}); 
    \begin{align}\label{eqn:Proof_MainResult_Eq1}
        \big( Y_t - \Y_{t} \big)^2 & \leq \E_{\ell} \Big[ \exp\big( 3\lipschitz{f} (\StoppMax) \big) \big( |\xi-\bar{\xi}|^2 + 2 \theGronwallV |\tau-\StoppDiscrete| \big)  \Big] \notag \\
        & \hspace*{1.0cm} + 12\lipschitz{f} h\ \E_{\ell} \bigg[ \exp\big( 3\lipschitz{f} (\StoppMax) \big) \Big( h\theGronwallR + \int_{0}^{\tau} \| Z_s\|^2 \de s \Big) \bigg] \notag \\
        & \hspace*{2.0cm} + \tfrac{3}{2\lipschitz{f}} \E_{\ell} \Big[ \exp\big( 3\lipschitz{f} (\StoppMax) \big) \int_{0}^{\tau} \big( f(s,Y_s) - \fd{\unboundedTimeGrid} (\phi(s),Y_s) \big)^2 \de s \Big] \hspace*{0.5cm}
    \end{align}
    With the Cauchy-Schwarz inequality, Assumption~\hyperref[exp:assumption]{(R-q)} and since $\theGronwallR\in L^2$ and $Z\in\mathcal{H}^4$ we have
    \begin{align} 
        & 12\lipschitz{f} \E \bigg[ \exp\big( 8\lipschitz{f} (\StoppMax) \big) \Big( h\theGronwallR + \int_{0}^{\tau} \| Z_s\|^2 \de s \Big) \bigg] \notag \\
        & \hspace*{0.5cm} \leq 12\lipschitz{f} \sup_{h\in (0,h_0]} \E  \Big[ \exp\big( 16\lipschitz{f} (\StoppMax) \big) \Big]^\frac{1}{2}\ \sqrt{2} 
         \bigg(  h_0 \E \big[ \theGronwallR^2 \big]^\frac{1}{2} + \E \Big[ \Big( \int_{0}^{\tau} \| Z_s\|^2 \de s \Big)^2 \Big]^\frac{1}{2} \bigg) \notag \\
        & \hspace*{0.5cm} \defined C < \infty \qedhere
    \end{align}
\end{proof}
The next result is the main result of this paper. It is a stronger version of the preceding result: A uniform bound on the maximally expected deviation between the solution of the BSDE and our discrete-time approximation scheme \eqref{def:FixpointScheme}.\footnote{
    The backbone of our proof is a stochastic Gronwall lemma for random time horizons. This enables us to consider to consider the entire discrete-time grid under the expectation. In contrast the classical error bounds (for bounded time horizons) rely on a deterministic Gronwall lemma, which in turn requires performing the estimation with the expectation. 
    Theorem~\ref{thm:BSDE_MainResult_no_SUPREMUM} is a byproduct of the proof of Theorem~\ref{thm:BSDE_MainResult}.
}

\begin{theorem}[Strong Bound]\label{thm:BSDE_MainResult}
    Assume the terminal condition satisfies $\xi\in L^8$ and the BSDE \eqref{def:BSDE_general} has a solution $(Y,Z)\in \mathcal{S}^8\times\mathcal{H}^8$. 
    Assume \hyperref[stepsize:assumption]{(h)}, \hyperref[exp:assumption]{(R-q)} with $q=8$, 
    \hyperref[discrete_f:assumption]{($\unboundedTimeGrid$-F)} with $p>16$ and \hyperref[discrete_xi:assumption]{($\unboundedTimeGrid$-T)} with $p=16$. 
    Then there is $\theGronwallV \in L^4$ and a constant $C>0$ (neither depending on $h$) such that
    \begin{align}
        \E \Big[ \sup_{ t \in\unboundedTimeGrid } \big| Y_t - \Y_t \big|^2  \Big] & \leq C h + C \E \Big[ \exp\big( 6\lipschitz{f} (\StoppMax) \big) |\xi-\bar{\xi}|^4 \Big]^\frac{1}{2} \\
        & \hspace*{1.0cm} + C \E \Big[ \exp\big( 6\lipschitz{f} (\StoppMax) \big) \theGronwallV^2 |\tau-\StoppDiscrete|^2  \Big]^\frac{1}{2} \\
        & \hspace*{2.0cm} + C \E \bigg[ \exp\big( 6\lipschitz{f} (\StoppMax) \big) \Big( \int_{0}^{\tau} \Delta F_s^2 \de s \Big)^2 \bigg]^\frac{1}{2}
    \end{align}
    where $\Delta F_s \defined f(s,Y_s) - \fd{\unboundedTimeGrid} (\phi(s),Y_s)$ for $s\geq 0$ (using the same notation as in Theorem~\ref{thm:BSDE_MainResult_no_SUPREMUM}). \close
\end{theorem}
\begin{proof}
    As in the proof of Theorem~\ref{thm:BSDE_MainResult_no_SUPREMUM} we invoke Proposition~\ref{prop:AdaptedGronwall} (with $\alpha_1=\tfrac{1}{3}, \alpha_2=\tfrac{2}{3}\lipschitz{f}$) to obtain \eqref{eqn:Proof_MainResult_Eq1}; where under the stronger integrability assumptions of this result we know that $\theGronwallR, \theGronwallV \in L^4$. 
    Using the Jensen and Doob $L^2$-inequality\footnote{
        For any non-negative $A\in L^2$ it holds that
        $\E [ \sup_{t\geq 0} \E_t [A] ]^2 \leq \E [ \sup_{t\geq 0} \E_t [A]^2 ] \leq 4 \E [ A^2 ]$
    } 
    we conclude
    \begin{align}
        \E \Big[ \sup_{ t \in\unboundedTimeGrid } \big| Y_t - \Y_t \big|  \Big] & 
        \leq 2 \E \Big[ \exp\big( 6\lipschitz{f} (\StoppMax) \big) \big( |\xi-\bar{\xi}|^2 + 2 \theGronwallV |\tau-\StoppDiscrete| \big)^2  \Big]^\frac{1}{2} \\
        & \hspace*{0.5cm} + 24\lipschitz{f} h\ \E \bigg[ \exp\big( 6\lipschitz{f} (\StoppMax) \big) \Big( h\theGronwallR + \int_{0}^{\tau} \| Z_s\|^2 \de s \Big)^2 \bigg]^\frac{1}{2} \notag \\
        & \hspace*{1.0cm} + \tfrac{3}{\lipschitz{f}} \E \bigg[ \exp\big( 6\lipschitz{f} (\StoppMax) \big) \Big( \int_{0}^{\tau} \big( f(s,Y_s) - \fd{\unboundedTimeGrid} (\phi(s),Y_s) \big)^2 \de s \Big)^2 \bigg]^\frac{1}{2}
    \end{align}
    we bound the expectation of the second line in the latter estimate similarly to the proof of Theorem~\ref{thm:BSDE_MainResult_no_SUPREMUM}. With the Cauchy-Schwarz inequality and an elementary bound we have
    \begin{align}
        &\E \bigg[ \exp\big( 6\lipschitz{f} (\StoppMax) \big) \Big( h\theGronwallR + \int_{0}^{\tau} \| Z_s\|^2 \de s \Big)^2 \bigg]^2 \\
        & \hspace*{1.0cm} \leq \sup_{h\in(0,h_0]} \E \Big[ \exp\big( 12\lipschitz{f} (\StoppMax) \big) \Big]\ 2^3\ \E \bigg[ h_0^4 \theGronwallR^4 + \Big( \int_{0}^{\tau} \| Z_s\|^2 \de s \Big)^4 \bigg] \\
        & \hspace*{1.0cm} \defined C_1 < \infty 
    \end{align}
    where we use Assumption~\hyperref[exp:assumption]{(R-q)}, the fact that $\theGronwallR \in L^4$ (due to the stronger assumptions; see Lemma~\ref{lemma:FRemainder}) and $Z\in\mathcal{H}^8$. Choose a suitably increased constant $C>0$ to conclude the claim.
\end{proof}

\begin{remark}
    Our proofs admit straightforward modifications to holds under weaker versions of Assumptions \hyperref[exp:assumption]{(R-q)}, \hyperref[discrete_f:assumption]{($\unboundedTimeGrid$-F)} and \hyperref[discrete_xi:assumption]{($\unboundedTimeGrid$-T)}. Specifically we may consider the supremum w.r.t. the admissible stepsizes outside of the expectations and derive the same error bound structure; with the difference that $\theGronwallR_h$ and $\theGronwallV_h$ are then depending on the particular discrete-time grid. Our choice of presentation with the stronger assumptions is due to the fact that in order to establish (either type of) uniform existence of positive exponential moments of the discrete terminal times, the methodology we use naturally yields that the family is bounded in $L^1$, see Section~\ref{sec.ExponentialMoments} and specifically Lemma~\ref{lemma:FreidlinTypeBound}. \close    
\end{remark}

We conclude the section with a corollary that highlights the error bound in the sense of a possible rate, i.e. we uniformly bound every dependence on the stepsize of the underlying discrete-time grid which does not vanish as $h\to 0$. 

\begin{corollary}[Uniform Bounds]\label{corollary:MainResults_AfterHoelder}
    Assume the setting of Theorem ~\ref{thm:BSDE_MainResult_no_SUPREMUM} then there is a constant $K_1>0$ (not depending on $h$) such that
    \begin{equation}
        \sup_{\ell\in\unboundedTimeGrid} \E \Big[ \big| Y_{\ell} - \Y_{\ell} \big|^2 \Big] \leq K_1 \Bigg( h +  \E \big[\ |\xi-\bar{\xi}|^4 \big]^\frac{1}{2} + \E \big[\ |\tau-\StoppDiscrete|^4 \big]^\frac{1}{4} 
        + \E \bigg[ \Big(\int_{0}^{\tau} \Delta F_s^2 \de s\Big)^2 \bigg]^\frac{1}{2} \Bigg)
    \end{equation}
    Moreover if we assume the stronger setting of Theorem~\ref{thm:BSDE_MainResult} then there is a constant $K_2>0$ (not depending on $h$) such that
    \begin{equation}
        \E \Big[ \sup_{\ell\in\unboundedTimeGrid} \big| Y_\ell - \Y_{\ell} \big|^2 \Big] \leq K_2 \Bigg( h +  \E \big[\ |\xi-\bar{\xi}|^8 \big]^\frac{1}{4} + \E \big[\ |\tau-\StoppDiscrete|^8 \big]^\frac{1}{8} 
        + \E \bigg[ \Big(\int_{0}^{\tau} \Delta F_s^2 \de s\Big)^4 \Big]^\frac{1}{4} \Bigg) \closeEqn
    \end{equation}
\end{corollary}
\begin{proof}
    Using a general exponential scalar $m$ we derive both results simultaneously; $m=1$ corresponds to estimating the bound of Theorem~\ref{thm:BSDE_MainResult_no_SUPREMUM} and $m=2$ to Theorem~\ref{thm:BSDE_MainResult}. 
    
    Since $\theGronwallV \in L^{2m}$ (not depending on $h$; see \eqref{def:GronwallVanisher}) with the Cauchy-Schwarz inequality we have
    \begin{equation}
        \E \Big[ \exp \big( 3m \lipschitz{f} (\StoppMax)\big) \theGronwallV^m |\tau-\StoppDiscrete|^m \Big] \leq \E \big[ \theGronwallV^{2m} \big]^\frac{1}{2} \E \Big[ \exp\big( 12 m\lipschitz{f} (\StoppMax)\big) \Big]^\frac{1}{4} \E \big[ |\tau-\StoppDiscrete|^{4m} \big]^\frac{1}{4}
    \end{equation}
    Using Assumption~\hyperref[exp:assumption]{(R-q)} allows us to set 
    \begin{equation}
        C_1^{(m)} \defined \E \big[ \theGronwallV^{2m} \big]^\frac{1}{2} \E \Big[ \sup_{h\in (0,h_0]} \exp\big( 12 m \lipschitz{f} (\StoppMax)\big) \Big]^\frac{1}{4} < \infty
    \end{equation}
    Analogously we have
    \begin{equation}
        \E \Big[ \exp\big( 3m \lipschitz{f} (\StoppMax)\big)\ \big|\xi-\bar{\xi}\big|^{2q} \Big] \leq \E \Big[ \exp\big( 6m \lipschitz{f} (\StoppMax)\big) \Big]^\frac{1}{2} \E \Big[ \big|\xi-\bar{\xi}\big|^{4m} \Big]^\frac{1}{2}
    \end{equation}
    and set
    \begin{equation}
        C_2^{(m)} \defined \E \Big[ \sup_{h\in (0,h_0]} \exp\big( 6m \lipschitz{f} (\StoppMax)\big) \Big]^\frac{1}{2} < \infty 
    \end{equation}
    Once more we repeat the Cauchy-Schwarz argument to bound the $\Delta F$ integral
    \begin{equation}
        \E \bigg[ \exp\big( 3m \lipschitz{f} (\StoppMax) \big) \Big( \int_{0}^{\tau}\! \Delta F_s^2 \de s \Big)^m \bigg]\! \leq \E \Big[ \exp\big( 6m \lipschitz{f} (\StoppMax)\big) \Big]^\frac{1}{2} \E \bigg[ \Big( \int_{0}^{\tau}\! \Delta F_s^2 \de s \Big)^{2m} \bigg]^\frac{1}{2}
    \end{equation}
    Thus we define the overall constants $K_{m} \defined C \max \{ C_1^{(m)}, C_2^{(m)}, 1\}$ where $C>0$ denotes the constant from Theorem~\ref{thm:BSDE_MainResult_no_SUPREMUM} resp. Theorem~\ref{thm:BSDE_MainResult} and the proof is complete.
\end{proof}

\section{Strong BSDE Approximation Error Bounds}\label{sec:BSDE_StrongErrorBound}

Subsequently we apply Theorem \ref{thm:BSDE_MainResult_no_SUPREMUM} and Theorem \ref{thm:BSDE_MainResult} to derive corollaries providing bounds of 
\begin{equation}\label{def:BSDE_Errors}
    \mathcal{E}^1 (Y) \defined \sup_{\ell\in\unboundedTimeGrid} \E \Big[ \sup_{t\in[\ell,\ell+h]} \big| Y_t - \Y_{\ell} \big|^2 \Big] 
    \quad\text{and}\quad
    \mathcal{E}^2 (Y) \defined \E \Big[ \sup_{\ell\in\unboundedTimeGrid} \sup_{t\in[\ell,\ell+h]} \big| Y_t - \Y_{\ell} \big|^2 \Big]
\end{equation}
In literature concerning the (random) bounded time horizon estimates of $\mathcal{E}^1 (Y)$ are established, see e.g. \cite{bender2008time, Bouchard_strong_BSDE_on_domain, bouchard2004discrete, gobetLabart2007, zhang2004numericalScheme}.
We are in position the derive bounds for the uniform-in-time $L^2$ error $\mathcal{E}^2 (Y)$ since the core of our analysis is an adapted Gronwall bound (see Proposition~\ref{prop:AdaptedGronwall}), which enables us to bound the maximal error on the entire discrete-time grid at once -- as presented in Theorem\ref{thm:BSDE_MainResult}. Together with a Kolmogorov's tightness criterion for the BSDE regarded as an It\=o process we derive a bound for the second error $\mathcal{E}^2 (Y)$, see Theorem~\ref{thm:Kolmogorov_BSDE_Error_Bound}.
Evidently $\mathcal{E}^1 (Y) \leq \mathcal{E}^2 (Y)$. Nevertheless the estimate of $\mathcal{E}^2 (Y)$ requires stronger regularity assumptions. Thus without imposing Kolmogorov's tightness criterion we present two possibilities to extend Theorem~\ref{thm:BSDE_MainResult_no_SUPREMUM} in order to bound $\mathcal{E}^1 (Y)$. 
\begin{corollary}\label{corollary:BSDE_Rate_Classic}
    Assume the setting of Theorem~\ref{thm:BSDE_MainResult_no_SUPREMUM}. Moreover assume \hyperref[bsde:assumption]{(F)} and that the martingale integrand $Z$ satisfies the integrability condition
    \begin{equation}\label{eqn:MainResult_AssumptionMartingaleIntegrand}
        \E \Big[ \int_{0}^{\tau} \|Z_s\|^4 \de s \Big] < \infty
    \end{equation}
    Then there is a constant $C>0$ (not depending on $h$) such that
    \begin{equation}
        \mathcal{E}^1 (Y) \leq C h^\frac{1}{2}
        + C \E \Big[ \exp\big( 3 \lipschitz{f} (\tau\vee\StoppDiscrete)\big) \Big( \big|\xi-\bar{\xi}\big|^2 + \theGronwallV |\tau-\StoppDiscrete| +  \int_{0}^{\tau} \Delta F_s^2 \de s \Big)\! \bigg]
        \closeEqn
    \end{equation}
\end{corollary}
\begin{proof}
    Using an elementary bound we have
    \begin{equation}
        \mathcal{E}^1 (Y)  = \sup_{\ell\in\unboundedTimeGrid} \E \Big[ \sup_{t\in[\ell,\ell+h]} \big| Y_t - \Y_{\ell} \big|^2 \Big] \leq 2 \sup_{\ell\in\unboundedTimeGrid} \E \Big[ \sup_{t\in[\ell,\ell+h]} \big| Y_t - Y_{\ell} \big|^2 \Big] 
        + 2 \sup_{\ell\in\unboundedTimeGrid} \E \Big[ \big| Y_\ell - \Y_{\ell} \big|^2 \Big]
    \end{equation}
    Hence we invoke Theorem~\ref{thm:BSDE_MainResult_no_SUPREMUM} and the proof is complete if we show there is $C_1>0$ such that
    \begin{equation}\label{eqn:Yregularity}
        \E \Big[ \sup_{t\in[\ell,\ell+h]} \big| Y_t- Y_\ell \big|^2 \Big] \leq C_1 h^\frac{1}{2} \qquad\text{for all } \ell\in\unboundedTimeGrid . 
    \end{equation}
    For this let $\ell\in\unboundedTimeGrid$ and $t\in[\ell,\ell+h)$ be fixed. Using the BSDE \eqref{def:BSDE_general} and Cauchy-Schwarz we have
    \begin{equation}
        \big| Y_t - Y_{\ell} \big|^2 \leq 2(t-\ell) \int_{\ell}^{t} \big| f (s,Y_s) \big|^2 \de s + 2\ \Big| \int_{\ell}^{t} Z_s \de W_s \Big|^2 
    \end{equation}
    and since $f$ is Lipschitz
    together with Remark~\ref{remark:ItoJensenBound} we conclude
    \begin{align}
        \E \Big[ \sup_{t\in[\ell,\ell+h]} \big| Y_t- Y_\ell \big|^2 \Big] & \leq 4 h \E \Big[\int_{\ell\wedge\tau}^{(\ell+h)\wedge\tau} f(s,0)^2 + \lipschitz{f}^2 Y_s^2 \de s \Big] + 4 h^\frac{1}{2} \E \Big[ \int_{\ell}^{\ell+h} \| Z_s \|^4 \de s \Big]^\frac{1}{2} \\
        & \leq 4 h \E \Big[ \int_{0}^{\tau} f(s,0)^2 \de s + h \lipschitz{f}^2 \sup_{r\geq 0} Y_r^2 \Big] + 4 h^\frac{1}{2} \E \Big[ \int_{0}^{\tau} \|Z_s\|^4 \de s \Big]^\frac{1}{2}   
    \end{align}
    By assumption we have $Y\in\mathcal{S}^2$ as well as $f(\cdot,0)\in\mathcal{H}^{2}$ and $Z$ satisfies the stronger integrability condition \eqref{eqn:MainResult_AssumptionMartingaleIntegrand}, whence there is $C_1>0$ (not depending on $\ell$ nor $t$) such that \eqref{eqn:Yregularity} holds.
\end{proof}

We emphasize the rate provided by Remark~\ref{remark:ItoJensenBound} is not necessary optimal. Nonetheless due to the presence of the $L^1$-distance of the terminal times in Theorem~\ref{thm:BSDE_MainResult_no_SUPREMUM} in order to effectively improve the overall bound this distance needs to be taken into account.
\begin{remark}\label{remark:ItoJensenBound}
    Let $H\in\mathcal{H}^p$ for some $p\geq 4$. Combine Doob's $L^2$-inequality with It\=o's isometry followed by the Cauchy-Schwarz and Jensen inequality to conclude that for any $\ell\geq 0$;
    \begin{equation}
        \E \bigg[ \sup_{t\in[\ell,\ell+h]} \Big|\int_{\ell}^{t} H_s \de W_s \Big|^2 \bigg] \leq 4 \E \Big[ \int_{\ell}^{\ell+h} \big\|H_s\big\|^2 \de s \Big] \leq 4 h^\frac{1}{2}\ \E \Big[ \int_{\ell}^{\ell+h} \big\|H_s\big\|^4 \de s \Big]^\frac{1}{2}. \closeEqn
    \end{equation}
\end{remark}

Next is a result describing to what extend the error bound from Corollary~\ref{corollary:BSDE_Rate_Classic} improves if the estimate of martingale integrand of the solution of the BSDE \eqref{def:BSDE_general} as described in Remark~\ref{remark:ItoJensenBound} can be improved. This is e.g. the case when $Z$ is bounded, see \cite{cheridito2014bsdes} for a general study on such BSDEs, or consider the Markovian case where the time horizon is the first exit of a diffusion from a sufficient regular domain. The latter is further discussed in Section~\ref{sec:Application_EM}.
\begin{corollary}\label{corollary:BoundedMartingaleIntegrand}
    Assume the setting of Theorem~\ref{thm:BSDE_MainResult_no_SUPREMUM}. If there is a constant $C_Z>0$ such that
    \begin{equation}
        \E \Big[ \int_{\ell}^{\ell+h} \big\|H_s\big\|^4 \de s \Big] \leq C_Z h \quad\text{for all } \ell\in\unboundedTimeGrid
    \end{equation}
    then there is a constant $C>0$ (not depending on $h$) such that
    \begin{equation}
        \mathcal{E}^1 (Y) \leq C h
        + C \E \Big[ \exp\big( 3 \lipschitz{f} (\tau\vee\StoppDiscrete)\big) \Big( \big|\xi-\bar{\xi}\big|^2 + \theGronwallV |\tau-\StoppDiscrete| +  \int_{0}^{\tau} \Delta F_s^2 \de s \Big)\! \bigg]
        \closeEqn
    \end{equation}
\end{corollary}
The next result is an application of our second main result Theorem~\ref{thm:BSDE_MainResult}. We derive a bound for $\mathcal{E}^2 (Y)$ (see \eqref{def:BSDE_Errors}) which takes into account the maximal deviation along the entire time grid. 
\begin{theorem}[Uniform-in-Time Bound]\label{thm:Kolmogorov_BSDE_Error_Bound}
    Assume the setting of Theorem~\ref{thm:BSDE_MainResult}. If furthermore the process $Y$ satisfies Kolmogorov's tightness criterion with a constant rate, i.e. there is $\alpha > 0$ such that for all $p\geq 2$ it holds that
    \begin{equation}\label{eqn:KolmogorovCriterion_BSDE_Corollary}
        \E \big[ |Y_s - Y_t|^p ] \leq C_p |s-t|^{\alpha p} \quad \text{for all } |s-t| \leq 1
    \end{equation}
    Then for every $\varepsilon > 0$ there is a constant $C>0$ (not depending on $h$) such that
    \begin{align}
        \mathcal{E}^2 (Y) & \leq C h^{2(\alpha - \varepsilon)} + C h + C \E \Big[ \exp\big( 6\lipschitz{f} (\StoppMax) \big) |\xi-\bar{\xi}|^4 \Big]^\frac{1}{2} \\
        & \hspace*{1.0cm} + C \E \Big[ \exp\big( 6\lipschitz{f} (\StoppMax) \big) \theGronwallV^2 |\tau-\StoppDiscrete|^2  \Big]^\frac{1}{2} \\
        & \hspace*{2.0cm} + C \E \bigg[ \exp\big( 6\lipschitz{f} (\StoppMax) \big) \Big( \int_{0}^{\tau} \Delta F_s^2 \de s \Big)^2 \bigg]^\frac{1}{2} . \closeEqn
    \end{align}
\end{theorem}
\begin{proof}
    With an elementary bound we have 
    \begin{equation}
        \mathcal{E}^2 (Y) 
        = \E \Big[ \sup_{\ell\in\unboundedTimeGrid} \sup_{t\in[\ell,\ell+h]} \big| Y_t - \Y_{\ell} \big|^2 \Big] 
        \leq 2 \E \Big[ \sup_{\ell\in\unboundedTimeGrid} \sup_{t\in[\ell,\ell+h]} \big| Y_t - Y_{\ell} \big|^2 \Big]
        + 2 \E \Big[ \sup_{\ell\in\unboundedTimeGrid} \big| Y_\ell - \Y_{\ell} \big|^2 \Big]
    \end{equation}
    where the second summand can be bounded using Theorem~\ref{thm:BSDE_MainResult}. Regarding the first notice
    \begin{equation}
        \E \Big[ \sup_{\ell\in\unboundedTimeGrid} \sup_{t\in[\ell,\ell+h]} \big| Y_t - Y_{\ell} \big|^2 \Big] 
        \leq \E \Big[ \sup_{ \ell\in\unboundedTimeGrid } \sup_{ \substack{ s,t\in[\ell,\ell+h] \\ s<t\leq\tau+h } } \big| Y_t - Y_{s} \big|^2 \Big]
    \end{equation}
    Since $Y$ satisfies Kolmogorov's criterion \eqref{eqn:KolmogorovCriterion_BSDE_Corollary} with constant ratio and $\tau +h$ has a positive exponential moment, we can invoke \cite[Corollary 1.3]{Seifried2024HoelderMoments} (a result that exploits the existence of the moments of random Hölder coefficients) and obtain $C>0$ (not depending on $h$) such that
    \[
        \E \Big[ \sup_{ \ell\in\unboundedTimeGrid } \sup_{ \substack{ s,t\in[\ell,\ell+h] \\ s<t\leq\tau+h } } \big| Y_t - Y_{\ell} \big|^2 \Big] \leq C h^{2(\alpha - \varepsilon)}. \qedhere
    \]
\end{proof}

\begin{remark}
    Condition \eqref{eqn:KolmogorovCriterion_BSDE_Corollary}, i.e. to enforce Kolmororov's criterion for all $p\geq 2$, is stronger than necessary but offers improved insight. The result \cite[Corollary 1.3]{Seifried2024HoelderMoments} merely requires that for a given $\varepsilon >0$ we can fix $\delta > 1$ such that there is a pair $(\alpha,\beta)$ with $\alpha > \tfrac{\delta}{\varepsilon} \vee 2$  and $\beta >0$ satisfying 
    \begin{equation}
        \E \big[ |Y_s-Y_t|^\alpha \big] \leq C_\alpha |s-t|^{1+\beta} \quad \text{for all } |s-t| \leq 1. \closeEqn
    \end{equation}
\end{remark}

In Section~\ref{sec:Application_EM} we apply this result to Markovian BSDEs on bounded domains which is a case where (under sufficient regularity assumtions) the condition \eqref{eqn:KolmogorovCriterion_BSDE_Corollary} is satisfied. For a general discussion of Kolmogorov's tightness criterion for It\=o Processes we refer to \cite[Section 2]{Seifried2024HoelderMoments}. 

\section{A Conditional Bound of the Difference Process}\label{sec:proof_AdaptedGronwall}

Here we present the fundamental component for our main results Theorem~\ref{thm:BSDE_MainResult_no_SUPREMUM} and~\ref{thm:BSDE_MainResult}: A closed martingale (stopped at $\StoppMax$) that serves as a majorant for the difference of $Y$ and a continuous, adapted interpolation of $\Y$, which is typically used to analyze strong errors, see e.g. \cite{Bouchard_strong_BSDE_on_domain, bouchard2004discrete}. Our proof is based on a stochastic Gronwall lemma for random times, see \cite{korea_gronwall} or \cite[Appendix A]{Schlegel2024ProbabilisticShape}.\footnote{
    The bound is originally established in \cite{korea_gronwall} under minimalistic assumptions. Under the standing assumptions of this article we are allowed to use the version from \cite[Appendix A]{Schlegel2024ProbabilisticShape} which requires stronger integrability.
}
By the martingale representation theorem for each $t\in\unboundedTimeGrid$ there is a unique $\Zc^t\in\mathcal{H}^2$ such that
\begin{equation}\label{def:MartingaleRepresentation_ForEach_t}
    \Y_{t+h} - \E_t\big[\Y_{t+h}\big] = \int_{t}^{t+h} \Zc_{s}^{t} \de W_s.
\end{equation}
For convenience we introduce the notation $\phi\colon [0,\infty)\to\unboundedTimeGrid, s\mapsto\argmax\{t\in\unboundedTimeGrid\ | \ t\leq s \}$ and with this the discrete-time backward approximation is continuously interpolated by\footnote{
    Notice by It\=o's formula for each $t\in\unboundedTimeGrid$ the martingale integrand responsible for the adapted interpolation between consecutive discrete-times of the scheme $\Y$ satisfies
    $ \E_{t} [ (W_{t+h} - W_{t}) \Y_{t+h} ] = \E_{t} [ \int_{t}^{t+h} \Zc_s^{t} \de s ] $.
    When the generator $f$ depends on the martingale integrand, it is suitable to extend the scheme by $h^{-1}$ times said term, see e.g. \cite{Bouchard_strong_BSDE_on_domain, bouchard2004discrete} for the case of (random) bounded time horizon. 
    Regarding the identity: Consider the martingale $\de \mathbb{M} \defined \Zc^t \de W$. Then $\E_t [ W_{t+h} \mathbb{M}_{t+h} - W_{t} \mathbb{M}_{t} ] =  \E_t [\int_{t}^{t+h} \de \langle W, \mathbb{M}\rangle_s ] $ and $\Y_{t+h} - \E_t[\Y_{t+h}] = \mathbb{M}_{t+h}-\mathbb{M}_{t}$. 
}
\begin{equation}\label{def:YcAdapted}
    \Yc_t \defined \bar{\xi} + \int_{t\wedge\StoppDiscrete}^{\StoppDiscrete} \fd{\unboundedTimeGrid} \big( \phi(s), \Y_{\phi(s)} \big) \de s - \int_{t\wedge\StoppDiscrete}^{\StoppDiscrete} \Zc_{s}^{\phi(s)} \de W_s \qquad t\geq 0
\end{equation} 
\begin{proposition}[Adapted Interpolation]\label{prop:InterpolationRegularity}
    Assume \hyperref[stepsize:assumption]{(h)}, \hyperref[exp:assumption]{(R-q)} with $q\defined 4$, \hyperref[discrete_f:assumption]{($\unboundedTimeGrid$-F)} with $p>2q$ and \hyperref[discrete_xi:assumption]{($\unboundedTimeGrid$-T)} with $p=2q$.
    Then it holds that $\Yc\in\mathcal{S}^4$. Particularly the family of interpolations (in the sense of collecting all admissible stepsizes) is uniformly in $\mathcal{S}^4$, i.e.
    \begin{equation}
        \sup_{ h\in (0,h_0] } \sup_{t\geq 0} \Yc_{t}^4 \in L^1
    \end{equation}
    Moreover (for each $h\in (0,h_0]$) the concatenation of martingale integrands $\Zc^\phi$ is in $\mathcal{H}^{2}$. \close
\end{proposition}
This is a direct consequence of Proposition~\ref{prop:InterpolationRegularity_q_Version} (with $q=4$); the proof can be found in Section~\ref{Sec:RegularityInterpolation}. Beforehand in Section~\ref{sec:Existence} we show existence of a solution to the discrete-time approximation scheme \eqref{def:FixpointScheme} and afterwards in Section~\ref{sec:Majorant_DiscreteTime} we derive a closed martingale that majorizes $\Y$. This majorant is then extended to the interpolation $\Yc$ and with this we derive Proposition~\ref{prop:InterpolationRegularity}.
\begin{proposition}[Adapted Gronwall]\label{prop:AdaptedGronwall}
    Assume that the terminal condition of the BSDE \eqref{def:BSDE_general} satisfies $\xi\in L^4$ and there is a solution $Y\in\mathcal{S}^4$.
    Assume \hyperref[stepsize:assumption]{(h)}, \hyperref[exp:assumption]{(R-q)} with $q\defined 4$, \hyperref[discrete_f:assumption]{($\unboundedTimeGrid$-F)} with $p>2q$ and \hyperref[discrete_xi:assumption]{($\unboundedTimeGrid$-T)} with $p=2q$.
    Then there are random variables $\theGronwallV, \theGronwallR \in L^2$ (not depending on $h$; see \eqref{def:GronwallVanisher} and \eqref{def:theGronwallR}) such that for any choice\footnote{
    	Provided that $\exp \big( ((2+\alpha_1)\lipschitz{f}+\alpha_2) (\StoppMax) \big) \in L^2$. 
    } 
    $\alpha_1\geq\tfrac{1}{3}$, $\alpha_2>0$ it holds that
    \begin{align}
		\big(Y_t - \Yc_t \big)^2 & \leq \E_t \Big[ \exp\big( ((2+\alpha_1)\lipschitz{f}+\alpha_2) (\StoppMax) \big) \big( |\xi-\bar{\xi}|^2 + 2 \theGronwallV |\tau-\StoppDiscrete| \big) \Big] \\
		& \hspace*{1.0cm} + \tfrac{4\lipschitz{f}}{\alpha_1} h\ \E_t \bigg[ \exp\big( ((2+\alpha_1)\lipschitz{f}+\alpha_2) (\StoppMax)\big)  \Big( h\theGronwallR + \int_{0}^{\tau} \| Z_s \|^2 \de s \Big) \bigg] \\
		& \hspace*{2.0cm} + \tfrac{1}{\alpha_2} \E_t \Big[ \exp\big( ((2+\alpha_1)\lipschitz{f}+\alpha_2) (\StoppMax) \big) \int_{0}^{\tau} \Delta F_s^2 \de s \Big] 
    \end{align}
    for all $t\geq 0$ where $\Delta F_s \defined f(s, Y_s) - \fd{\unboundedTimeGrid} (\phi(s),Y_s)$ for $s\geq 0$.

    If additionally the terminal condition of the BSDE \eqref{def:BSDE_general} satisfies $\xi\in L^8$ and the solution satisfies $Y\in\mathcal{S}^8$, and moreover the stronger integrability \hyperref[exp:assumption]{(R-q)} with $q\defined 8$, \hyperref[discrete_f:assumption]{($\unboundedTimeGrid$-F)} with $p>2q$ and \hyperref[discrete_xi:assumption]{($\unboundedTimeGrid$-T)} with $p=2q$ holds then $\theGronwallR, \theGronwallV \in L^4$. \close
\end{proposition}
\begin{proof} 
    On $[0,\StoppMax]$ we define difference process $\Delta \defined Y-\Yc$;
    \begin{equation}
        \Delta_t = \xi-\bar{\xi} + \int_{t\wedge\zeta}^{\zeta} \1_{\{s\leq\tau\}} f(s,Y_s) - \1_{\{s\leq\StoppDiscrete\}} \fd{\unboundedTimeGrid} (\phi(s),\Y_{\phi(s)}) \de s - \int_{t\wedge\zeta}^{\zeta} Z_s - \Zc_s^{\phi(s)} \de W_s 
    \end{equation}
    where we use the notation $\zeta\defined\StoppMax$ and the fact that 
    $\1_{\{\cdot\geq\tau\}} Z + \1_{\{\cdot\geq\StoppDiscrete\}} Z^{\phi} = 0$.
    By It\=o's formula\footnote{For a process $\Delta$ satisfying $\de \Delta_t = F_t\de t + \de \M_t$ it holds that $\de \Delta_t^2 = 2 \Delta_tF_t\de t + 2 \Delta_t \de \M_t + \de\langle\M\rangle_t$.} on $\{t\leq\zeta\}$ one checks that
	\begin{multline}
		\Delta_\zeta^2 - \Delta_t^2 = 2 \int_{t}^{\zeta} \Delta_s \big(\1_{\{s<\tau\}} f(s,Y_s) - \1_{\{s<\StoppDiscrete\}} \fd{\unboundedTimeGrid} (\phi(s),\Y_{\phi(s)}) \big) \de s \\
		+ \int_{t}^{\zeta} \big\| Z_s - \Zc_s^t \big\|^2 \de s
		-2\int_{t}^{\zeta} \Delta_s \big(Z_s - \Zc_s^t \big) \de W_s
	\end{multline}
	where due to Proposition~\ref{prop:InterpolationRegularity} we know that $\Delta(Z-\Zc^\phi)\de W$ is an $L^2$-martingale. Thus\footnote{
        We give a brief motivation for the split regarding the arguments of $\fd{\unboundedTimeGrid}$: The first integral provides the difference process on $[0,\StoppMax]$. The second is a non-constant prolongation of $Y,\Yc$ onto $[\tau\wedge\StoppDiscrete, \StoppMax]$. The third compensates the replacement of $\Y$ by $\Yc$ on $[0,\StoppDiscrete]$. The last is the deviation between discrete-times regarding the intrinsic time dependency of $f$ resp. $\fd{\unboundedTimeGrid}$ along trajectories of $Y$ on $[0,\tau]$. {In the same order Lemma~\ref{lemma:AdaptedGronwallInt_1} - \ref{lemma:AdaptedGronwallInt_4} provide bounds.}
    }
    \begin{align}\label{eqn:ProofBackwardGronwall1}
		& \Delta_t^2 + \E_t \Big[ \int_{t\wedge\zeta}^{\zeta} \big\| Z_r - \Zc_r^{\phi(r)} \big\|^2 \de r \Big] \notag \\ 
        & \hspace*{1.0cm} = \E_t \big[ (\xi-\bar{\xi} )^2 \big] \notag \\
        & \hspace*{2.0cm} - \E_t \bigg[ 2 \int_{t\wedge\zeta}^{\zeta} \Delta_s \big( \fd{\unboundedTimeGrid} (\phi(s),Y_s) - \fd{\unboundedTimeGrid} (\phi(s),\Yc_s) \big) \de s  \notag \\
        & \hspace*{3.0cm} - 2 \int_{t\vee(\tau\wedge\StoppDiscrete)}^{\StoppMax} \Delta_s \big( \1_{\{s>\tau\}} \fd{\unboundedTimeGrid} (\phi(s),\xi) - \1_{\{s>\StoppDiscrete \}} \fd{\unboundedTimeGrid} (\phi(s),\bar{\xi}) \big) \de s \notag \\
        & \hspace*{4.0cm} +2 \int_{t\wedge\StoppDiscrete}^{\StoppDiscrete} \Delta_s \big( \fd{\unboundedTimeGrid} (\phi(s), \Yc_s) - \fd{\unboundedTimeGrid} (\phi(s), \Y_{\phi(s)}) \big) \de s \notag \\
        & \hspace*{5.0cm} + 2 \int_{t\wedge\tau}^{\tau} \Delta_s \big( f(s, Y_s) - \fd{\unboundedTimeGrid} (\phi(s),Y_s) \big) \de s \bigg].
    \end{align} 
    From Lemma~\ref{lemma:AdaptedGronwallInt_1}, Lemma~\ref{lemma:FVanisher}, Lemma~\ref{lemma:FRemainder} (applied with $\alpha_1\geq \tfrac{1}{3}$) and Lemma~\ref{lemma:AdaptedGronwallInt_4} (with $\alpha_2>0$) we know that there are $\theGronwallV,\theGronwallR\in L^2$ (see \eqref{def:GronwallVanisher} and \eqref{def:theGronwallR}) satisfying   
    \begin{align}\label{eqn:AdaptedGronwallRawBound}
		& \Delta_{t}^{2} + \big( 1 - \tfrac{4\lipschitz{f}}{\alpha_1} h \big) \E_t \Big[ \int_{t\wedge\zeta}^{\zeta} \big\| Z_r - \Zc_r^{\phi(r)} \big\|^2 \de r \Big] \leq \big( 2\lipschitz{f} + \alpha_1 \lipschitz{f} + \alpha_2 \big) \E_{t} \Big[ \int_{t\wedge\zeta}^{\zeta} \Delta_s^2 \de s \Big] \notag \\ 
        & \hspace*{3.0cm} + \E_t \big[ |\xi-\bar{\xi}|^2 \big] + 2 \E_t \Big[ \theGronwallV \big|\tau-\StoppDiscrete\big| \Big] + \tfrac{4\lipschitz{f}}{\alpha_1} h\ \E_t \Big[ h\theGronwallR + \int_{0}^{\tau} \| Z_s \|^2 \de s \Big] \notag \\
        & \hspace*{4.0cm} + \tfrac{1}{\alpha_2} \E_t \Big[ \int_{t\wedge\tau}^{\tau} \big( f(s, Y_s) - \fd{\unboundedTimeGrid} (\phi(s),Y_s) \big)^2 \de s \Big] .
    \end{align}
    where due to Assumption~\hyperref[stepsize:assumption]{(h)} we have $1-\tfrac{4\lipschitz{f}}{\alpha_1}h \geq 0$.
    We set (note there is no dependency on $t$)
    \begin{multline}\label{def:GronwallMasterUpperBound}
    	\theGronwallMaster (\alpha_1, \alpha_2) \defined |\xi-\bar{\xi}|^2 + 2 \theGronwallV \big|\tau-\StoppDiscrete \big| + \tfrac{4\lipschitz{f}^2}{\alpha_1} h\ \Big( h\theGronwallR + \int_{0}^{\tau} \| Z_s \|^2 \de s \Big) \\
    	+ \tfrac{1}{\alpha_2} \int_{0}^{\tau} \big( f(s, Y_s) - \fd{\unboundedTimeGrid} (\phi(s),Y_s) \big)^2 \de s
    \end{multline}
    and we have $\theGronwallMaster (\alpha_1, \alpha_2) \in L^2$. Since $t\geq 0$ was fixed in the derivation of \eqref{eqn:AdaptedGronwallRawBound} we established
	\begin{equation}
		\Delta_{t}^{2} \leq \E_t \big[ \theGronwallMaster (\alpha_1,\alpha_2) \big] + ( 2\lipschitz{f} +\alpha_1\lipschitz{f} + \alpha_2) \E_{t} \Big[ \int_{t\wedge\zeta}^{\zeta} \Delta_s^2 \de s \Big] \quad\text{for all } t\geq 0 
	\end{equation}    
    Whence with a stochastic Gronwall lemma, see e.g. \cite[Appendix A]{Schlegel2024ProbabilisticShape}, we conclude
    \[
        \Delta_t^2 \leq \E_t \Big[ \exp\big( ( 2\lipschitz{f} +\alpha_1 \lipschitz{f} + \alpha_2)\ \zeta \big)\ \theGronwallMaster (\alpha_1,\alpha_2) \Big]\quad\text{for all }t\geq 0. \qedhere
    \]
\end{proof}
\begin{lemma}\label{lemma:AdaptedGronwallInt_1}
 	Assume \hyperref[discrete_f:assumption]{($\unboundedTimeGrid$-F)}. With the Lipschitz continuity of $\fd{\unboundedTimeGrid}$ it holds that
 	\begin{equation}
 		\Big| \int_{t\wedge\zeta}^{\zeta}  \Delta_s \big( \fd{\unboundedTimeGrid} (\phi(s),Y_s) - \fd{\unboundedTimeGrid} (\phi(s),\Yc_s) \big) \de s \Big| \leq \lipschitz{f}  \int_{t\wedge\zeta}^{\zeta} \Delta_s^2 \de s \quad\text{for all }t\geq 0 . \closeEqn
 	\end{equation}
\end{lemma}
\begin{lemma}\label{lemma:FVanisher}
    Assume that the terminal condition of the BSDE \eqref{def:BSDE_general} satisfies $\xi\in L^4$ and the solution satisfies $Y\in\mathcal{S}^4$.
    Assume \hyperref[stepsize:assumption]{(h)}, \hyperref[exp:assumption]{(R-q)} with $q\defined 4$, \hyperref[discrete_f:assumption]{($\unboundedTimeGrid$-F)} with $p>2q$ and \hyperref[discrete_xi:assumption]{($\unboundedTimeGrid$-T)} with $p=2q$.
    Then there is a random variable $\theGronwallV \in L^2 $ (not depending on $h$; see \eqref{def:GronwallVanisher}) such that
    \begin{equation}
        2 \Big|\int_{t\vee(\tau\wedge\StoppDiscrete)}^{\StoppMax} \Delta_s \big( \1_{\{s>\tau\}} \fd{\unboundedTimeGrid} (\phi(s),\xi) - \1_{\{s>\StoppDiscrete\}} \fd{\unboundedTimeGrid} (\phi(s),\bar{\xi}) \big) \de s\Big| \leq 2 \theGronwallV \big|\tau-\StoppDiscrete \big| \quad\text{for all }t\geq 0 
    \end{equation}
    If additionally the terminal condition of the BSDE \eqref{def:BSDE_general} satisfies $\xi\in L^8$ and there solution $Y\in\mathcal{S}^8$, and moreover Assumption~\hyperref[exp:assumption]{(R-q)} holds with $q=8$ as well as
    Assumptions \hyperref[discrete_f:assumption]{($\unboundedTimeGrid$-F)} with $p>2q$ and \hyperref[discrete_xi:assumption]{($\unboundedTimeGrid$-T)} with $p=2q$ then $\theGronwallV \in L^4$. \close
\end{lemma}
\begin{proof}
    With the Lipschitz continuity of $f$ for stopping times $\Stopp_1\leq \Stopp_2$ and $A\in\F_{\Stopp_2}$ we have
    \begin{equation}
        \int_{\Stopp_1}^{\Stopp_2} \big|\fd{\unboundedTimeGrid} (\phi(s),A)\big| \de s \leq \Big(\sup_{\ell\in\unboundedTimeGrid\cap[0,\Stopp_2]} \big|\fd{\unboundedTimeGrid} (\phi(\ell),0)\big| + \lipschitz{f} |A| \Big) \big(\Stopp_2-\Stopp_1\big).
    \end{equation}
    Using an elementary bound\footnote{$ab+cd\leq(a+c)(b+d)$ for non-negative scalars.} we compute
    \begin{align}\label{def:GronwallVanisher_h_dependent}
        & \Big|\int_{t\vee(\tau\wedge\StoppDiscrete)}^{\StoppMax} \Delta_s \big( \1_{\{s>\tau\}} \fd{\unboundedTimeGrid} (\phi(s),\xi) - \1_{\{s>\StoppDiscrete \}} \fd{\unboundedTimeGrid} (\phi(s),\bar{\xi}) \big) \de s \Big| \notag \\
        & \hspace*{1.0cm} \leq \Big(\big|\xi\big| + \sup_{r\in[\tau\wedge\StoppDiscrete,\StoppMax ]}\big|\Yc_{r}\big|\Big) \int_{\tau\wedge\StoppDiscrete}^{\StoppMax}  \big|\fd{\unboundedTimeGrid} (\phi(s),\xi)\big| \de s\notag \\
        & \hspace*{5.5cm} + \Big(\sup_{r\in[\tau\wedge\StoppDiscrete,\StoppMax ]}\big|Y_r\big|+\big|\bar{\xi}\big|\Big) \int_{\tau\wedge\StoppDiscrete}^{\StoppMax}\big| \fd{\unboundedTimeGrid} (\phi(s),\bar{\xi}) \big| \de s\notag \\
        & \hspace*{1.0cm} \leq 2(1+\lipschitz{f}) \Big(\big|\xi\big| + \big|\bar{\xi}\big| + \sup_{r\in[\tau\wedge\StoppDiscrete,\StoppMax ]}\big|\Yc_{r}\big| + \sup_{r\in[\tau\wedge\StoppDiscrete,\StoppMax ]}\big|Y_{r}\big|\Big) \dots\notag \\
        & \hspace*{5.5cm} \dots \Big(\big|\xi\big| + \big|\bar{\xi}\big| + \sup_{r\in[0,\StoppMax ]} \big|\fd{\unboundedTimeGrid} (\phi(r),0)\big| \Big) \ \big|\tau-\StoppDiscrete \big| \notag \\
        & \hspace*{1.0cm} \defined \theGronwallV (h) |\tau-\StoppDiscrete |
    \end{align}
    At this point we highlight all dependencies on the stepsize $h$ in order to show that $(\theGronwallV (h))_{h\in(0,h_0]}$ is bounded in $L^2$. For this note that by definition
    \begin{multline}\label{eqn:bound_GronwallVanisher_by_SUP}
        \theGronwallV (h) \leq \sup_{\Delta\in (0,h_0]} \theGronwallV (\Delta) 
        = \sup_{\Delta\in (0,h_0]} 2(1+\lipschitz{f}) \Big( |\xi | + \big|\bar{\xi}(\Delta)\big| + \sup_{r\in[\tau\wedge\StoppDiscrete(\Delta),\StoppMax(\Delta) ]}\big|\Yc_{r}(\Delta)\big| \\
        + \sup_{r\in[\tau\wedge\StoppDiscrete(\Delta),\StoppMax(\Delta) ]} |Y_{r}|\Big)
        \Big( |\xi | + \big|\bar{\xi}(\Delta)\big| + \sup_{r\in[0,\StoppMax(\Delta) ]} \big| \fd{\Delta} (\phi(r),0)\big| \Big)
    \end{multline}
    By Assumption~\hyperref[discrete_f:assumption]{($\unboundedTimeGrid$-F)}, Assumption~\hyperref[discrete_xi:assumption]{($\unboundedTimeGrid$-T)} and Proposition~\ref{prop:InterpolationRegularity} we know that
    \begin{equation}
        \sup_{\Delta\in (0,h_0]} \sup_{r\in[0,\StoppMax(\Delta) ]} \big| \fd{\Delta} (\phi(r),0)\big| \ + \ 
        \sup_{\Delta\in (0,h_0]} |\bar{\xi}| \ + \ 
        \sup_{\Delta\in (0,h_0]} \sup_{r\geq 0} \big|\Yc_r(\Delta)\big| \in L^4
    \end{equation}
    Moreover by assumption we have $\xi\in L^4$ and $Y\in\mathcal{S}^4$; neither depends on $\unboundedTimeGrid$. Thus combine those integrability properties with \eqref{eqn:bound_GronwallVanisher_by_SUP} and use the Cauchy-Schwarz inequality to conclude
    \begin{equation}\label{def:GronwallVanisher}
        \theGronwallV \defined \sup_{h\in(0,h_0]} \theGronwallV (h) \in L^2
    \end{equation}
    The claim now follows from \eqref{def:GronwallVanisher_h_dependent} since the same bound holds for $\theGronwallV$ (see \eqref{eqn:bound_GronwallVanisher_by_SUP}) .

    The additional part of the statement is obtained in the same way. To be precise from Assumption \hyperref[discrete_f:assumption]{($\unboundedTimeGrid$-F)} and \hyperref[discrete_xi:assumption]{($\unboundedTimeGrid$-T)} with together with Proposition~\ref{prop:InterpolationRegularity_q_Version} (now applied with $q=8$) we obtain 
    \begin{equation}
        \sup_{\Delta\in (0,h_0]} \sup_{r\in[0,\StoppMax(\Delta) ]} \big|f_\Delta (\phi(r),0)\big| \ + \ 
        \sup_{\Delta\in (0,h_0]} |\bar{\xi}| \ + \ 
        \sup_{\Delta\in (0,h_0]} \sup_{r\geq 0} \big|\Yc_r(\Delta)\big| \in L^8
    \end{equation}
    Combine this with the additional assumption on the BSDE \eqref{def:BSDE_general} which ensures the terminal condition satisfies $\xi\in L^8$ and the solution $Y$ is in $\mathcal{S}^8$ to concule $\theGronwallV \in L^4$. 
\end{proof}
\begin{lemma}\label{lemma:FRemainder}
    Assume \hyperref[stepsize:assumption]{(h)}, \hyperref[exp:assumption]{(R-q)} with $q=4$ and \hyperref[discrete_f:assumption]{($\unboundedTimeGrid$-F)} with $p>2q$. 
	Then there is $\theGronwallR \in L^2$ (not depending on $h$; see \eqref{def:theGronwallR}) such that for any $\alpha >0$ it holds that
    \begin{multline}
        \bigg| \E_{t} \Big[ \int_{t\wedge\StoppDiscrete}^{\StoppDiscrete} \Delta_s \big( \fd{\unboundedTimeGrid} (\phi(s), \Yc_s) - \fd{\unboundedTimeGrid} (\phi(s), \Y_{\phi(s)}) \big) \de s \Big] \bigg| 
        \leq \tfrac{\lipschitz{f}\alpha}{2} \E_t \Big[ \int_{t\wedge\zeta}^{\zeta} \Delta_s^2 \de s \Big] \\ 
        + \tfrac{2\lipschitz{f}}{\alpha} h\ \E_t \Big[ \int_{t\wedge\StoppDiscrete}^{\StoppDiscrete} \big\|Z_r - \Zc_{r}^{\phi(r)} \big\|^2 \de r \Big] 
        + \tfrac{2\lipschitz{f}}{\alpha} h \ 
        \E_t \Big[ h\theGronwallR + \int_{0}^{\tau} \big\| Z_{r} \big\|^2 \de r \Big] \quad\text{ for all } t\geq 0. 
    \end{multline}
    If additionally \hyperref[exp:assumption]{(R-q)} holds for $q=8$ and \hyperref[discrete_f:assumption]{($\unboundedTimeGrid$-F)} holds for $p>2q$ then $\theGronwallR \in L^4$. \close
\end{lemma}
\begin{proof}
	Let $t\geq 0$ be fixed. We derive a bound of the integrated difference of $\Y$ and $\Yc$. 
	Observe
	\begin{equation}
		\Yc_{s} - \Y_{\phi(s)} = \big(\phi(s)\wedge\StoppDiscrete-s\wedge\StoppDiscrete\big) \fd{\unboundedTimeGrid} \big(\phi(s), \Y_{\phi(s)} \big) + \int_{\phi(s)}^{s} \Zc_{r}^{\phi(s)}  \de W_r \quad\text{for all } s\geq 0.
	\end{equation}
	With a conditional version of Fubini's theorem (see e.g. \cite[Proposition 2.4.6]{Ethier2009MarkovProcesses}) we compute
	\begin{align}
		\E_t \Big[ \int_{t\wedge\StoppDiscrete}^{\StoppDiscrete} \big| \Yc_s - \Y_{\phi(s)} \big|^2 \de s \Big] 
		& = \sum_{\ell\in\unboundedTimeGrid} \1_{ \{t\leq\ell\} } \E_t \Big[ \1_{ \{\ell < \StoppDiscrete \} } \int_{\ell}^{\ell+h} \E_{\ell} \big[ | \Yc_s - \Y_\ell |^2 \big] \de s \Big] \\
		& \leq 2 \sum_{\ell\in\unboundedTimeGrid} \1_{ \{t\leq\ell\} } \E_t \Big[ \1_{ \{\ell < \StoppDiscrete \} } \int_{\ell}^{\ell+h} (s-\ell)^2\ \E_\ell \big[ \fd{\unboundedTimeGrid} ( \ell, \Y_{\ell})^2 \big] \de s \\
		& \hspace*{3.0cm} + \1_{ \{\ell < \StoppDiscrete \} } \int_{\ell}^{\ell+h} \E_{\ell} \Big[ \big|\int_{\ell}^{s} \Zc_{r}^{\ell} \de W_r \big|^2 \Big] \de s \Big] \\
		& \leq 2 \E_t \Big[ \sum_{\ell\in\unboundedTimeGrid\cap[t,\StoppDiscrete)}\!\! h^3 \fd{\unboundedTimeGrid} \big( \phi(\ell),\Y_{\ell} \big)^2 + \int_{\ell}^{\ell+h} \E_{\ell} \Big[ \big| \int_{\ell}^{s} \Zc_{r}^{\ell} \de W_r \big|^2 \Big] \de s \Big].
	\end{align}
	Regarding the second part of the preceding bound, using conditional versions of It\=o's isometry and Fubini's theorem on $\{\ell<\StoppDiscrete\}$ we have
    \begin{equation}
		\int_{\ell}^{\ell+h} \E_{\ell} \Big[ \big|\int_{\ell}^{s} \Zc_{r}^{\ell} \de W_r \big|^2 \Big] \de s 
        = \int_{\ell}^{\ell+h} \E_{\ell} \Big[ \int_{\ell}^{s} \big\| \Zc_{r}^{\ell} \big\|^2 \de r \Big] \de s 
        \leq h \E_{\ell} \Big[ \int_{\ell}^{\ell+h} \big\| \Zc_{r}^{\ell} \big\|^2 \de r \Big] .
	\end{equation}
	Moreover, with the Lipschitz continuity of $\fd{\unboundedTimeGrid}$ we have
	\begin{align}\label{def:theGronwallR_h_dependent}
		h^3\ \E_t \Big[ \sum_{\ell\in\unboundedTimeGrid\cap [t,\StoppDiscrete)} \fd{\unboundedTimeGrid} (\phi(\ell), \Y_{\ell})^2 \Big] & 
        \leq 2 h^3\ \E_t \Big[ \sum_{ \ell\in\unboundedTimeGrid\cap [t,\StoppDiscrete)} \fd{\unboundedTimeGrid} (\phi(\ell),0)^2 + \lipschitz{f}^2 \Y_{\ell}^2\ \Big] \notag \\
		& \leq 2h^2\ \E_t \Big[ \StoppDiscrete\ \sup_{s\in\unboundedTimeGrid\cap [0,\StoppDiscrete)} \fd{\unboundedTimeGrid} (\phi(s),0)^2 + \StoppDiscrete\ \lipschitz{f}^2\ \sup_{s\in\unboundedTimeGrid\cap [0,\StoppDiscrete)} \Y_{s}^2\ \Big]\notag \\
		& \defined 2 h^2\  \E_{t} \big[ \theGronwallR (h) \big]
	\end{align} 
    where due to Assumption~\hyperref[discrete_f:assumption]{($\unboundedTimeGrid$-F)}, Proposition~\ref{prop:FixPointScheme_Lq_Integrable} (applied with $q=4$ resp. $q=8$ for the \textit{additional part}) and the existence of a positive exponential moment of $\StoppDiscrete$ we obtain\footnote{
        To be precise: We highlight the dependency on the underlying stepsize of all involved processes. We have
        \begin{equation}
            \E \big[ \sup_{h\in(0,h_0]} \StoppDiscrete(h)^r \big]^\frac{1}{2} \Big( \E \big[ \sup_{h\in(0,h_0]} \sup_{s\in\unboundedTimeGrid\cap[0,\StoppDiscrete(h))} f_h(\phi(s), 0)^{2r} \big]^\frac{1}{2} 
            + \lipschitz{f}^2 \E \big[ \sup_{h\in(0,h_0]} \sup_{s\in\unboundedTimeGrid\cap [0,\StoppDiscrete(h)) } \Y_s(h)^{2r} \big]^\frac{1}{2} \Big) < \infty 
        \end{equation}
        where the bound is finite due to Assumption~\hyperref[discrete_f:assumption]{($\unboundedTimeGrid$-F)}, Proposition~\ref{prop:FixPointScheme_Lq_Integrable} (with $q=4$ resp. $q=8$) and the combination of Lemma~\ref{lemma:A_exponentialMoment_ForAll_PowerMoments_SUPREMUM_VERSION} with Assumption~\hyperref[exp:assumption]{(R-q)}. From an application of the Cauchy-Schwarz inequality (case $r=2$) we obtain $\theGronwallR \in L^2$ resp. (case $r=4$) $\theGronwallR \in L^4$.
    }
    \begin{equation}\label{def:theGronwallR}
        \theGronwallR \defined \sup_{h\in (0,h_0]} \theGronwallR (h) \in L^2 \qquad \big(\text{resp. } \theGronwallR \in L^4 \big)
    \end{equation}
	Combine this with the estimates from above (using monotone convergence) to conclude
	\begin{align}
		\E_{t} \Big[ \int_{t\wedge\StoppDiscrete}^{\StoppDiscrete} \big| \Yc_s - \Y_{\phi(s)} \big|^2 \de s \Big] 
		& \leq 4h^2 \E_t \big[ \theGronwallR \big]  + 2h \E_t \Big[ \int_{t\wedge\StoppDiscrete}^{\StoppDiscrete} \big\| \Zc_{r}^{\phi(r)} \big\|^2 \de r \Big] \\
		& \leq 4h^2 \E_t \big[ \theGronwallR \big] + 4h \E_t \Big[ \int_{t\wedge\StoppDiscrete}^{\StoppDiscrete} \big\| Z_{r} \big\|^2 \de r + \int_{t\wedge\StoppDiscrete}^{\StoppDiscrete} \big\|Z_r - \Zc_{r}^{\phi(r)} \big\|^2 \de r \Big]
	\end{align}
	Now let $\alpha >0$. With an elementary bound\footnote{
		For any $a,b\geq 0$ and $\alpha>0$ it holds that $2ab\leq \alpha a^2 + \alpha^{-1} b^2$.
	}
	and the Lipschitz continuity of $\fd{\unboundedTimeGrid}$ we compute 
	\begin{align}
		& \bigg| \E_{t} \Big[ \int_{t\wedge\StoppDiscrete}^{\StoppDiscrete} \Delta_s \big( \fd{\unboundedTimeGrid} (\phi(s), \Yc_s) - \fd{\unboundedTimeGrid} (\phi(s), \Y_{\phi(s)}) \big) \de s \Big] \bigg| \\
		& \hspace*{1.0cm} \leq \tfrac{\lipschitz{f}\alpha}{2} \E_t \Big[ \int_{t\wedge\zeta}^{\zeta} \Delta_s^2 \de s \Big] + \tfrac{\lipschitz{f}}{2\alpha} \E_{t} \Big[ \int_{t\wedge\StoppDiscrete}^{\StoppDiscrete} \big| \Yc_s - \Y_{\phi(s)} \big|^2 \de s \Big] \\
		& \hspace*{1.0cm} \leq \tfrac{\lipschitz{f}\alpha}{2} \E_t \Big[ \int_{t\wedge\zeta}^{\zeta} \Delta_s^2 \de s \Big] + \tfrac{2\lipschitz{f}}{\alpha} h\ \E_t \Big[ \int_{t\wedge\StoppDiscrete}^{\StoppDiscrete} \big\|Z_r - \Zc_{r}^{\phi(r)} \big\|^2 \de r \Big] \\
        & \hspace*{2.0cm} + \tfrac{2\lipschitz{f}}{\alpha} h \ 
		\E_t \Big[ h\theGronwallR + \int_{t\wedge\StoppDiscrete}^{\StoppDiscrete} \big\| Z_{r} \big\|^2 \de r \Big]
	\end{align}
	We complete the proof by further bounding the integral of the (continuous time) martingale integrand $Z$ up to $[0,\infty)$.
\end{proof}
\begin{lemma}\label{lemma:AdaptedGronwallInt_4}    
    With an elementary bound for any $\alpha>0$ it holds that 
    \begin{multline}
        \Big|\int_{t\wedge\tau}^{\tau} \Delta_s \big( f(s, Y_s) - \fd{\unboundedTimeGrid} (\phi(s),Y_s) \big) \de s \Big| \leq \tfrac{\alpha}{2} \int_{t\wedge\tau}^{\tau} \Delta_s^2 \de s\\ + 
        \tfrac{1}{2\alpha}\int_{t\wedge\tau}^{\tau} \big( f(s, Y_s) - \fd{\unboundedTimeGrid} (\phi(s),Y_s) \big)^2 \de s \quad\text{for all } t\geq 0. \closeEqn
    \end{multline}
\end{lemma}
\section{Existence of Discrete-Time Approximation}\label{sec:Existence}

Let $p \geq 4$ and $\Bar{\xi} \in L^p (\F_{\StoppDiscrete})$. Let $h\in(0,h_0]$ be fixed. Consider the collection of sequences
\begin{equation}\label{def:SequenceSpace_V_DiscreteTerminal}
    \mathcal{V} (\StoppDiscrete) \defined \big\{ (R_t)_{t\in\unboundedTimeGrid} \ \big| \ R_t \in L^2 (\F_t),\ \1_{\{t > \StoppDiscrete\}} R_t = 0\ \text{for all}\ t\in\unboundedTimeGrid\  \big\}
\end{equation}
and set
$ \|R\|_{\StoppDiscrete} \defined \sum_{t\in\unboundedTimeGrid} \varphi (t) \| \1_{\{ t \leq \StoppDiscrete\} } R_t \|_{L^2}$ 
where 
$\varphi\colon\unboundedTimeGrid\to [1,\infty);\ t\mapsto (1-3\lipschitz{f} h)^{-\frac{t}{h}}. $
Subsequently we analyze the operator $T\colon (\mathcal{V}(\StoppDiscrete),\|\cdot\|_{\StoppDiscrete}) \to (\mathcal{V}(\StoppDiscrete), \|\cdot\|_{\StoppDiscrete})$ defined as
\begin{equation}\label{def:FixedPointOperator}
    T(R) \defined \big(\ \1_{\{t<\StoppDiscrete\}} \E_t [R_{t+h}] +\1_{\{t<\StoppDiscrete\}} h \fd{\unboundedTimeGrid} (t, R_t ) + \1_{ \{ t = \StoppDiscrete \} } \Bar{\xi} \ \big)_{t\in\unboundedTimeGrid}
\end{equation}
Below in Lemma~\ref{lemma:CompletenessV} we show that $(\mathcal{V}(\StoppDiscrete), \|\cdot\|_{\StoppDiscrete})$ is a Banach space. Proposition~\ref{prop:Characterization} provides a characterization of the norm and with this we conclude that $T$ is a contraction, see Lemma~\ref{lemma:T_Contraction}.
By Banach's fixed-point theorem there is a unique $\Y^* \in (\mathcal{V}(\StoppDiscrete), \|\cdot\|_{\StoppDiscrete})$ such that $\Y^* = T(\Y^*)$. Thus a solution of the discrete-time BSDE approximation scheme \eqref{def:FixpointScheme} is given by the composition
\begin{equation}
    \Y_t = \1_{ \{ t\geq\StoppDiscrete \} } \Bar{\xi} + \1_{ \{ t<\StoppDiscrete \} } \Y^*_t = \1_{ \{ t\geq\StoppDiscrete \} } \Bar{\xi} + \Y^*_t .
\end{equation}
\begin{lemma}\label{lemma:CompletenessV}
    The space $(\mathcal{V}(\StoppDiscrete), \|\cdot\|_{\StoppDiscrete})$ is Banach. \close
\end{lemma}
\begin{proof}
    \step{1} We show that $(\mathcal{V}(\StoppDiscrete), \|\cdot\|_{\StoppDiscrete})$ is a normed space.

    The linear space properties are inherited component-wise from $L^q$ to $\mathcal{V}$. Furthermore for $R,S \in \mathcal{V}(\StoppDiscrete)$ and $\alpha \in \R$ we observe that 
    $\1_{\{t\leq\StoppDiscrete\}} (\alpha R+S)_t = \alpha \1_{\{t\leq\StoppDiscrete\}} R_t + \1_{\{t\leq\StoppDiscrete\}} S_t = 0$.
    Regarding the norm we observe that homogeneity and the triangle inequality carry over from the $L^2$-norm to $\|\cdot\|_{\StoppDiscrete}$. We verify positive definiteness; Let $R\in\mathcal{V}(\StoppDiscrete)$ such that $\|R\|_{\StoppDiscrete} = 0$. For each $t\in\unboundedTimeGrid$ from the $L^2$-norm we obtain $R_t = 0$ on $\{t \leq \StoppDiscrete \}$, thus
    $R_t = \1_{\{t\leq\StoppDiscrete\}} R_t = 0$.
    
    \step{2} We show completeness.

    For each $t\in\unboundedTimeGrid$ consider $\de \Q \defined \1_{\{t\leq\StoppDiscrete\}} \de \prob$ and set
    $E_t \defined L^{2} \big( \Omega\cap\{t\leq\StoppDiscrete\}, \F_{t}\cap\{t\leq\StoppDiscrete\}, \Q \big)$. Then $(E_t)_{t\in\unboundedTimeGrid}$ is a sequence of Banach spaces. Particularly we can express
    \begin{equation}
        \big( \mathcal{V}(\StoppDiscrete), \|\cdot\|_{\StoppDiscrete} \big) = \Big\{ (X_t)_{t\in\unboundedTimeGrid} \in \prod_{t\in\unboundedTimeGrid} E_n \ \Big|\ \sum_{t\in\unboundedTimeGrid} \varphi(t) \| X_t \|_{E_t} < \infty  \Big\}
    \end{equation}
    and thus completeness of the spaces $E_t$ for all $t\in\unboundedTimeGrid$ is inherited to the product space.\footnote{
        Given a Cauchy sequence in the product space one can construct its limit via the component-wise limits. To be precise: Let $X = (X^n)_{n\in\nat} \subset \mathcal{V}(\StoppDiscrete)$ denote a $\|\cdot\|_{\StoppDiscrete}$ Cauchy sequence. For each $t\in\unboundedTimeGrid$ the respective component has a limit in $L^2 (\F_t)$, i.e. there is $X_t^\infty$ such that $X_t^n \to X_t^\infty$ as $n\to \infty$ in $L^2 (\F_t)$. By setting $X^\infty \defined (X_t^\infty)_{t\in\unboundedTimeGrid}$ we obtain the corresponding limit in $(\mathcal{V}(\StoppDiscrete), \|\cdot\|_{\StoppDiscrete})$. The convergence is a consequence of Fatou's lemma;
        \begin{equation}
            \| X^\infty\! -\! X^n \|_{\StoppDiscrete} =\! \sum_{t\in\unboundedTimeGrid}\! \varphi(t)\! \lim_{m\to\infty}\! \| X_t^m\! -\! X_t^n \|_{L^2}\! 
            \leq \liminf_{m\to\infty} \sum_{t\in\unboundedTimeGrid}\! \varphi(t) \| X_t^m\! -\! X_t^n \|_{L^2}\! 
            \leq\!\! \sup_{m,n > n_0}\!\! \| X^m\! -\! X^n \|_{\StoppDiscrete} <\! \varepsilon
        \end{equation}
    } 
\end{proof}
\begin{proposition}[Characterization of $\|\cdot\|_{\StoppDiscrete}$]\label{prop:Characterization}
    Assume \hyperref[stepsize:assumption]{(h)} and \hyperref[exp:assumption]{(R-q)} with $q=4$. Let $p\geq 4$ and $R=(R_t)_t$ be an adapted sequence in $L^p$ with $\sup_{t\in\unboundedTimeGrid} \|\1_{ \{ t\leq\StoppDiscrete \}}R_t\|_{L^p} <\infty$ then the \textit{cut-off} process $(\1_{\{t\leq\StoppDiscrete\}} R_t)_{t\in\unboundedTimeGrid}$ is contained in $(\mathcal{V}(\StoppDiscrete),\|\cdot\|_{\StoppDiscrete})$. 
    Moreover any $S\in(\mathcal{V}(\StoppDiscrete),\|\cdot\|_{\StoppDiscrete})$ necessarily satisfies 
    \begin{equation}
        \Big( \tfrac{1}{1-3\lipschitz{f} h} \Big)^\frac{t}{h} \1_{ \{t\leq\StoppDiscrete\}} \ S_t^2 \to 0 \quad\text{as } t\to\infty \quad\text{in } L^1. \closeEqn
    \end{equation}
\end{proposition}
\begin{proof}
    Let $t\in\unboundedTimeGrid$. With the Cauchy-Schwarz inequality we compute
    \begin{equation}\label{eqn:ProofExistenceNorm_EQ1}
        \big\| \1_{ \{ t\leq\StoppDiscrete \} } R_t \big\|_{L^2} \leq \prob \big[ t \leq \StoppDiscrete \big]^\frac{1}{4}\ \E \big[ \1_{ \{ t\leq\StoppDiscrete \}} R_t^4 \big]^\frac{1}{4} \leq \prob \big[ t \leq \StoppDiscrete \big]^\frac{1}{4}\ \sup_{s\in\unboundedTimeGrid} \big\|\1_{ \{ s\leq\StoppDiscrete \}}R_s\big\|_{L^4} .
    \end{equation}
    Let $\rho > 16\lipschitz{f}$ denote the scalar from Assumption~\hyperref[exp:assumption]{(R-q)}. With the Markov inequality we have
    $\prob [ t \leq \StoppDiscrete ] \leq \exp(-\rho t) \E [ \exp(\rho\StoppDiscrete) ] < \infty$.
    Thus there is $C\in [0,\infty)$ (not depending on $t$); 
    \begin{equation}
        \big\| \1_{ \{ t\leq\StoppDiscrete \} } R_t \big\|_{L^2} \leq \E \big[ \exp({\rho}\StoppDiscrete) \big]^\frac{1}{4} \sup_{s\in\unboundedTimeGrid} \big\|\1_{ \{ s\leq\StoppDiscrete \}}R_s\big\|_{L^4} \exp \big( -\tfrac{1}{4} {\rho} t \big) \defined C \exp \big( -\tfrac{1}{4} {\rho} t \big)
    \end{equation}
    Using Assumption~\hyperref[stepsize:assumption]{(h)} and an elementary bound\footnote{
        By Assumption~\hyperref[stepsize:assumption]{(h)} we have $h_0 \leq \tfrac{1}{4} \tfrac{1}{3\lipschitz{f}}$. 
        Together with $1+x\leq\exp(x)$ yields
        $\varphi (t) \leq \exp ( 4\lipschitz{f} t )$
        for $t\in\unboundedTimeGrid$.
    }
    we conclude that the geometric limit exists;
    \begin{align}
        \| R \|_{\StoppDiscrete} & \leq C \sum_{t\in\unboundedTimeGrid} \varphi(t) \exp \big( -\tfrac{1}{4} {\rho} t \big) \leq C \sum_{t\in\unboundedTimeGrid} \exp \Big( \big(4\lipschitz{f} -\tfrac{1}{4} {\rho} \big) t \Big) < \infty
    \end{align}
    Regarding the second part of the statement: let $S\in (\mathcal{V}(\StoppDiscrete), \|\cdot\|_{\StoppDiscrete})$ and notice by definition of the norm necessarily it holds that
    $\varphi(t) \ \| \1_{\{ t \leq \StoppDiscrete\} } S_t \|_{L^2} \to 0$ as $t\to \infty$
    whence the respective squares also are a null-sequence. The claim now is due to $\varphi (\cdot) \leq \varphi(\cdot)^2$;
    \begin{align}
        \Big( \tfrac{1}{1-3\lipschitz{f} h} \Big)^\frac{t}{h}\ \E \big[ \1_{ \{t \leq \StoppDiscrete\}} \ S_t^2 \big] & 
        = \big\| \varphi(t) \1_{\{ t \leq \StoppDiscrete\} } S_t^2 \big\|_{L^1}
        \leq \big\| \varphi(t) \1_{\{ t \leq \StoppDiscrete\} } S_t \big\|_{L^2}^2   \qedhere
    \end{align}
\end{proof}
\begin{lemma}[Contraction]\label{lemma:T_Contraction}
    Assume \hyperref[stepsize:assumption]{(h)}, \hyperref[exp:assumption]{(R-q)} with $q=4$, \hyperref[discrete_f:assumption]{($\unboundedTimeGrid$-F)} and \hyperref[discrete_xi:assumption]{($\unboundedTimeGrid$-T)} (both with $p=4$). 
    Then the operator $T$ defined in \eqref{def:FixedPointOperator} is a contraction (w.r.t $\|\cdot\|_{\StoppDiscrete}$). \close
\end{lemma}
\begin{proof}
    Notice by definition $\1_{\{t\geq \StoppDiscrete\}} T (\cdot) = 0$. Hence if $\|T(0)\|_{\StoppDiscrete} < \infty $ and $T$ is a contraction (w.r.t. $\|\cdot\|_{\StoppDiscrete}$) then the range of $T$ is contained in $(\mathcal{V}(\StoppDiscrete),\|\cdot\|_{\StoppDiscrete})$.

    By \hyperref[discrete_f:assumption]{($\unboundedTimeGrid$-F)} and \hyperref[discrete_xi:assumption]{($\unboundedTimeGrid$-T)} we have $\sup_{t\in\unboundedTimeGrid} \1_{ \{t\leq\StoppDiscrete\} } \fd{\unboundedTimeGrid} (t,0) \in L^p$ and $\Bar{\xi}\in L^p$. Hence with Proposition~\ref{prop:Characterization}; 
    \begin{align}
        \big\|T(0)\big\|_{\StoppDiscrete} & = \sum_{t\in\unboundedTimeGrid} \varphi(t) \big\| \1_{ \{t<\StoppDiscrete\} } \fd{\unboundedTimeGrid} (t,0) + \1_{ \{t=\StoppDiscrete\} } \Bar{\xi} \big\|_{L^2} \\
        & \leq \sum_{t\in\unboundedTimeGrid} \varphi(t) \big\| \1_{ \{t\leq\StoppDiscrete\} } \fd{\unboundedTimeGrid} (t,0) \big\|_{L^2} +  \sum_{t\in\unboundedTimeGrid} \varphi(t) \big\| \1_{ \{t\leq\StoppDiscrete\} } \1_{ \{t = \StoppDiscrete\} } \Bar{\xi} \big\|_{L^2} \\
        & = \Big\| \big( \1_{ \{t\leq\StoppDiscrete\} } \fd{\unboundedTimeGrid} (t,0) \big)_{t\in\unboundedTimeGrid} \Big\|_{\StoppDiscrete} + \Big\| \big( \1_{ \{t =\StoppDiscrete\} } \Bar{\xi} \big)_{t\in\unboundedTimeGrid} \Big\|_{\StoppDiscrete} < \infty
    \end{align}
    where we make explicit use of $\Bar{\xi}\in\F_{\StoppDiscrete}$ to use that $(\1_{ \{t =\StoppDiscrete\} } \Bar{\xi})_{t\in\unboundedTimeGrid}$ is adapted as required and thus contained in $\mathcal{V}(\StoppDiscrete)$.
    Regarding contractivity let $R,S \in (\mathcal{V}(\StoppDiscrete),\|\cdot\|_{\StoppDiscrete})$. Let $t\in\unboundedTimeGrid$ and observe
    \begin{align}
        T(R)_t - T(S)_t 
        & = \E_{t} \big[ \1_{ \{ t+h \leq \StoppDiscrete \} } (R_{t+h} - S_{t+h}) \big] + \1_{ \{t<\StoppDiscrete\} } h \big( \fd{\unboundedTimeGrid} (t,R_t) - \fd{\unboundedTimeGrid} (t,S_t) \big)
    \end{align}
    Furthermore with the Lipschitz continuity of $f$ we obtain
    \begin{equation}
        \big\| \1_{ \{t<\StoppDiscrete\} } \big(\fd{\unboundedTimeGrid} (t,R_t) - \fd{\unboundedTimeGrid} (t,S_t) \big) \big\|_{L^2} \leq \lipschitz{f} \E \big[ \1_{ \{t\leq\StoppDiscrete\} } (R_t-S_t)^2 \big]^\frac{1}{2} = \lipschitz{f} \big\| \1_{ \{t\leq\StoppDiscrete\} } (R_t-S_t) \big\|_{L^2}
    \end{equation}
    With those observations at hand we can verify that $T$ is a contraction. We use the triangle inequality of $\|\cdot\|_{L^2}$, Jensen's inequality and the fact that $\varphi(t)=\varphi(t+h)(1-3\lipschitz{f}h)$ to compute
    \begin{align}
        \big\| T(R) - T(S) \big\|_{\StoppDiscrete} & = \sum_{t\in\unboundedTimeGrid} \varphi (t) \big\| \1_{ \{t\leq\StoppDiscrete\} } \big(T(R)_t-T(S)_t\big) \big\|_{L^2} \\
        & \leq \sum_{t\in\unboundedTimeGrid} \varphi (t) \Big( \big\| \1_{ \{t+h\leq\StoppDiscrete\} } (R_{t+h}-S_{t+h}) \big\|_{L^2} + \lipschitz{f} h \big\| \1_{ \{t\leq\StoppDiscrete\} } (R_t-S_t) \big\|_{L^2} \Big)\\
        & = (1-2\lipschitz{f}h) \| R-S \|_{\StoppDiscrete}. \qedhere
    \end{align}
\end{proof}
\section{Majorant of Discrete-Time Approximation}\label{sec:Majorant_DiscreteTime}

In this section we derive a majorant in terms of a closed martingale;
\begin{equation}\label{def:ClosedMartingal_Majorant_FixpointScheme}
    \1_{ \{t<\StoppDiscrete\} }\Y_t^2 \leq \E_t \bigg[ \exp( 4\lipschitz{f} \StoppDiscrete) \Big( \Bar{\xi}^2 + \tfrac{1}{ (1 - 3\lipschitz{f}h) \lipschitz{f} } \sum_{\ell\in\unboundedTimeGrid\cap[0,\StoppDiscrete)} h \fd{\unboundedTimeGrid} (\ell,0)^2 \Big)\bigg]\quad\text{for all } t\in\unboundedTimeGrid
\end{equation}
This is a direct consequence of a conditional bound derived in Proposition~\ref{prop:Majorant_FixpointScheme} below. With this majorant we derive a stronger integrability: The maximal degree of $L^q$-boundedness allowed by the parameters of the BSDE resp. the discrete-time approximation, see Proposition~\ref{prop:FixPointScheme_Lq_Integrable}.

\begin{proposition}[Majorant $\Y$]\label{prop:Majorant_FixpointScheme}
     Assume \hyperref[stepsize:assumption]{(h)}, \hyperref[exp:assumption]{(R-q)} with $q=4$, \hyperref[discrete_f:assumption]{($\unboundedTimeGrid$-F)} and \hyperref[discrete_xi:assumption]{($\unboundedTimeGrid$-T)} (both with $p=4$). 
    For each $t\in\unboundedTimeGrid$ on $\{t<\StoppDiscrete\}$ it holds that
    \begin{equation}
        \Y_t^2 \leq \E_t \Big[ \exp\big( 4\lipschitz{f}(\StoppDiscrete-t) \big) \Bar{\xi}^2
        + \tfrac{1}{(1 - 3\lipschitz{f}h) \lipschitz{f} }  \sum_{\ell\in\unboundedTimeGrid\cap[t,\StoppDiscrete)} \exp\big( 4\lipschitz{f}(\StoppDiscrete-\ell) \big) h \fd{\unboundedTimeGrid} (\ell,0)^2 \Big]. \closeEqn
    \end{equation}
\end{proposition}
\begin{proof}
    The majorant is obtained as consequence of a pathwise discrete Gronwall bound (see Lemma~\ref{lemma:DiscreteGronwall}). Subsequently we argue that the required preconditions are met.

    Let $t\in\unboundedTimeGrid$ be fixed.
    By definition (see \eqref{def:FixpointScheme} in the sense of Section~\ref{sec:Existence}) we have $\Y_t^2 \in L^1 (\F_t)$ as well as
    $\1_{ \{ t\geq \StoppDiscrete \} } \Y_t^2 = \1_{ \{ t\geq\StoppDiscrete \} } {\Bar{\xi}}^2$
    and by \hyperref[discrete_xi:assumption]{($\unboundedTimeGrid$-T)} we have $\Bar{\xi}\in L^4$;
    thus condition \eqref{lemma:pathwiseDiscreteGronwall_AssumptionConstantAfterStopp} is satisfied. From the properties of the space $(\mathcal{V}(\StoppDiscrete),\|\cdot\|_{\StoppDiscrete})$, see Proposition~\ref{prop:Characterization}, we know that $(\varphi(s) \1_{ \{s<\StoppDiscrete\} } \Y_s^2 )_{s\in\unboundedTimeGrid}$ is a $L^1$(-norm) null-sequence. Therefore by continuity of the conditional expectation we have
    \begin{equation}\label{eqn:ProofMajorant_FixpointScheme_EQ1}
        \lim_{I\to\infty} \Big( \frac{1}{1-3\lipschitz{f} h} \Big)^\frac{I+t}{h} \E_t \big[ \1_{ \{ t+I<\StoppDiscrete \} } \Y_{t+I}^2 \big] = 0 
    \end{equation}
    thus condition \eqref{eqn:AssumptionASConditionalConvergenceTerminal} is fulfilled. Furthermore Lemma~\ref{lemma:ConditionalMonotinicityFixedPointScheme} ensures that the conditional bound \eqref{lemma:pathwiseDiscreteGronwall_AssumptionConditionalBound} holds. To be precise the statment reads;
    \begin{equation}\label{eqn:ProofMajorant_FixpointScheme_EQ2}
        \big| \Y_{t}\big|^2 \leq \frac{1}{1 - 3\lipschitz{f} h} \E_{t}\big[ \Y_{t+h}^2\big] + \1_{\{t < \StoppDiscrete\}} \frac{h}{(1-3\lipschitz{f}h)\lipschitz{f}} \fd{\unboundedTimeGrid} (t,0)^2 
    \end{equation}
    From Assumption \hyperref[discrete_f:assumption]{($\unboundedTimeGrid$-F)} and the existence of a positive exponential moment of $\StoppDiscrete$ (see Assumption~\hyperref[exp:assumption]{(R-q)} and
    Lemma~\ref{lemma:A_exponentialMoment_ForAll_PowerMoments_SUPREMUM_VERSION}) using Hölder's inequality we have
    \begin{align}
        \E \Big[ \sum_{s\in\unboundedTimeGrid\cap [0,\StoppDiscrete) } h \fd{\unboundedTimeGrid} (s,0)^2 \Big] & \leq \E \Big[ \StoppDiscrete \sup_{s\in\unboundedTimeGrid\cap[0,\StoppDiscrete)} \fd{\unboundedTimeGrid} (s,0)^2 \Big] \\
        & \leq  \E \big[ \StoppDiscrete^2 \big]^\frac{1}{2} \E \Big[ \sup_{s\in\unboundedTimeGrid\cap[0,\StoppDiscrete)} \fd{\unboundedTimeGrid} (s,0)^4 \Big]^\frac{1}{2} < \infty.
    \end{align}
    Whence we are indeed in position to apply Lemma~\ref{lemma:DiscreteGronwall} and obtain 
    \begin{equation}
        \1_{ \{ t<\StoppDiscrete \} } \Y_{t}^2 \leq \E_{t} \bigg[ \Big( \frac{1}{1- 3\lipschitz{f}h} \Big)^\frac{\StoppDiscrete-t}{h} \Bar{\xi}^2 + \tfrac{1}{(1-3\lipschitz{f}h)\lipschitz{f}} \sum_{\ell\in\unboundedTimeGrid\cap[t,\StoppDiscrete)} \Big( \frac{1}{1 - 3\lipschitz{f}h } \Big)^\frac{\StoppDiscrete-\ell}{h} h \fd{\unboundedTimeGrid} (\ell,0)^2 \bigg]
    \end{equation}
    By Assumption~\hyperref[stepsize:assumption]{(h)} we have $3\lipschitz{f}h\leq \tfrac{1}{4}$ and with an elementary bound
    we bound the scalar;
    \begin{align}
        \Big( \frac{1}{ 1- 3\lipschitz{f}h } \Big)^\frac{\StoppDiscrete-\ell}{h} & \leq  \big(1 + 4\lipschitz{f}h \big)^\frac{\StoppDiscrete-\ell}{h} \leq \exp\big( 4\lipschitz{f} (\StoppDiscrete-\ell) \big) 
        \quad\text{for all } \ell \in \unboundedTimeGrid\cap[t,\StoppDiscrete). \qedhere
    \end{align}
\end{proof}
The next result is a property of the continuous interpolation. We emphasize that here we merely consider fixed $t\in\unboundedTimeGrid$ and its direct successor, not the entire concatenation as introduced in \eqref{def:YcAdapted}.
\begin{lemma}\label{lemma:Single_Step_Interpolation}
    Assume \hyperref[stepsize:assumption]{(h)}, \hyperref[exp:assumption]{(R-q)} with $q=4$, \hyperref[discrete_f:assumption]{($\unboundedTimeGrid$-F)} and \hyperref[discrete_xi:assumption]{($\unboundedTimeGrid$-T)} (both with $p=4$).
    For each $t\in\unboundedTimeGrid$ there is $\Zc^t \in \mathcal{H}^2$ such that
    \begin{equation}\label{eqn:lemmaSingleStepInterpolation_ExplicitRepresentation}
        \Yc_s = \Y_{t+h} + \int_{s}^{t+h} \fd{\unboundedTimeGrid} (t, \Y_t) \de t - \int_{s}^{t+h} \Zc_r^t \de W_r \quad\text{for } s\in[t,t+h]
    \end{equation}
    and particularly it holds that $(\Yc_s)_{s\in[t,t+h]} \in\mathcal{S}^2$. \close
\end{lemma}
\begin{proof}
    The construction in Section~\ref{sec:Existence} yields $\Y_t\in L^2$ for all $t\in\unboundedTimeGrid$. 
    Now fix $t\in\unboundedTimeGrid$. We apply the martingale representation theorem (cf. \eqref{def:MartingaleRepresentation_ForEach_t}) to $\Y_{t+h}$ and obtain a unique $\Zc^t\in H^2$ such that
    \begin{equation}
        \Y_{t+h} = \E_t \big[ \Y_{t+h} \big] + \int_{t}^{t+h} \Z_r^t \de W_r
    \end{equation}
    With this we can establish \eqref{eqn:lemmaSingleStepInterpolation_ExplicitRepresentation} by setting
    \begin{equation}
        \Yc_s \defined \Y_{t+h} + (t+h - s) \fd{\unboundedTimeGrid} (t, \Y_t) - \int_{s}^{t+h} \Zc_r^t \de W_r \quad\text{for } s\in[t,t+h]
    \end{equation}
    For the second part; Since $\Y_t, \Y_{t+h} \in L^2$ we have to consider the stochastic integral. We use the Burkholder-Davis-Gundy inequality together with $\Zc^t\in H^2$ and the It\=o isometry to compute
    \begin{align}
        \E \bigg[ \sup_{s\in[t,t+h]} \Big| \int_{s}^{t+h} \Zc_r^t \de W_r \Big|^2 \bigg] & 
        \leq 2 \E \bigg[ \Big| \int_{t}^{t+h} \Zc_r^t \de W_r \Big|^2 \bigg] + 2 \E \bigg[ \sup_{s\in[t,t+h]} \Big| \int_{t}^{s} \Zc_r^t \de W_r \Big|^2 \bigg] \\
        & \leq 2 (1+C_2) \E \Big[ \int_{t}^{t+h} \big| \Zc_r^t \big\|^2 \de r \Big] < \infty. \qedhere
    \end{align}
\end{proof}
\begin{lemma}\label{lemma:ConditionalMonotinicityFixedPointScheme}   
    Assume \hyperref[stepsize:assumption]{(h)}, \hyperref[exp:assumption]{(R-q)} with $q=4$, \hyperref[discrete_f:assumption]{($\unboundedTimeGrid$-F)} and \hyperref[discrete_xi:assumption]{($\unboundedTimeGrid$-T)} (both with $p=4$). 
    Then
    \begin{equation}
        \big| \Y_{t}\big|^2 \leq \frac{1}{1 - 3\lipschitz{f} h} \E_{t}\big[ \Y_{t+h}^2\big] + \1_{\{t < \StoppDiscrete\}} \frac{h}{(1-3\lipschitz{f}h)\lipschitz{f}} \fd{\unboundedTimeGrid} (t,0)^2 \quad\text{for all } t\in\unboundedTimeGrid. \closeEqn
    \end{equation}
\end{lemma}
\begin{proof}
    Throughout the proof let $t\in\unboundedTimeGrid$ be fixed. 
    We use the interpolation $(\Yc_s)_{s\in[t,t+h]}\in\mathcal{S}^2$ (see Lemma~\ref{lemma:Single_Step_Interpolation}) for a single step on the discrete-time grid $\unboundedTimeGrid$. By It\=o's formula we can express 
    \begin{equation}
		\Y_{t+h}^2-\Y_{t}^2 = \int_{t\wedge\StoppDiscrete}^{(t+h)\wedge\StoppDiscrete} 2 \Yc_{s} \fd{\unboundedTimeGrid} (t,\Y_t) + \big\|\Zc_s^t \big\|^2 \de s
		- 2 \int_{t\wedge\StoppDiscrete}^{(t+h)\wedge\StoppDiscrete} \Yc_{s} \Zc_s^t \de W_s
	\end{equation}
	and since $\Yc \Zc^{t}\de W$ is a martingale (on $[t,t+h]$) we obtain
    \begin{equation}
        \E_{t}\big[ \Y_{t+h}^2\big] - \Y_{t}^2 = \E_t \Big[ \int_{t\wedge\StoppDiscrete}^{(t+h)\wedge\StoppDiscrete} 2 \Yc_{s} \fd{\unboundedTimeGrid} (t,\Y_{t}) + \big\|\Zc_s^t \big\|^2 \de s \Big] .
    \end{equation}
    Using the definition of $\Yc$ we can express
    \begin{align}
        \int_{t\wedge\StoppDiscrete}^{(t+h)\wedge\StoppDiscrete} \Yc_{s} \fd{\unboundedTimeGrid} (t,\Y_{t})\de s & = \big( (t+h)\wedge\StoppDiscrete - t\wedge\StoppDiscrete\big)\ \Y_{t+h} \fd{\unboundedTimeGrid} (t,\Y_{t}) \\
        & \hspace*{1.0cm} + \int_{t\wedge\StoppDiscrete}^{(t+h)\wedge\StoppDiscrete} \big( (t+h)\wedge\StoppDiscrete - s\big) \de s\ \fd{\unboundedTimeGrid} (t,\Y_{t})^2 \\
        & \hspace*{2.0cm} - \int_{t\wedge\StoppDiscrete}^{(t+h)\wedge\StoppDiscrete} \int_s^{t+h} \Zc_{r}^{t} \de W_r \de s\ \fd{\unboundedTimeGrid} (t,\Y_{t}).
    \end{align}
    Since on $\{ t \geq \StoppDiscrete \}$ we have $\Zc^{t} = 0$, we use a conditional version of Fubini's theorem to compute
    \begin{align}
        \E_t \Big[ \int_{t\wedge\StoppDiscrete}^{(t+h)\wedge\StoppDiscrete} \int_s^{t+h} \Zc_{r}^{t} \de W_r \de s \Big] 
        & = \int_{t}^{(t+h)} \E_t \bigg[ \E_s \Big[ \int_s^{t+h} \Zc_{r}^{t} \de W_r \Big] \bigg] \de s = 0.
    \end{align}
    Thus combining all yields the identity 
    \begin{multline}
        \Y_{t}^2 + 2 \E_t \Big[ \int_{t\wedge\StoppDiscrete}^{(t+h)\wedge\StoppDiscrete} \big( (t+h)\wedge\StoppDiscrete - s\big) \de s\ \fd{\unboundedTimeGrid} (t,\Y_{t})^2 \Big] + \E_t \Big[ \int_{t\wedge\StoppDiscrete}^{(t+h)\wedge\StoppDiscrete} \big\|\Zc_s^t \big\|^2 \de s \Big] \\
        = \E_{t}\big[ \Y_{t+h}^2\big] - 2 \big( (t+h)\wedge\StoppDiscrete - t\wedge\StoppDiscrete\big)\ \E_{t} \big[ \Y_{t+h} \big] \fd{\unboundedTimeGrid} (t,\Y_{t})
    \end{multline}
    and particularly 
    \begin{equation}\label{proof:ConditionalBound_EQ1}
        \Y_{t}^2 + \E_t \Big[ \int_{t\wedge\StoppDiscrete}^{(t+h)\wedge\StoppDiscrete} \big\|\Zc_s^t \big\|^2 \de s \Big] \\
        \leq \E_{t}\big[ \Y_{t+h}^2\big] - 2 \big( (t+h)\wedge\StoppDiscrete - t\wedge\StoppDiscrete\big)\ \E_{t} \big[ \Y_{t+h} \big] \fd{\unboundedTimeGrid} (t,\Y_{t}) .
    \end{equation}
    With the Lipschitz continuity of $\fd{\unboundedTimeGrid}$, the fact that $2ab\leq 2\lipschitz{f} a^2 + \tfrac{1}{2\lipschitz{f}} b^2$ for any $a,b\geq 0$ and Jensen's inequality on $\{t < \StoppDiscrete\}$ we have
    \begin{align}
        2 h \Big| \ \E_{t} \big[ \Y_{t+h} \big] \fd{\unboundedTimeGrid} (t,\Y_{t}) \Big| & 
        \leq 2\lipschitz{f} h \E_{t} \big[ \Y_{t+h}^2 \big] + \lipschitz{f}h \Y_t^2 + \tfrac{1}{\lipschitz{f}} h \fd{\unboundedTimeGrid} \big( t, 0 \big)^2
    \end{align}
    Use this estimate to further bound \eqref{proof:ConditionalBound_EQ1} and rearrange to obtain
    \begin{multline}\label{eqn:TheGronwallBound_BothWays}
        \big( 1 - \1_{\{t < \StoppDiscrete\}} \lipschitz{f} h \big) \Y_{t}^2 
        + \E_t \Big[ \int_{t\wedge\StoppDiscrete}^{(t+h)\wedge\StoppDiscrete} \big\|\Zc_s^t \big\|^2 \de s \Big] \\
        \leq \big(1 + \1_{\{t < \StoppDiscrete\}} 2\lipschitz{f} h \big) \E_{t}\big[ \Y_{t+h}^2\big] + \1_{\{t < \StoppDiscrete\}} \tfrac{h}{\lipschitz{f}} \fd{\unboundedTimeGrid} (t,0)^2
    \end{multline}
    where the factor on the left is positive since by Assumption~\hyperref[stepsize:assumption]{(h)} we have $h_0 < \lipschitz{f}$. Furthermore by Assumption~\hyperref[stepsize:assumption]{(h)} we have $\lipschitz{f} h_0 < \tfrac{1}{3}$;
    thus with an elementary calculation\footnote{
        For $a>0$ it holds that $(1+2a)(1-3a) = 1 - a -6a^2 \leq 1-a$. Thus whenever $a < \tfrac{1}{3}$ it holds that $\tfrac{1+2a}{1-a} \leq \tfrac{1}{1-3a}$. 
    } 
    we obtain
    \[
        \frac{1 + \1_{\{t < \StoppDiscrete\}} 2\lipschitz{f} h}{1 - \1_{\{t < \StoppDiscrete\}} \lipschitz{f} h} \leq \frac{1 + 2\lipschitz{f} h}{1 - \lipschitz{f} h} \leq \frac{1}{1-3\lipschitz{f}h}. \qedhere
    \]
\end{proof}
\begin{remark}\label{remark:Majorant_MartRep_CompleteStep}
	The estimate \eqref{eqn:TheGronwallBound_BothWays} of the preceding proof particularly yields that
	\begin{equation}
		\E_t \Big[ \int_{t\wedge\StoppDiscrete}^{(t+h)\wedge\StoppDiscrete} \big\| \Zc_s^{t} \big\|^2 \de s \Big] \leq (1+2\lipschitz{f}) \E_t \big[ \Y_{t+h}^2 \big] + \tfrac{h}{\lipschitz{f}} \fd{\unboundedTimeGrid} (t,0)^2 \quad\text{for all } t\in\unboundedTimeGrid 
	\end{equation}
	Combined with Proposition~\ref{prop:Majorant_FixpointScheme} (under \hyperref[stepsize:assumption]{(h)}, \hyperref[exp:assumption]{(R-q)} with $q=4$, \hyperref[discrete_f:assumption]{($\unboundedTimeGrid$-F)} and \hyperref[discrete_xi:assumption]{($\unboundedTimeGrid$-T)}) it holds that 
	\begin{multline}
		\E_t \Big[ \int_{t\wedge\StoppDiscrete}^{(t+h)\wedge\StoppDiscrete} \big\| \Zc_s^{t} \big\|^2 \de s \Big] \leq 2 \E_t \Big[ \exp\big( 4\lipschitz{f} (\StoppDiscrete-t) \big) \Bar{\xi}^2 \\
		+ \tfrac{1}{ (1 - 3\lipschitz{f}h) \lipschitz{f} } h \sum_{\ell\in\unboundedTimeGrid\cap[t,\StoppDiscrete)} \exp\big( 4\lipschitz{f} (\StoppDiscrete-\ell) \big) \fd{\unboundedTimeGrid} (\ell,0)^2 \Big] + \tfrac{h}{\lipschitz{f}} \fd{\unboundedTimeGrid} (t,0)^2. \closeEqn
	\end{multline}
\end{remark}
\begin{proposition}[Uniform $\mathcal{S}^q$-Property of $\Y$]\label{prop:FixPointScheme_Lq_Integrable}
    Assume \hyperref[stepsize:assumption]{(h)}, \hyperref[exp:assumption]{(R-q)} with $q\geq 4$, \hyperref[discrete_f:assumption]{($\unboundedTimeGrid$-F)} with $p>2q$ and \hyperref[discrete_xi:assumption]{($\unboundedTimeGrid$-T)} with $p=2q$. 
    Then it holds that
    \begin{equation}
        \sup_{h\in (0,h_0] } \sup_{t\in\unboundedTimeGrid} |\Y_t|^q \in L^1 \closeEqn
    \end{equation}
\end{proposition}
\begin{proof}
    Let $h\in (0,h_0]$ be fixed. 
    There is a constant $K>0$ (not depending on $h$) such that the random variable in the bound of Proposition~\ref{prop:Majorant_FixpointScheme} satisfies
    \begin{multline}\label{eqn:Majorant_2_wiht_CONSTANT_K}
        \exp \big( 4\lipschitz{f} (\StoppDiscrete -t)  \big) \Bar{\xi}^2 + \tfrac{1}{ (1-3\lipschitz{f}h)\lipschitz{f} } \sum_{\ell\in\unboundedTimeGrid\cap[t,\StoppDiscrete)} h \exp \big( 4\lipschitz{f} (\StoppDiscrete -\ell ) \big) \fd{\unboundedTimeGrid} (\ell,0)^2 \\
        \leq \exp \big( 4\lipschitz{f} \StoppDiscrete \big) \Big( \Bar{\xi}^2 + K \sum_{\ell\in\unboundedTimeGrid\cap[0,\StoppDiscrete)} h \fd{\unboundedTimeGrid} (\ell,0)^2 \Big) \quad\text{for all } t\in\unboundedTimeGrid
    \end{multline}
    and thus we obtain the majorant in terms of a closed martingale (see also \eqref{def:ClosedMartingal_Majorant_FixpointScheme})   
    \begin{equation}
        \1_{ \{t<\StoppDiscrete\} }\Y_t^2 \leq \E_t \bigg[ \exp( 4\lipschitz{f} \StoppDiscrete) \Big( \Bar{\xi}^2 + K \sum_{\ell\in\unboundedTimeGrid\cap[0,\StoppDiscrete)} h \fd{\unboundedTimeGrid} (\ell, 0)^2 \Big)\bigg]\quad\text{for all } t\in\unboundedTimeGrid 
    \end{equation}
    Recall $\Y, \Bar{\xi}, \fd{\unboundedTimeGrid}$ and $\StoppDiscrete$ depend on the above fixed stepsize $h$.\footnote{
        For completeness notice $t = t(h) \in \unboundedTimeGrid$. However the contribution to the bound is purely through the conditioning and negligible by allowing the conditional expectation on the continuous time interval $[0,\infty)$.
    }
    For the sake of the subsequent argument we restate the preceding bound, now emphasizing the dependence on the stepsize;
    \begin{align}
        \1_{ \{t<\StoppDiscrete(h) \} }\Y_t(h)^2 
        & \leq \E_t \bigg[ \exp\big( 4\lipschitz{f} \StoppDiscrete(h) \big) \Big( \Bar{\xi}(h)^2 + K \!\! \sum_{\ell\in\unboundedTimeGrid\cap[0,\StoppDiscrete(h) )} h \fd{h} (\ell,0)^2 \Big)\bigg] \\
        & \leq \E_t \Big[ \sup_{ \Delta \in (0,h_0] } \exp\big( 4\lipschitz{f} \StoppDiscrete(\Delta) \big) \big( \Bar{\xi}(\Delta)^2 + K\StoppDiscrete(\Delta) \!\! \sup_{s\in\unboundedTimeGrid\cap[0,\StoppDiscrete(\Delta))}\!  \fd{\Delta} (s,0)^2 \big) \Big]
    \end{align}
    Using Doob's $L^\frac{q}{2}$-inequality we obtain
    \begin{align}\label{eqn:FixPointScheme_Lq_Integrable_EQ1}
        & \E \Big[ \sup_{h\in(0,h_0]} \sup_{t\in\unboundedTimeGrid} \1_{ \{t<\StoppDiscrete\} } |\Y_t(h)|^q \Big] \notag \\
        & \hspace*{0.5cm} \leq \E \bigg[ \sup_{t\geq 0} \E_t \Big[ \sup_{ \Delta \in (0,h_0] } \exp\big( 4\lipschitz{f} \StoppDiscrete(\Delta) \big) \big( \Bar{\xi}(\Delta)^2 + K\StoppDiscrete(\Delta) \!\! \sup_{s\in\unboundedTimeGrid\cap[0,\StoppDiscrete(\Delta))}  \fd{\Delta} (s,0)^2 \big) \Big]^\frac{q}{2} \bigg] \notag \\
        & \hspace*{0.5cm} \leq C_{\frac{q}{2}}\ \E \Big[ \sup_{ \Delta \in (0,h_0] } \exp\big( 2q \lipschitz{f} \StoppDiscrete(\Delta) \big) \big( \Bar{\xi}(\Delta)^2 + K\StoppDiscrete(\Delta) \!\! \sup_{s\in\unboundedTimeGrid\cap[0,\StoppDiscrete(\Delta))}  \fd{\Delta} (s,0)^2 \big)^\frac{q}{2} \Big] 
    \end{align}
    At this point we use the Cauchy-Schwarz inequality
    to conclude that
    \begin{multline}
        \E \Big[ \sup_{ \Delta \in (0,h_0] } \exp\big( 2q \lipschitz{f} \StoppDiscrete(\Delta) \big) \big( \Bar{\xi}(\Delta)^2 + K\StoppDiscrete(\Delta) \!\! \sup_{s\in\unboundedTimeGrid\cap[0,\StoppDiscrete(\Delta))}  \fd{\Delta} (s,0)^2 \big)^\frac{q}{2} \Big]^2 \\
        \leq \E \Big[ \sup_{ \Delta \in (0,h_0] } \exp\big( 4q \lipschitz{f} \StoppDiscrete(\Delta) \big) \Big]
        \E \Big[ \sup_{ \Delta \in (0,h_0] } \big( \Bar{\xi}(\Delta)^2 + K\StoppDiscrete(\Delta) \!\! \sup_{s\in\unboundedTimeGrid\cap[0,\StoppDiscrete(\Delta))}  \fd{\Delta} (s,0)^2 \big)^q \Big]
    \end{multline}
    where the first factor is finite by Assumption~\hyperref[exp:assumption]{(R-q)}. With an elementary bound we compute
    \begin{multline}
        \sup_{ \Delta \in (0,h_0] } \big( \Bar{\xi}(\Delta)^2 + K\StoppDiscrete(\Delta) \!\! \sup_{s\in\unboundedTimeGrid\cap[0,\StoppDiscrete(\Delta))}  \fd{\Delta} (s,0)^2 \big)^q \\
        \leq 2^{q-1} \sup_{ \Delta \in (0,h_0] } \big|\Bar{\xi}(\Delta)\big|^{2q} 
        + 2^{q-1} K^q \sup_{ \Delta \in (0,h_0] } \StoppDiscrete(\Delta)^q \!\! \sup_{s\in\unboundedTimeGrid\cap[0,\StoppDiscrete(\Delta))}  \big|\fd{\Delta}(s,0)\big|^{2q}
    \end{multline}
    By Assumption \hyperref[discrete_xi:assumption]{($\unboundedTimeGrid$-T)} we have
    $ \sup_{ \Delta \in (0,h_0] } |\Bar{\xi}(\Delta)|^{2q} \in L^1 $
    and applying Hölder's inequality to $r \defined \tfrac{p}{2q} > 1$ and $\Tilde{r}\defined \tfrac{r}{r-1}$ yields
    \begin{multline}
        \E \Big[ \sup_{ \Delta \in (0,h_0] } \StoppDiscrete(\Delta)^q \!\! \sup_{s\in\unboundedTimeGrid\cap[0,\StoppDiscrete(\Delta))}  \big|\fd{\Delta} (s,0)\big|^{2q} \Big] \\
        \leq \E \Big[ \sup_{ \Delta \in (0,h_0] } \StoppDiscrete(\Delta)^{q\Tilde{r}} \Big]^\frac{1}{\Tilde{r}} 
        \E \Big[ \sup_{ \Delta \in (0,h_0] } \sup_{s\in\unboundedTimeGrid\cap[0,\StoppDiscrete(\Delta))}  \big|\fd{\Delta} (s,0)\big|^{p} \Big]^\frac{1}{r} < \infty
    \end{multline}
    where the $q\Tilde{r}$-moment of $\StoppDiscrete$ is finite due the existence of a positive exponential moment (see Lemma~\ref{lemma:A_exponentialMoment_ForAll_PowerMoments_SUPREMUM_VERSION}) and the second factor is finite by Assumption \hyperref[discrete_f:assumption]{($\unboundedTimeGrid$-F)}. Thus \eqref{eqn:FixPointScheme_Lq_Integrable_EQ1} is indeed finite and invoking once more Assumption \hyperref[discrete_xi:assumption]{($\unboundedTimeGrid$-T)} we finally obtain the claim;
    \[
        \tfrac{1}{2^{q}} \E \Big[ \sup_{ h \in (0,h_0] } \sup_{t\in\unboundedTimeGrid} \big|\Y_t(h)\big|^q \Big]\! \leq\! \E \Big[ \sup_{ h \in (0,h_0] } \sup_{t\in\unboundedTimeGrid} \1_{ \{t<\StoppDiscrete\} } \big|\Y_t(h)\big|^q \Big]\! +\! \E \Big[ \sup_{ h \in (0,h_0] } \big| \Bar{\xi}(h)\big|^q \big]\! <\! \infty. \qedhere
    \]
\end{proof}
\section{Regularity of the Adapted Interpolation}\label{Sec:RegularityInterpolation}

This section establishes the properties of the continuous adapted interpolation process $\Yc$ (see \eqref{def:YcAdapted}) as presented in Proposition~\ref{prop:InterpolationRegularity}. 
We derive Proposition~\ref{prop:InterpolationRegularity} (where $q=4$) by establishing a result where Assumption~\hyperref[exp:assumption]{(R-q)} is imposed for a general $q\geq 4$. 

\begin{proposition}\label{prop:InterpolationRegularity_q_Version}
    Assume \hyperref[stepsize:assumption]{(h)}, \hyperref[exp:assumption]{(R-q)} with $q\geq 4$, \hyperref[discrete_f:assumption]{($\unboundedTimeGrid$-F)} with $p>2q$ and \hyperref[discrete_xi:assumption]{($\unboundedTimeGrid$-T)} with $p=2q$. 
    Then 
    \begin{equation}
        \sup_{ h\in (0,h_0] } \sup_{t\geq 0} |\Yc_{t}|^q \in L^1
    \end{equation}
    Moreover (for each $h\in (0,h_0]$) the concatenation of martingale integrands $\Zc^\phi$ is in $\mathcal{H}^{2}$. \close
\end{proposition}
\begin{proof}
    The construction of $\Y$ in Section~\ref{sec:Existence} yields an adapted sequence, thus continuity and adaptedness of $\Yc$ are immediate from the definition.
    In Section~\ref{sec:Majorant_DiscreteTime} we derived a majorant in terms of a closed martingale for $\Y$.\footnote{
        See \eqref{def:ClosedMartingal_Majorant_FixpointScheme} or Proposition~\ref{prop:Majorant_FixpointScheme}. Moreover we refer to  \eqref{eqn:Majorant_2_wiht_CONSTANT_K} in the proof of Proposition~\ref{prop:FixPointScheme_Lq_Integrable}.
    }
    Since $\Yc = \Y$ on $\unboundedTimeGrid$ we particularly have
    \begin{equation}\label{eqn:ClosedMartingale_Interpolation_Along_DiscreteTimeGrid}
    	\Yc_{\ell}^2 = \Y_{\ell}^2 \leq \E_{\ell} \Big[ \exp(4\lipschitz{f}\StoppDiscrete) \bar{\xi}^2 + \exp (4\lipschitz{f}\StoppDiscrete) \tfrac{1}{(1-3\lipschitz{f}h_0)\lipschitz{f}} \int_{0}^{\StoppDiscrete} \fd{\unboundedTimeGrid} ( \phi(s),0)^2 \de s \Big] \ \ \text{for all }\ell\in\unboundedTimeGrid
    \end{equation}
    Subsequently we use a stochastic Gronwall lemma to extend this to $[0,\infty)$. We bound $\Yc_t$ for any given $t\geq 0$ with the scheme $\Y$ at the surrounding discrete-times $\phi(t)$ and $\phi(t)+h$.
    
    \step{1} We derive a majorant for $\Yc^2$. 
    
    Let $t\geq 0$ be fixed. Now pick any $r\in[\phi(t),\phi(t)+h]$ and invoke It\=o's formula to obtain 
	\begin{equation}
		\Yc_{\phi(t)+h}^2-\Yc_{r}^2 = \int_{r\wedge\StoppDiscrete}^{(\phi(t)+h)\wedge\StoppDiscrete}\!\! 2\Yc_{s}\fd{\unboundedTimeGrid} (\phi(t),\Y_{\phi(t)}) + \big\|\Zc_s^{\phi(t)} \big\|^2 \de s
		-\int_{r\wedge\StoppDiscrete}^{(\phi(t)+h)\wedge\StoppDiscrete}\!\! 2\Yc_{s} \Zc_s^{\phi(t)} \de W_s
	\end{equation}
    Since $\Yc \Zc^{\phi(t)}\de W$ is a martingale (with constant horizon $\phi(t)+h$; see Lemma~\ref{lemma:Single_Step_Interpolation}) we have
	\begin{multline}
		\Yc_{r}^2 + \E_r\Big[\int_{r\wedge\StoppDiscrete}^{(\phi(t)+h)\wedge\StoppDiscrete} \big\|\Zc_s^{\phi(t)} \big\|^2 \de s \Big]  = \E_r \big[ \Y_{\phi(t)+h}^2 \big] + \E_r\Big[\int_{r\wedge\StoppDiscrete}^{(\phi(t)+h)\wedge\StoppDiscrete}\!\! 2 \Yc_{s} \fd{\unboundedTimeGrid} (\phi(t),\Y_{\phi(t)}) \de s \Big] \\
		\leq \E_r \Big[ \Y_{\phi(t)+h}^2 + \1_{ \{\phi(t)<\StoppDiscrete\}} h \fd{\unboundedTimeGrid} \big(\phi(t),\Y_{\phi(t)}\big)^2 \Big] + \E_r \Big[ \int_{r\wedge\StoppDiscrete}^{(\phi(t)+h)\wedge\StoppDiscrete} \Yc_{s}^2 \de s \Big] 
	\end{multline}
	From a stochastic Gronwall lemma (see e.g. \cite[Appendix A]{Schlegel2024ProbabilisticShape}) we obtain
	\begin{equation}
		\Yc_{r}^2 \leq \E_r \bigg[ \exp\big(\phi(t)+h -r\big) \Big( \Y_{\phi(t)+h}^2 + \1_{ \{\phi(t)<\StoppDiscrete\}} h \fd{\unboundedTimeGrid} \big(\phi(t),\Y_{\phi(t)}\big)^2 \Big) \bigg] \quad\text{for } r\in[\phi(t),\phi(t)+h].
	\end{equation}
    Especially for $t\geq 0$ (fixed above), and using the Lipschitz continuity of $\fd{\unboundedTimeGrid}$ we have
    %
    \begin{align}\label{eqn:Proof_Interpolation_L2_Bounded}
        \Yc_{t}^2 
        & \leq \exp(h) \E_{t} \Big[ \Y_{\phi(t)+h}^2 + \1_{ \{\phi(t)<\StoppDiscrete\}} h \fd{\unboundedTimeGrid} \big( \phi(t),\Y_{\phi(t)} \big)^2 \Big] \notag \\
        & = \exp(h) \E_{t} \big[ \Y_{\phi(t)+h}^2 \big] + \exp(h) 2\lipschitz{f}^2 h \Y_{\phi(t)}^2 + \exp(h) 2h \1_{ \{\phi(t)<\StoppDiscrete\}} \fd{\unboundedTimeGrid} \big( \phi(t), 0 \big)^2
    \end{align}
    and we have established the desired majorant. \close    

    \step{2} We establish $\Yc \in \mathcal{S}^q$ by showing the \textit{uniform} $\mathcal{S}^q$\textit{-property}; 
    \begin{equation}
        \sup_{h\in(0,h_0]} \sup_{t\geq 0} |\Yc_{t}(h)|^q \in L^1
    \end{equation}
    
    %
    Observe that from \eqref{eqn:Proof_Interpolation_L2_Bounded} with an elementary bound we obtain (with $c_q \defined 2^{q-2}$)
    \begin{align}
        \Yc_{t}^4 
        & \leq c_q \exp \big( \tfrac{q}{2} h\big) \Big( \E_{t} \Big[ \sup_{s\in\unboundedTimeGrid} \Y_{s}^2 \Big]^\frac{q}{2} + \lipschitz{f}^q h^\frac{q}{2} \sup_{s\in\unboundedTimeGrid} |\Y_{s}|^q + h^\frac{q}{2} \sup_{s\in\unboundedTimeGrid\cap [0,\StoppDiscrete)} \big| \fd{\unboundedTimeGrid} (s,0) \big|^q \Big) 
        \quad \text{for } t\geq 0
    \end{align}
    As in the proof of Proposition~\ref{prop:FixPointScheme_Lq_Integrable} we highlight the dependencies on the stepsize. Accordingly 
    \begin{multline}
        \sup_{h\in(0,h_0]} \sup_{t\geq 0} |\Yc_{t}(h)|^q \leq c_q \exp \big(\tfrac{q}{2}h_0\big) 
        \bigg(
            \sup_{t\geq 0} \E_{t} \Big[ \sup_{\Delta\in (0,h_0]} \sup_{s\in\unboundedTimeGrid} \Y_{s}^2 \Big]^\frac{q}{2} \\
            + \lipschitz{f}^q h_0^\frac{q}{2} \sup_{\Delta\in (0,h_0]} \sup_{s\in\unboundedTimeGrid} |\Y_{s}(\Delta)|^q
            + h_0^\frac{q}{2} \sup_{\Delta\in (0,h_0]} \sup_{s\in\unboundedTimeGrid\cap [0,\StoppDiscrete)} \big| \fd{\unboundedTimeGrid} (s,0) \big|^q
        \bigg)    
    \end{multline}
    From Proposition~\ref{prop:FixPointScheme_Lq_Integrable} (with $q\geq 4$) and Assumption~\hyperref[discrete_f:assumption]{($\unboundedTimeGrid$-F)} we know there is $C>0$ such that
    \begin{equation}
         \lipschitz{f}^q h_0^\frac{q}{2} \E \Big[ \sup_{\Delta\in (0,h_0]} \sup_{s\in\unboundedTimeGrid} | \Y_{s}(\Delta)|^q \Big]
            + h_0^\frac{q}{2} \Big[ \sup_{\Delta\in (0,h_0]} \sup_{s\in\unboundedTimeGrid\cap [0,\StoppDiscrete)} \big| \fd{\unboundedTimeGrid} (s,0)\big|^q \Big] \defined C < \infty
    \end{equation}
    Using this together with Doob's $L^\frac{q}{2}$-inequality (with constant $C_{\frac{q}{2}}$) we conclude
    \begin{align}
        \E \Big[ \sup_{h\in(0,h_0]} \sup_{t\geq 0} | \Yc_{t}(h)|^q \Big] & 
        \leq c_q \exp \big(\tfrac{q}{2}h_0\big) \E \bigg[ \sup_{t\geq 0} \E_{t} \Big[ \sup_{\Delta\in (0,h_0]} \sup_{s\in\unboundedTimeGrid} \Y_{s}^2 \Big]^\frac{q}{2} \bigg] 
        + c_q \exp \big(\tfrac{q}{2}h_0\big) C \\
        & \leq c_q \exp \big(\tfrac{q}{2}h_0\big) \bigg( C_{\frac{q}{2}} \E \Big[ \sup_{\Delta\in (0,h_0]} \sup_{s\in\unboundedTimeGrid} | \Y_{s}|^q \Big] 
        +  C \bigg) 
        < \infty \closeEqn
    \end{align}
    
    \step{3} We show that $\Zc^\phi \de W \in \mathcal{H}^2$ for each $h\in(0,h_0]$.
    
    From Remark~\ref{remark:Majorant_MartRep_CompleteStep} for each $\ell\in\unboundedTimeGrid$ we obtain
    \begin{equation}
    	\E_{\ell} \Big[ \int_{\ell \wedge \StoppDiscrete }^{ (\ell+h)\wedge\StoppDiscrete } \big\| \Zc_{s}^{\ell} \big\|^2 \de s \Big] 
        \leq 2 \E_{\ell} \bigg[ \exp (C\StoppDiscrete) \Big( \Bar{\xi}^2 + K \int_{0}^{\StoppDiscrete} f\big( \phi(s),0 \big)^2 \de s \Big) \bigg] + \tfrac{h}{\lipschitz{f}} \fd{\unboundedTimeGrid} (\ell,0)^2
    \end{equation}
    with $C \defined 4\lipschitz{f}, K\defined ((1-3\lipschitz{f}h_0)\lipschitz{f})^{-1}$. Recall $\1_{ \{\cdot\geq\StoppDiscrete\} } \Zc^\phi = 0 $ and use monotone convergence
    \begin{align}\label{eqn:ProofRegInterpol_EQ}
		\E \Big[ \int_{0}^{\infty} \big\| \Zc_s^{\phi(s)} \big\|^2 \de s \Big] & = \sum_{ \ell\in\unboundedTimeGrid } \E \bigg[ \1_{ \{ \ell\leq\StoppDiscrete\} } \E_t \Big[ \int_{\ell\wedge\StoppDiscrete}^{(\ell+h)\wedge\StoppDiscrete} \big\| \Zc_s^{t} \big\|^2 \de s \Big] \bigg] \notag \\
		& \leq 2 \sum_{\ell\in\unboundedTimeGrid} \E \Big[ \1_{ \{ \ell\leq\StoppDiscrete\} } \exp (C\StoppDiscrete) \Bar{\xi}^2 + \exp (C\StoppDiscrete) K\int_{0}^{\StoppDiscrete} \fd{\unboundedTimeGrid} \big( \phi(s),0 \big)^2 \de s \Big] \notag \\
		& \hspace*{1.0cm} + \tfrac{h}{\lipschitz{f}} \sum_{\ell\in\unboundedTimeGrid} \E \big[ \1_{ \{ \ell\leq\StoppDiscrete\} } \fd{\unboundedTimeGrid} (\ell,0)^2 \big] \notag \\
		& \leq 2 \tfrac{1}{h} \E \bigg[ \StoppDiscrete \Big( \exp(C\StoppDiscrete) \Bar{\xi}^2 + \exp(C\StoppDiscrete) K \int_{0}^{\StoppDiscrete} \fd{\unboundedTimeGrid} \big(\phi(s),0\big)^2 \de s  \Big) \bigg] \notag \\
		& \hspace*{1.0cm} + \tfrac{1}{\lipschitz{f}} \E \Big[ \StoppDiscrete\ \sup_{ s\in\unboundedTimeGrid\cap[0,\StoppDiscrete) } \fd{\unboundedTimeGrid} (s,0)^2 \Big] 
    \end{align}
    and with Hölder's inequality we check that the expectation in the bound is finite. For this notice 
	\begin{equation}
		\StoppDiscrete \exp(C\StoppDiscrete) \Big( \Bar{\xi}^2 + K \int_{0}^{\StoppDiscrete} \fd{\unboundedTimeGrid} \big(\phi(s),0\big)^2 \de s \Big) 
        \leq \exp(C\StoppDiscrete) \Big( \StoppDiscrete \Bar{\xi}^2 + \StoppDiscrete^2 K \sup_{s\in\unboundedTimeGrid\cap[0,\StoppDiscrete)} \fd{\unboundedTimeGrid} (s,0)^2 \Big)
	\end{equation}
	Now for any non-negative $A\in L^p$ with $p>4$ we use Hölder's inequality to compute
	\begin{equation}
		\E \big[ \exp(C\StoppDiscrete) \StoppDiscrete^k A^2 \big]^2 \leq \E \big[ \exp (2C\StoppDiscrete) \big] \E \big[ \StoppDiscrete^{2k} A^4 \big] \leq \E \big[ \exp (2C\StoppDiscrete) \big] \E \big[ \StoppDiscrete^{2k\Tilde{q}} \big]^\frac{1}{\Tilde{q}} \E \big[ A^p \big]^\frac{1}{q}
	\end{equation}
	for $q\defined\tfrac{p}{4}$, $\Tilde{q}\defined \tfrac{q}{q-1}$ and any $k\in\nat$. With Assumptions \hyperref[exp:assumption]{(R-q)}, \hyperref[stepsize:assumption]{(h)}, \hyperref[discrete_f:assumption]{($\unboundedTimeGrid$-F)} and \hyperref[discrete_xi:assumption]{($\unboundedTimeGrid$-T)} together with Lemma~\ref{lemma:A_exponentialMoment_ForAll_PowerMoments_SUPREMUM_VERSION}   
    we conclude all expectations in \eqref{eqn:ProofRegInterpol_EQ} are finite. Thus 
    $ \E [ \int_{0}^{\infty} \big\| \Zc_s^{\phi(s)} \big\|^2 \de s ] < \infty. $
\end{proof}
\section{Exponential Moments of Discrete Terminal Time}\label{sec.ExponentialMoments}

We discuss a possibility to ensure the \textit{stepsize-uniform} existence positive exponential moment as required by Assumption~\hyperref[exp:assumption]{(R-q)}. 
Let $D>0$ and $\alpha\in[0,\tfrac{1}{2})$ denote given parameters and consider the first exit of the Browninan motion from the $\infty$-norm ball with radius $D-\alpha$, i.e.
\begin{equation}\label{def:StoppDiam}
    \StoppDiam \defined \inf\big\{ t>0 \ \big|\ t\in\unboundedTimeGrid,  W_{t} \notin \mathbb{B}(D - h^\alpha) \big\} 
\end{equation}
This may be regarded as a guarantee that the Brownian paths remain bounded along $\unboundedTimeGrid$. In view of the BSDE approximation scheme \eqref{def:FixpointScheme} as discrete terminal time we then set $\StoppDiscrete\wedge\StoppDiam$.\footnote{
    Particularly $\StoppDiscrete = \infty$ is admissible whenever there is no suitable guess for a discrete-time approximation. 
}

We establish $L^1$-boundedness of $\{\exp (m \StoppDiam )\}_{h\in(0,h_0]}$ for $m$ strictly smaller than a constant only depending on $D>0$. To resolve this we use Freidlin-type bounds see \cite[Chapter III]{Freidlin1985} and \cite{BGG_forward_exit} for insightful utilization of this machinery. 
We start with \cite[Theorem 1.7 and Lemma 1.3]{Kazamaki1994}.
\begin{lemma}[$1$-dimensional]\label{lemma:ContinuousExpMoment_1D}
    Let $\tau_{a,b}$ denote the first exit of a Brownian motion from $(-a,b)$ for $a,b>0$ then
    $ \exp ( \tfrac{1}{2} m^2\ \tau_{a,b} ) \in L^1 $ for all $m<\tfrac{\pi}{a+b}\defined m_0$ and $\exp ( \tfrac{1}{2} m_0^2\ \tau_{a,b} ) \notin L^1. $ \close
\end{lemma}
The existence of exponential moments extends to the multi-dimensional case. Throughout this section we set
$\tau \defined \inf \{ t\geq 0\ |\ W_t \notin\mathbb{B} (D) \}$ and have the following immediate corollary.
\begin{corollary}[$d$-dimensional]\label{corollary:ContinuousExpMoment_DD}
    For the first exit of the $d$-dimensional Brownian motion from $\mathbb{B}(D)$ it holds that 
    $ \exp(m\tau)\in L^1$ for all $m < \tfrac{\pi^2}{8}\tfrac{d}{D^2} $. \close
\end{corollary}
\begin{proof}
    Let $\tau_{D,D}$ denote the first exit of a $1$-dimensional Brownian motion from $(-D,D)$. Then $\prob [\tau \geq \lambda] = \prob [ \tau_{D,D} \geq \lambda ]^d$ for any $\lambda\geq 0$. Fix $m>0$ and with the Markov inequality we have
    \begin{multline}
        \E \big[ \exp (m\tau) \big] 
        = \int_{0}^{\infty} \prob \big[ \tau \geq \tfrac{1}{m} \ln (\lambda) \big] \de \lambda 
        = \int_{0}^{\infty} \prob \big[ \tau_{D,D} \geq \tfrac{1}{m} \ln (\lambda) \big]^d \de \lambda \\
        = \int_{0}^{\infty} \prob \big[ \exp(m\tau_{D,D}) \geq \lambda \big]^d \de \lambda 
        \leq \E \big[ \exp(m\tau_{D,D}) \big]^d \int_{0}^{\infty} \exp(-dm\lambda) \de \lambda
    \end{multline}
    Thus with Jensen's inequality and the explicit value of the deterministic integral we obtain
    \begin{equation}
        \E \big[ \exp (\tfrac{m}{d}\tau) \big] \leq \E \big[ \exp(m\tau) \big]^\frac{1}{d} \leq \tfrac{1}{dm} \E \big[ \exp (m\tau_{D,D}) \big]
    \end{equation}
    and the upper bound is finite whenever $m<\tfrac{\pi^2}{8D^2}$ due to Lemma~\ref{lemma:ContinuousExpMoment_1D} (with $a=b=D$). 
\end{proof}
The main result of this section is Proposition~\ref{prop:StoppDiam_expMoments}. Recall the dependency $\StoppDiam=\StoppDiam (h)$ through the discrete-time grid $\unboundedTimeGrid$ and the radius $D-h^\alpha$. We require the following assumption.
\begin{stepsize_condition_exp}\label{stepsize_exp:assumption}
    We consider $h\in(0,h_0]$ with maximal stepsize $h_0 < \tfrac{1}{2} \tfrac{D^2}{d}$
    where we further impose that the Gaussian tail satisfies 
    $ \prob [ N \notin \mathbb{B} ( h_0^{\alpha-\frac{1}{2}} ) ] \leq \tfrac{2}{5}$ for $N\sim \mathcal{N} (0,1)$. \close
\end{stepsize_condition_exp}
The proof is via a geometric series limit, this is enabled due to a deterministic bound of the $\F_{\tau}$-conditional $p$-moments of the discrete-time delay $\StoppDiam\vee\tau-\tau$ (see Lemma~\ref{lemma:FreidlinTypeBound}). 
\begin{proposition}\label{prop:StoppDiam_expMoments}
    Assume \hyperref[stepsize_exp:assumption]{(h')}. Then
    \begin{equation}
        \sup_{h\in(0,h_0]} \exp (m \StoppDiam ) \in L^ 1\quad \text{for all } m<\tfrac{d}{8D^2}. \closeEqn
    \end{equation}
\end{proposition}
\begin{proof}
    Observe that with monotone convergence for any constant $m>0$ we have
    \begin{align}\label{eqn:proofStoppDiamExpMoment_1}
        \E \big[\exp(m\StoppDiam)\big] 
        & \leq \E \Big[ \E \big[ \exp(m (\StoppDiam\vee\tau-\tau)\ \big|\ \F_{\tau} \big] \ \exp(m\tau) \Big]\notag \\
        & = \E \Big[\ \sum_{p=0}^{\infty} \tfrac{1}{p!} m^p \E \big[ (\StoppDiam\vee\tau - \tau)^p\ \big|\ \F_{\tau} \big]\ \exp(m\tau) \Big].
    \end{align}
    From Lemma~\ref{lemma:FreidlinTypeBound} (for $t=0$) we know that
    \begin{equation}
        \sup_{h\in (0,h_0]} \E \Big[ \big(\StoppDiam\vee\tau - \tau\big)^p\ \Big|\ \F_{\tau} \Big] \leq p!\ \big(8\tfrac{D^2}{d}\big)^p\quad\text{for all } p\in\nat_0.
    \end{equation}
    Let $m < \tfrac{d}{8D^2}$ be fixed. Then we obtain the geometric limit 
    \begin{equation}
        \sup_{h\in (0,h_0]}\ \sum_{p=0}^{\infty} \tfrac{1}{p!} m^p \E \big[ (\StoppDiam\vee\tau - \tau)^p\ \big|\ \F_{\tau} \big] \leq \sum_{p=0}^{\infty} \big(m\ \tfrac{8D^2}{d} \big)^p = \frac{1}{1-m\ \tfrac{8D^2}{d}} \defined C_m < \infty.
    \end{equation}
    From Corollary~\ref{corollary:ContinuousExpMoment_DD} we know that $\exp(m\tau)\in L^1$ (since by choice $m < \tfrac{\pi^2}{8}\tfrac{d}{D^2}$). Combine this with the geometric bound to further estimate 
    \eqref{eqn:proofStoppDiamExpMoment_1} and obtain the assertion;
    \[
        \E \Big[ \sup_{h\in (0,h_0]}\ \exp(m\StoppDiam)\Big] \leq C_m \E \big[ \exp(m\tau) \big] < \infty . \qedhere
    \]
\end{proof}
We conclude with the conditional bound on the $p$-moments. Notice in the proof of Proposition~\ref{prop:StoppDiam_expMoments} we only require $t=0$. Nonetheless we derive the bound uniformly on $\{t\vee \tau \ |\ t\geq 0 \}$, which we encountered as the minimal family to establish the result. The overall idea is inspired by \cite[Section 2.3]{BGG_forward_exit} where in a general setting for finite stopping times a similar bound is established.

\textit{Sketch of idea}: For $p=1$ we use the strong Markov property of the Brownian motion to exploit the gap $\mathbb{B}(D-h^{\alpha})\subset \mathbb{B}(D)$ between the domains triggering the exit times, and a conditional version of a \textit{good}-$\lambda$-\textit{inequality}, see e.g. \cite{Jacka1989GoodLambda}. Via induction using a Freidlin-type inequality, see e.g. \cite[Chapter 3]{Freidlin1985}, this is extended for $p\in\nat$.
\begin{lemma}\label{lemma:FreidlinTypeBound}
    Assume \hyperref[stepsize_exp:assumption]{(h')}. Then for each $p\in\nat_0$ the $p$-moments satisfy
    \begin{equation}
        \sup_{t\geq 0,\ h\in (0,h_0]}\ \E \Big[ \big((\StoppDiam\vee t\vee\tau) - (t\vee\tau)\big)^p\ \Big|\ \F_{t\vee\tau} \Big] \leq p!\ \big(8\tfrac{D^2}{d}\big)^p. \closeEqn
    \end{equation}
\end{lemma}
We introduce notation to express the first time on $\unboundedTimeGrid$ after a stopping time $\Stopp$ as well as for the first continuous exit from $\mathbb{B}(D)$, i.e. 
\begin{equation}\label{def:+_operator_AND_next_cont_exit}
    \Stopp^+ \defined \inf \big\{ t\in\unboundedTimeGrid \ \big|\ \Stopp \leq t \big\} \quad\text{and}\quad \zeta (\Stopp) \defined  \inf \big\{ t\geq \Stopp \ \big|\ W_t \notin \mathbb{B} (D) \big\}.
\end{equation}

\begin{proof}[Proof of Lemma~\ref{lemma:FreidlinTypeBound}]
    The proof is via induction. 

    \step{1} We show that
    \begin{equation}\label{eqn:ConditionalFirstMomentBound}
        \E \big[ (\StoppDiam\vee t\vee\tau)-(t\vee\tau) \ \Big|\ \F_{t\vee\tau} \big] \leq 8\tfrac{D^2}{d} \quad \text{for all } t\geq 0, h\in (0,h_0].
    \end{equation}

    Let $t\geq 0$ and $h\in(0,h_0]$ be fixed. We consider $\zeta (t\vee\tau)$ the first exit after $t\vee\tau$. Observe\footnote{
        By definition (see \eqref{def:+_operator_AND_next_cont_exit}) it holds that $s^+ - s =0, s\in\unboundedTimeGrid$. Thus $s^+ < s + h$ for all $s\geq 0$.
    }
    \begin{equation}
        \{\StoppDiam - \zeta (t\vee\tau) \geq h\} \subseteq \{\StoppDiam - \zeta (t\vee\tau)^+ > 0\} \subseteq \{ W_{\zeta(t\vee\tau)^+}\in\mathbb{B}(D-h^\alpha) \}.
    \end{equation}
    With the strong Markov property of the Brownian motion we have
    \begin{align}\label{def:varphi_h}
        \prob \big[ \StoppDiam - \zeta (t\vee\tau) \geq h \ \big|\ \F_{\zeta (t\vee\tau)} \big] & \leq \prob \big[ x + \widetilde{W}_s \in\mathbb{B}(D-h^\alpha) \big]\Big|_{s=\zeta (t\vee\tau)^+-\zeta (t\vee\tau), x = W_{\zeta (t\vee\tau)}} \\
        & \leq \prob \big[ \widetilde{W}_h \notin \mathbb{B} (h^\alpha)  \big] \defined \varphi (h).
    \end{align}
    Recall $\alpha <\tfrac{1}{2}$ and notice $\varphi (h) \leq \varphi (h_0)$. Let $c\geq 0$ and use the strong Markov property; 
    \begin{align}
        \prob \big[ \StoppDiam - \zeta (t\vee\tau)^+ \geq c \big|  \F_{\zeta (t\vee\tau)^+} \big] 
        & = \prob \Big[ \inf \big\{  t\in\unboundedTimeGrid  \big|  x + \widetilde{W}_t \notin \mathbb{B} (D-h^\alpha) \big\} \geq c \Big]\bigg|_{x = W_{\zeta(t\vee\tau)^+} } \\
        & \leq \sup_{y\in\R^d} \prob \Big[ \inf \big\{ t\in\unboundedTimeGrid \big| y + \widetilde{W}_t \notin \mathbb{B} (D-h^\alpha) \big\} \geq c \Big]\\
        & = \prob \Big[ \inf \big\{  t\in\unboundedTimeGrid \big| \widetilde{W}_t \notin \mathbb{B} (D-h^\alpha) \big\} \geq c \Big] \\
        & = \prob \big[ \StoppDiam \geq c \big]
    \end{align}
    Combine both strong Markov property observations to conclude that
    \begin{align}
        & \prob \big[ \StoppDiam - \zeta (t\vee\tau) \geq c + h \ \big|\ \F_{\zeta (t\vee\tau)} \big] \\ 
        & \hspace*{1.0cm} \leq \E \Big[ \1_{ \{ W_{ \zeta(\tau\vee\tau)^+} \in\mathbb{B}(D-h^\alpha) \} }
        \prob \big[ \StoppDiam - \zeta (t\vee\tau)^+ \geq c \big|  \F_{\zeta (t\vee\tau)^+} \big]
        \ \Big| \ \F_{\zeta(t\vee\tau)} \Big] \\
        & \hspace*{1.0cm} \leq \prob \big[ W_{ \zeta(\tau\vee\tau)^+} \in\mathbb{B}(D-h^\alpha) \ \big| \ \F_{\zeta(t\vee\tau)} \big]  \ \prob \big[ \StoppDiam \geq c \big] \\
        & \hspace*{1.0cm} \leq \varphi (h_0) \prob \big[ \StoppDiam \geq c \big]
    \end{align}
    For any $\lambda\geq 0$ we have that
    \begin{align}
        & \prob \big[ \StoppDiam-(t\vee\tau) \geq \lambda\ \big|\ \F_{t\vee\tau} \big] \\
        & \hspace*{1.0cm} \leq \prob \big[ \StoppDiam-\zeta(t\vee\tau) \geq \tfrac{\lambda}{2} + h \big|\ \F_{t\vee\tau} \big] 
        + \prob \big[ \zeta(t\vee\tau)-t\vee\tau \geq \tfrac{\lambda}{2} - h \big|\ \F_{t\vee\tau} \big] \\
        & \hspace*{1.0cm} \leq \varphi(h_0) \prob \Big[ 2 \StoppDiam \geq \lambda \Big] 
        + \prob \Big[ 2\big( \zeta(t\vee\tau)-t\vee\tau + h\big) \geq \lambda\ \Big|\ \F_{t\vee\tau} \Big] .
    \end{align}
    Using an elementary expansion (see  Lemma~\ref{lemma:RepresentationConditionalExpectation}) and 
    the fact that $( \| W_t\|^2 - dt )_{t\geq 0}$ is a martingale which is bounded on $[t\vee\tau, \zeta(t\vee\tau)]$ we conclude
    \begin{align}\label{eqn:LemmaConditionalFirstMomentBound_1}
        \E \big[ (\StoppDiam\vee t\vee\tau)-(t\vee\tau) \ \Big|\ \F_{t\vee\tau} \big] & = \int_{0}^{\infty} \prob \big[ \StoppDiam-(t\vee\tau) \geq \lambda\ \big|\ \F_{t\vee\tau} \big] \de \lambda \notag \\
        & \leq 2 \varphi (h_0) \E \big[ \StoppDiam \big] + 2 \E \big[ \zeta (t\vee\tau) - t\vee\tau + h \ \big|\ \F_{t\vee\tau}\big] \notag \\
        & = 2 \varphi (h_0) \E \big[ \StoppDiam \big] + \tfrac{2}{d} \E \Big[ \|W_{\zeta (t\vee\tau)}\|^2 - \|W_{t\vee\tau}\|^2 \big|\ \F_{t\vee\tau}\Big] + 2 h \notag \\
        & = 2 \varphi (h_0) \E \big[ \StoppDiam \big] + 2 \tfrac{D^2}{d} - 2 \tfrac{1}{d} \|W_{t\vee\tau}\|^2 + 2h
    \end{align}
    Particularly for $t = 0$ (using that $\tau$ is an exit time) we obtain
    \begin{align}
        \E \big[\StoppDiam\big] \leq \E \big[ (\StoppDiam\vee\tau)-\tau \big] + \E \big[ \tau \big] \leq 2 \varphi (h_0) \E \big[ \StoppDiam \big] + 2h + \tfrac{D^2}{d}.
    \end{align}
    Rearrange and further bound \eqref{eqn:LemmaConditionalFirstMomentBound_1};
    \begin{align}
        \E \big[ (\StoppDiam\vee t\vee\tau)-(t\vee\tau) \ \Big|\ \F_{t\vee\tau} \big] \leq 2 \varphi (h_0) \tfrac{1}{1-2\varphi(h_0)} \big( \tfrac{D^2}{d} + 2h_0 \big) + 2 \tfrac{D^2}{d} - W_{t\vee\tau}^2 + 2h_0
    \end{align}
    We conclude the bound \eqref{eqn:ConditionalFirstMomentBound} since by Assumption~\hyperref[stepsize_exp:assumption]{(h')} on the stepsize\footnote{
        Assumption \hyperref[stepsize_exp:assumption]{(h')} ensures that $2h_0\leq \tfrac{D^2}{d}$ and $\varphi(h_0) = \prob [N\notin\mathbb{B}(h_0^{\alpha-\frac{1}{2}})]\leq \tfrac{2}{5}$ where $N\sim\mathcal{N}(0,1)$.
    }  
    \begin{align}
        2 \varphi (h_0) \tfrac{1}{1-2\varphi(h_0)} \big( \tfrac{D^2}{d} + 2h_0 \big) + 2 \tfrac{D^2}{d} + 2h_0 \leq \Big( 1 + \tfrac{\varphi (h_0)}{1-2\varphi(h_0)} \Big)\ 2\big( \tfrac{D^2}{d} + 2h_0 \big) \leq 8\tfrac{D^2}{d} \closeEqn
    \end{align}
    
    \step{2} Whenever the claim is satisfied for $p\in\nat$ then also for $p+1$. 
    
    Let $t\geq 0$ and $h\in(0,h_0]$ be fixed. Using a conditional version of Fubini's theorem we compute
    \begin{multline}
        \tfrac{1}{p+1} \E \Big[ \big(\StoppDiam\vee t\vee\tau - t\vee\tau \big)^{p+1}\ \Big|\ \F_{t\vee\tau} \Big] = \E \Big[ \int_{0}^{\infty} \1_{\{\StoppDiam > s \geq t\vee\tau\}} \big( \StoppDiam\vee t\vee\tau - s)^p\de s\ \Big|\ \F_{t\vee\tau} \Big] \\
        = \int_{0}^{\infty} \E \Big[  \1_{\{\StoppDiam > s \geq t\vee\tau\}} \E \big[ ( \StoppDiam\vee s - s)^p\ \big|\ \F_{s} \big] \ \Big|\ \F_{t\vee\tau} \Big] \de s .
    \end{multline} 
    With an elementary property\footnote{
     	Let $\Stopp_1, \Stopp_2$ be a.s. finite $(\Omega, \F, (\F_t)_t)$-stopping times. Then 
     	\begin{equation}
     		\E \big[ 1_{\{\Stopp_1<\Stopp_2\} }X \ \big|\ \F_{\Stopp_1\vee\Stopp_2} \big] = \E \big[ 1_{\{\Stopp_1<\Stopp_2 \} }X\ \big|\ \F_{\Stopp_2} \big]\quad\text{for all } X\in L^2 (\Omega,\F,\prob).
     	\end{equation}
        Let $X\in L^2 (\Omega,\F,\prob)$. Set $A\defined\{\Stopp_1<\Stopp_2\}$ and define $\de\Q\defined\1_A\de\prob$. Then $\1_A X\in L^2(\Omega\cap A,\F\cap A,\Q)$ with the trace sigma algebra $\F\cap A\defined\{A\cap B\ |\ B\in\F\}$. 
        One checks that $\F_{\Stopp_1\vee\Stopp_2}\cap A = \F_{\Stopp_2}\cap A$. Consequently $ \E^\prob[\1_A X\ |\ \F_{\Stopp_1\vee\Stopp_2}] \in L^2(\F_{\Stopp_1\vee\Stopp_2}\cap A,\Q)=L^2(\F_{\Stopp_2}\cap A,\Q)$ 
        and thus $\E^\prob[ \1_A X\ |\ \F_{\Stopp_1\vee\Stopp_2}] \in L^2(\Omega, \F_{\Stopp_2},\prob)$. \label{footnote:ConditionalExpectationMaxStoppTimes}
    }
    we can apply the bound for $p$ on $\{s\geq t\vee\tau\}$ to obtain 
    \begin{align}
        \1_{\{s\geq t\vee\tau\}} \E \big[ ( \StoppDiam\vee s - s)^p\ \big|\ \F_{s} \big] = \1_{\{s\geq t\vee\tau\}} \E \big[ ( \StoppDiam\vee s\vee\tau - s\vee\tau)^p\ \big|\ \F_{s\vee\tau} \big] \leq p!\ \big(8\tfrac{D^2}{d}\big)^p
    \end{align}
    which together with \eqref{eqn:ConditionalFirstMomentBound} and Lemma~\ref{lemma:RepresentationConditionalExpectation} allows us to further estimate the $(p+1)$-th moment;
    \begin{align}
        \E \Big[ \big(\StoppDiam\vee t\vee\tau - t\vee\tau \big)^{p+1}\ \Big|\ \F_{\tau} \Big] & \leq (p+1)!\ \big(8\tfrac{D^2}{d}\big)^p \  \int_{0}^{\infty} \prob \big[ \StoppDiam > s \geq t\vee\tau\ \big|\ \F_{\tau} \big] \de s \\
        & = (p+1)!\ \big(8\tfrac{D^2}{d}\big)^p \ \E \big[ \StoppDiam\vee t\vee\tau - t\vee\tau \big|\ \F_{\tau} \big] \\
        & \leq (p+1)!\ \big(8\tfrac{D^2}{d}\big)^{p+1}. \qedhere
    \end{align}
\end{proof}

\section{Application to decoupled Markovian FBSDE}\label{sec:Application_EM}

Let $\domain\subset\R^d$ be a bounded domain. Fix $x\in\domain$ and consider the decoupled FBSDE
\begin{align}\label{eqnFBSDE}
	X_t^{x} & = x + \int_0^t \mu(X_s^{x})\de s + \int_{0}^t \sigma (X_s^{x}) \de W_s  \notag \\
	Y_t^{x} & = g(X_\tau^{x}) + \int_{t\wedge\tau}^{\tau} f(X_s^{x},Y_s^{x}) \de s - \int_{t\wedge\tau}^{\tau} Z_s^{x} \de W_s \qquad t \geq 0
\end{align}
where $\tau^{x} \defined \inf \{ t\geq 0 \, |\, X_t^{x} \notin \domain \}$. The forward equation is approximated with a classical Euler-Maruyama scheme: For $\X_{0} \defined x$ consider (subsequently the starting point is omitted);
\begin{equation}\label{def:EulerMaruyamaSequence}
	\X_{t+h} \defined \X_{t} + \mu\big( \X_{t} \big) h + \sigma (\X_{t}) (W_{t+h}- W_{t}) \qquad t\in\unboundedTimeGrid
\end{equation}
and set $\StoppDiscrete \defined \inf \{ t\in\pi \,|\, \X_t \notin \domain \}$. We consider the discrete-time approximation of the backward equation in \eqref{eqnFBSDE} (following the general setting \eqref{def:FixpointScheme}) as the unique solution (in the sense of Section~\ref{sec:Existence}) of the fixed-point scheme
\begin{equation}\label{def:FBSDE_Disretization}
	\Y_{t} \defined \1_{\{t \geq \StoppDiscrete\}} g(\X_{\StoppDiscrete}) + \1_{\{t < \StoppDiscrete \}} \big( \E_{t} [\Y_{t+h}] + h f (\X_{t}, \Y_{t}) \big) \quad t\in\unboundedTimeGrid 
\end{equation}
\begin{sde_assumption}\label{sde:assumption}
    The maps $\mu\colon \R^d \to \R^d$ and $\sigma\colon\R^d \to \R^{d,d}$ are Lipschitz continuous (with constants $\lipschitz{\mu}$ resp. $\lipschitz{\sigma}$) and have compact support.
    \close
\end{sde_assumption}
\begin{fbsde_BSDE_assumption}\label{fbsde_bsde:assumption}
    Let $f\colon\R^d\times\R\to\R$ be jointly Lipschitz with constant $\lipschitz{f}>0$ and assume 
    $\sup_{x\in\R^d} |f(x,0)| < \infty$.
    Moreover $g\colon\R^d\to\R$ is Lipschitz with constant $\lipschitz{g}>0$ and bounded. \close
\end{fbsde_BSDE_assumption}
Throughout this section we consider the Euler-Maruyama scheme \eqref{def:EulerMaruyamaSequence} for stepsizes $h\in (0,h_0]$ where we consistently assume $h_0$ satisfies \hyperref[stepsize:assumption]{(h)}.
Similarly to \hyperref[exp:assumption]{(R-q)} in Section~\ref{sec:BSDE_ErrorBoundMain} we postulate an assumption on the existence of positive exponential moments.
\begin{domain_ratio_assumption}\label{domain_ratio:assumption}
    Let $q_1 \geq 2$ and $q_2 \geq 2$ be given parameters. Assume there is $\rho>0$ such that
    \begin{equation}
        \rho > \max \big\{ q_1d(6\lipschitz{\mu}+3q_1\lipschitz{\sigma}^2) , 4q_2 \lipschitz{f} \big\} 
        \quad\text{and}\quad 
        \exp(\rho\tau) \in L^1
        \ \ \text{and}\ \
        \sup_{ h\in(0,h_0] } \exp ( \rho \StoppDiscrete ) \in L^1 \closeEqn
    \end{equation}
\end{domain_ratio_assumption}
Within the next result we collect well-know results from literature and furthermore connect the assumptions of the present section to Section~\ref{sec:BSDE_ErrorBoundMain}.
\begin{proposition}
    Assume \hyperref[sde:assumption]{(SDE)}, \hyperref[fbsde_bsde:assumption]{(BSDE)}, \hyperref[stepsize:assumption]{(h)} and that \hyperref[domain_ratio:assumption]{($\domain$-R)} holds with $q_1,q_2\geq 2$. Then the FBSDE has a unique solution $(X,Y,Z)$ in 
    $\mathcal{S}^2 (\tau) \times \mathcal{M}^2 (\tau,\R) \times \mathcal{M}^2 (\tau,\R^d)$.\footnote{
        With notation 
        $\mathcal{S}^p (\tau) \defined \{ X\ |\ \E[ \sup_{s\in [0,\tau]} |X_s|^p ] < \infty \}$ 
        and 
        $\mathcal{M}^2 (\tau,\R^d) \defined \{ X\ |\ \E [ \int_0^\tau \exp(\rho s) \|X\|^2 \de s ] < \infty \}$
    }
    Furthermore the Assumptions \hyperref[bsde:assumption]{(F)}, \hyperref[discrete_f:assumption]{($\unboundedTimeGrid$-F)}, \hyperref[discrete_xi:assumption]{($\unboundedTimeGrid$-T)} hold for all $p\geq 2$ and \hyperref[exp:assumption]{(R-q)} holds with $q=q_2$.
    \close
\end{proposition}
\begin{proof}
    Regarding the solution $X$ of the SDE we refer to e.g. \cite[Theorem 5.2.5]{KaratzasShreve} and for the solution of the BSDE see e.g. \cite[Theorem 3.4]{PardouxRandomTimeBSDE}.
    Notice since $x\mapsto f(x,0)$ is bounded \hyperref[bsde:assumption]{(F)} holds for all $p\geq 2$ if we identify the generator of Section~\ref{sec:BSDE_ErrorBoundMain} with the map $(t, Y_t) \mapsto f(X_t, Y_t)$. Similarly we identify $ \fd{\unboundedTimeGrid} (t, \cdot) \defined f(\X_t, \cdot) $  to conclude \hyperref[discrete_f:assumption]{($\unboundedTimeGrid$-F)} is satisfied, and \hyperref[discrete_xi:assumption]{($\unboundedTimeGrid$-T)} holds as desired since $g$ is bounded. Finally by choice \hyperref[exp:assumption]{(R-q)} holds with $q=q_2$.
\end{proof}
Our main result Theorem~\ref{thm:BSDE_MainResult} yields the following error bound.
\begin{theorem}\label{thm:FBSDE_ErrorBound}
    Assume \hyperref[sde:assumption]{(SDE)}, \hyperref[fbsde_bsde:assumption]{(BSDE)}, \hyperref[stepsize:assumption]{(h)} and \hyperref[domain_ratio:assumption]{($\domain$-R)} with $q_1 = 64 $ and $q_2 = 8$. 
    Assume \eqref{eqnFBSDE} has a solution $(Y,Z) \in \mathcal{S}^8\times\mathcal{H}^8$. If the process $Y$ satisfies Komogorov's tightness criterion with a constant rate $\tfrac{1}{2}$ for all $p\geq 2$ (see \eqref{def:KolmogorovCriterionRatio}).\footnote{
        E.g. this is the case if $\domain$ is bounded and its boundary $\partial\domain$ is of class $\HoelderSpace{2}{\gamma}$ for some $\gamma\in(0,1)$, i.e.\ $\partial\domain$ admits a representation via maps of class $\HoelderSpace{2}{\gamma}$.
        Furthermore $\mu\colon\R^d\rightarrow\R^d$ is of class $\C^1$, $\sigma\colon\R^d\rightarrow\R^{d\times d}$ is of class $\C^2$, and $\sigma\sigma^\top$ is uniformly elliptic on $\domain$, i.e. the matrix $\sigma\sigma^\top (x)$ is positive definite for each $x\in\domain$ and the eigenvalues are uniformly bounded away from zero; see \cite[p.31]{GilbargTrudinger}.
        Moreover $f\colon\R^d\times\R\rightarrow\R$ is of class $\C^{1,2}$ and there is a constant $C \geq 0$ such that
        $\sign (u) f(x,u) \leq C, (x,u)\in\domain\times\R.$
        Finally $g\colon\R^d\rightarrow\R^d$ is of class $\HoelderSpace{2}{\gamma}(\overline{\domain})$ for some $\gamma\in (0,1)$. Then the associated elliptic Dirichlet problem has a unique solution $u\in\HoelderSpace{2}{\gamma}$, see e.g. \cite{GilbargTrudinger} or \cite[Proposition 3.1]{Schlegel2024ProbabilisticShape}. Now set $Y \defined u(X)$ and $Z\defined \sigma(X)^\top \Diff u(X)$.
    }
    Then there is a constant $C>0$ such that
    \begin{equation}
        \E \Big[ \sup_{\ell\in\unboundedTimeGrid} \sup_{t\in[\ell,\ell+h]} \big| Y_t - \Y_{\ell} \big|^2 \Big]
        \leq C \Big( h^{\frac{1}{8}} 
        + \E \big[\ |\tau-\StoppDiscrete|^{8} \big]^\frac{1}{8} 
        + \E \big[\ |\tau-\StoppDiscrete|^{12} \big]^\frac{1}{8}
        + \E \big[\ |\tau-\StoppDiscrete|^{4} \big]^\frac{1}{8} \Big) \closeEqn
    \end{equation}
\end{theorem}
This is a direct consequence of Theorem~\ref{thm:MarkovianFBSDE_StrongBound} for $\varepsilon = \frac{1}{4}$; the result is located at the end of this section. Beforehand are several remarks motivating our choice of presentation and to some extend outlining possible directions for future research.
\begin{remark}[Lipschitz Generator]
    In Literature see e.g. \cite{briand2003lp, BriandHuHomogenization, PardouxRandomTimeBSDE, Peng1991, royer2004bsdes} the Lipschitz assumption on the generator $f$ has been weakened to a one-sided monotonicity condition of the form
    \begin{equation}
        (y_1 - y_2) \big( f(x,y_1) - f(x,y_2) \big) \leq -c |y_1-y_2|^2 \quad\text{for all } y_1,y_2 \in \R, x\in\R^d
    \end{equation}
    for some $c\in\R$.
    To ensure the existence of a unique solution of the BSDE (among other conditions) there has to be $\rho > -2c$ satisfying $\E [ \exp(\rho \tau) g(X_\tau) ] < \infty $. With regards to possible applications the case $c\geq 0$ is obviously of greater interest.
    
    Nevertheless in our studies we encountered that in order to ensure existence of a solution of the discrete-time BSDE approximation \eqref{def:FixpointScheme} the Lipschitz property is unavoidable in Lemma~\ref{lemma:T_Contraction}. We further employ the Lipschitz continuity at several other steps in this manuscript where it is possible to weaken the requirement to a one-sided monotonicity condition. \close
\end{remark}
\begin{remark}[Temporal Cut-Off]\label{remark:Temporal_CutOff}
    Combining the observations from Section~\ref{sec.ExponentialMoments} about the existence of positive exponential moments with the overall framework, which is presented under the Assumptions \hyperref[exp:assumption]{(R-q)} resp. \hyperref[domain_ratio:assumption]{($\domain$-R)}, a sufficient condition is the following ratio between the Lipschitz constants and the radius of the ball for the discrete-time Brownian motion;
    \begin{equation}\label{def:LipschitzDomainRatio_}
        \max \big\{ q_1d(6\lipschitz{\mu}+3q_1\lipschitz{\sigma}^2) , 4q_2 \lipschitz{f} \big\} < \tfrac{1}{8} \tfrac{d}{D^2}
    \end{equation}
    For given parameters $D, \alpha$ see \eqref{def:StoppDiam} we set $\DiscreteTerminal \defined \StoppDiscrete \wedge \StoppDiam$ and assume \hyperref[sde:assumption]{(SDE)}, \hyperref[fbsde_bsde:assumption]{(BSDE)}, \hyperref[stepsize:assumption]{(h)} and \hyperref[stepsize_exp:assumption]{(h')}. Then \eqref{def:LipschitzDomainRatio_} is sufficient to ensure that the uniform discrete-time positive exponential moments as assumed in \hyperref[domain_ratio:assumption]{($\domain$-R)} exist with the same parameters $q_1, q_2$. Particularly we require $(q_1,q_2) = (64,8)$ for the classic error bound as presented in Theorem~\ref{thm:MarkovianFBSDE_ClassicBound} (see also Theorem~\ref{thm:BSDE_MainResult_no_SUPREMUM}) and $(q_1,q_2) = (32,4)$ for the strong error bound as presented in Theorem~\ref{thm:FBSDE_ErrorBound} (see also Theorem~\ref{thm:BSDE_MainResult} and Theorem~\ref{thm:MarkovianFBSDE_StrongBound}).

    In general (especially in the case $\mu\neq 0$) there is no finite diameter $D>0$ such that $\StoppDiscrete \leq \StoppDiam$ and thus this sufficient approach has no general assurance to provide a terminal condition $g(\X_{\DiscreteTerminal})$ where $\X_{\DiscreteTerminal} \notin \domain$. Nonetheless using global Hölder coefficient bounds we control the spatial deviation by the $L^p$-distances of the terminal conditions yielding the error bound as presented in Theorem~\ref{thm:FBSDE_ErrorBound} which is merely depending on a time difference. With the usual elementary bound we have
    \begin{equation}\label{eqn:SplittTheBoundIntoExitDistances}
        \E \big[\ |\tau-\DiscreteTerminal|^{p}\ \big] \leq 2^{p-1} \E \big[\ |\tau-\StoppDiscrete|^{p}\ \big] + 2^{p-1} \E \big[\ |\StoppDiscrete-\StoppDiam|^{p}\ \big] 
    \end{equation}
    where the rate of the first summand in the bound is addressed within \cite[Theorem 3.9]{BGG_forward_exit}.
    On the other hand to the best of our knowledge the second summand is no yet investigated in literature and given the substantial length of this article remains for future work. \close    
\end{remark}
\begin{remark}[Uniform First Moment Bound]
    A different possibility to obtain a sufficient condition for the Assumptions \hyperref[exp:assumption]{(R-q)} resp. \hyperref[domain_ratio:assumption]{($\domain$-R)} is the generalization of Lemma~\ref{lemma:FreidlinTypeBound} from the discrete-time Brownian motion to a general Euler-Maruyama scheme. The remaining results in Section~\ref{sec.ExponentialMoments} can then be obtained analogously. This analysis is outside of the scope of this article. We mention the pioneering work \cite{BGG_forward_exit} where such considerations were already done in a considerably more general framework. The paper requires \cite[Assumption (L)]{BGG_forward_exit} (see also \cite[(3.4) on p.1653]{BGG_forward_exit} for the specific formulation regarding the Euler-Maruyama scheme); the authors assume there is a constant $L\geq 0$ such that 
    \begin{equation}\label{eqn:BGG_Assumption_L}
        \E \big[ \theta^\unboundedTimeGrid (\Stopp) - \Stopp \ \big|\ \F_{\Stopp} \big] \leq L \quad\text{for all finite stopping times } \Stopp
    \end{equation}
    where $\theta^\unboundedTimeGrid (\Stopp)$ denotes the first discrete-time exit from $\domain$ after $\Stopp$. In that case \cite[Lemma 2.8]{BGG_forward_exit} provides existence of exponential moments for scalars $\rho \in [0,L^{-1})$, i.e. $\exp (c(\theta^\unboundedTimeGrid (\Stopp) - \Stopp)) \in L^1$. There are equivalent characterizations for \eqref{eqn:BGG_Assumption_L} see \cite[Proposition 2.2 and Lemma A.4]{BGG_forward_exit} or \cite[Lemma III.3.1]{Freidlin1985}.

    Within this framework one would change the upper bound of \eqref{def:LipschitzDomainRatio_} to $L^{-1}$. Thus in each application it does crucially depend to derive smallest possible $L\geq 0$. Concatenating this with Theorem~\ref{thm:FBSDE_ErrorBound} and \cite[Theorem 3.9]{BGG_forward_exit} yields then the rate
    \begin{equation}
        \E \Big[ \sup_{\ell\in\unboundedTimeGrid} \sup_{t\in[\ell,\ell+h]} \big| Y_t - \Y_{\ell} \big|^2 \Big] \leq C h^\frac{1}{16}
    \end{equation}
    One may expect the rate to be of order $\tfrac{1}{2}$. The decay to $\tfrac{1}{16}$ is due to the absence of a result about the behavior of optimal exponents $\alpha_p >0$ such that $\E [|\tau-\StoppDiscrete|^p] \leq c_p h^{\alpha_p}$. The seminal work \cite{BGG_forward_exit} regarding general exit times provides $\alpha_p = \tfrac{1}{2}$  for all $p\in\nat$. This exponent is optimal for $p=1$, see also \cite[Theorem 3.1]{Bouchard_strong_BSDE_on_domain}. For $p>1$ the problem remains unresolved. \close 
\end{remark}
\begin{remark}[Surrogate for Forward Scheme Convergence Analysis]
    We regard our proposed scheme for discrete-time approximation of a FBSDE as a theoretical construct which should not be numerically simulated. Instead it is supposedly supportive for the convergence analysis of numerical algorithms for Markovian FBSDEs with random (unbounded) terminal times. In view of Remark~\ref{remark:Temporal_CutOff} such algorithms e.g. may be conditioned on $\{ \StoppDiscrete < \StoppDiam \}$ and circumvent the additional terms in the error bound. \close    
\end{remark}

\paragraph*{Discussion of the Bound's Application}
Subsequently we further refine the error bounds as given in Corollary\ref{corollary:MainResults_AfterHoelder}, specifically exploiting the properties of the forward diffusion and its Euler-Maruyama approximation. The bounds in Theorem~\ref{thm:MarkovianFBSDE_StrongBound} and~\ref{thm:MarkovianFBSDE_ClassicBound} correspond to both errors defined in \eqref{def:BSDE_Errors} for the case of a decoupled FBSDE. As a direct consequence we obtain the illustrative result Theorem~\ref{thm:FBSDE_ErrorBound}.

We introduce notation for a general bounding term
\begin{multline}\label{def:MarkovianSettingGeneralBound}
    \mathcal{B} (q; X,Y ) \defined
    \E \Big[ \big| g(X_\tau) - g(\X_{\StoppDiscrete}) \big|^{4q} \Big]^\frac{1}{2q} + \E \big[\ |\tau-\StoppDiscrete|^{4q} \big]^\frac{1}{4q} \\
    + \E \bigg[ \Big(\int_{0}^{\tau} \big( f(X_s,Y_s) - f(\X_{\phi(s)}, Y_s) \big)^2 \de s\Big)^{2q} \bigg]^\frac{1}{2q} \quad\text{for } q\geq 1 
\end{multline}
\begin{proposition}[{Bounding the Bounds}]\label{prop:MarkovianBoundinTheBounds}
    Let $q\geq 1$. Assume \hyperref[sde:assumption]{(SDE)}, \hyperref[fbsde_bsde:assumption]{(BSDE)}, \hyperref[stepsize:assumption]{(h)} and \hyperref[domain_ratio:assumption]{($\domain$-R)} with $q_1 = 32 q$ and $q_2 \geq 0$. 
    Then for every $\varepsilon >0$ there is $C >0$ (not depending on $h$) such that
    \begin{equation}
        \mathcal{B} (q; X,Y ) 
        \leq C h^{2(\frac{1}{2} - \varepsilon)} 
        + C \E \big[\ |\tau-\StoppDiscrete|^{4q} \big]^\frac{1}{4q} 
        + C \E \big[\ |\tau-\StoppDiscrete|^{8q (1-\varepsilon) } \big]^\frac{1}{4q}
        + C \E \big[\ |\tau-\StoppDiscrete|^{8q (\frac{1}{2}-\varepsilon) } \big]^\frac{1}{4q} \closeEqn
    \end{equation}
\end{proposition}
\begin{proof}
    Let $\varepsilon >0$. 
    With the Lipschitz continuity of $g$ and Proposition~\ref{prop:EM_Strong_DifferentStopp_Bound} (applied with $\Stopp_1 = \tau$, $\Stopp_2=\StoppDiscrete$) we obtain a constant $C_1>0$ such that
    \begin{equation}
        \E \Big[ \big|g(X_\tau) - g(\X_{\StoppDiscrete}) \big|^{4q} \Big]^{\frac{1}{2q}} 
        \leq C_1 \lipschitz{g}^{2} \Big( h^{2q} + \E \big[ |\tau - \StoppDiscrete|^{8q(1-\varepsilon)} \big]^\frac{1}{2} + \E \big[ |\tau - \StoppDiscrete|^{8q(\frac{1}{2}-\varepsilon)} \big]^\frac{1}{2} 
        \Big)^\frac{1}{2q}
    \end{equation}
    Furthermore we invoke Lemma~\ref{lemma:TheDelta_F_Bound} ($\Stopp = \tau$) and obtain a constant $C_2>0$ such that
    \[
        \E \bigg[ \Big(\int_{0}^{\tau} \big( f(X_s,Y_s) - f(\X_{\phi(s)}, Y_s) \big)^2 \de s\Big)^{2q} \bigg]^\frac{1}{2q} \leq C h^{2(\frac{1}{2} - \varepsilon)} \qedhere
    \]
\end{proof}
\begin{theorem}[Markovian Strong Error Bound]\label{thm:MarkovianFBSDE_StrongBound}
    Assume \hyperref[sde:assumption]{(SDE)}, \hyperref[fbsde_bsde:assumption]{(BSDE)}, \hyperref[stepsize:assumption]{(h)} and \hyperref[domain_ratio:assumption]{($\domain$-R)} with $q_1 = 64 $ and $q_2 = 8$. 
    Assume \eqref{eqnFBSDE} has a solution $(Y,Z) \in \mathcal{S}^8\times\mathcal{H}^8$ and furthermore for a specifically fixed $\varepsilon >0$ there are $\delta, \alpha > 1$ such that $\alpha > \tfrac{\delta}{\varepsilon} \vee 2$  and
    \begin{equation}
        \sup_{ \substack{ s,t\in[0,\infty) \\ 0 < |s-t| \leq 1} } \E \Big[ \int_{s}^{t} |f(X_r, Y_r)|^\alpha \de r \Big]
        + \sup_{ \substack{ s,t\in[0,\infty) \\ 0 < |s-t| \leq 1} } \E \Big[ \int_{s}^{t} \| Z_r \|^\alpha \de r \Big] < \infty
    \end{equation}
    then there is a constant $C>0$ (not depending on $h$) such that
    \begin{multline}
        \E \Big[ \sup_{\ell\in\unboundedTimeGrid} \sup_{t\in[\ell,\ell+h]} \big| Y_t - \Y_{\ell} \big|^2 \Big]
        \leq C h^{2 (\frac{1}{2} - \frac{1}{\alpha} - \varepsilon)} \\
        + C \E \big[\ |\tau-\StoppDiscrete|^{8} \big]^\frac{1}{8} 
        + C \E \big[\ |\tau-\StoppDiscrete|^{16 (1-\varepsilon) } \big]^\frac{1}{8}
        + C \E \big[\ |\tau-\StoppDiscrete|^{16 (\frac{1}{2}-\varepsilon) } \big]^\frac{1}{8} \closeEqn
    \end{multline}
\end{theorem}
\begin{proof}
    From the second part of Corollary~\ref{corollary:MainResults_AfterHoelder} we obtain $K>0$ such that 
    \begin{equation}
        \E \Big[ \sup_{\ell\in\unboundedTimeGrid} \sup_{t\in[\ell,\ell+h]} \big| Y_t - \Y_{\ell} \big|^2 \Big] 
        \leq 2 \E \Big[ \sup_{\ell\in\unboundedTimeGrid} \sup_{t\in[\ell,\ell+h]} \big| Y_t - Y_{\ell} \big|^2 \Big] 
        + 2C_1 h + 2C_1 \mathcal{B} (2; X, Y)
    \end{equation}
    Then the claim is a consequence of Proposition~\ref{prop:MarkovianBoundinTheBounds} in combination with the fact that from 
    \cite[Corollary~1.3 and Proposition~2.1 ]{Seifried2024HoelderMoments} we know there is a constant $C_2 >0$ satisfying
    \[
        \E \Big[ \sup_{\ell\in\unboundedTimeGrid} \sup_{t\in[\ell,\ell+h]} \big| Y_t - Y_{\ell} \big|^2 \Big] 
        \leq \E \Big[ \sup_{\ell\in\unboundedTimeGrid} \sup_{ \substack{ s,t\in[\ell,\ell+h] \\ s\leq t\leq\tau } } \big| Y_{t} - Y_{S} \big|^2 \Big] \leq C_2 h^{2 (\frac{1}{2} - \frac{1}{\alpha} - \varepsilon)} \qedhere
    \]
\end{proof}
\begin{theorem}[Markovian Classic Error Bound]\label{thm:MarkovianFBSDE_ClassicBound}
    Assume \hyperref[sde:assumption]{(SDE)}, \hyperref[fbsde_bsde:assumption]{(BSDE)}, \hyperref[stepsize:assumption]{(h)} and \hyperref[domain_ratio:assumption]{($\domain$-R)} with $q_1 = 32 $ and $q_2 = 4$. 
    Assume \eqref{eqnFBSDE} has a solution $(Y,Z)\in\mathcal{S}^4\times\mathcal{H}^4$ with $Z$ satisfying the integrability condition \eqref{eqn:MainResult_AssumptionMartingaleIntegrand}.
    Then for each $\varepsilon > 0$ there is $C >0$ (not depending on $h$) such that
    \begin{multline}
        \sup_{\ell\in\unboundedTimeGrid} \E \Big[ \sup_{t\in[\ell,\ell+h]} \big| Y_t - \Y_{\ell} \big|^2 \Big] \leq C h^\frac{1}{2} \\
        + C \E \big[\ |\tau-\StoppDiscrete|^{4} \big]^\frac{1}{4} 
        + C \E \big[\ |\tau-\StoppDiscrete|^{8 (1-\varepsilon) } \big]^\frac{1}{4}
        + C \E \big[\ |\tau-\StoppDiscrete|^{8 (\frac{1}{2}-\varepsilon) } \big]^\frac{1}{4} \closeEqn
    \end{multline}
\end{theorem}
\begin{proof}
    From the first bound of Corollary~\ref{corollary:MainResults_AfterHoelder} we obtain a constant $C_1 >0$ such that
    \begin{equation}
        \E \Big[ \sup_{t\in[\ell,\ell+h]} \big| Y_t - \Y_{\ell} \big|^2 \Big] 
        \leq 2 \E \Big[ \sup_{t\in[\ell,\ell+h]} \big| Y_t - Y_{\ell} \big|^2 \Big] 
        + 2C_1 h + 2 C_1 \mathcal{B} (1; X, Y)
    \end{equation}
    Analogously to the second part of the proof of Corollary~\ref{corollary:BSDE_Rate_Classic} (see \eqref{eqn:Yregularity} and Remark~\ref{remark:ItoJensenBound}) there is a constant $C_2>0$ (not depending on $h$) such that
    \begin{equation}
        \E \Big[ \sup_{t\in[\ell,\ell+h]} \big| Y_t - Y_{\ell} \big|^2 \Big] \leq C_1 h^\frac{1}{2} \quad\text{for all } \ell\in\unboundedTimeGrid
    \end{equation}
    Finally we use Proposition~\ref{prop:MarkovianBoundinTheBounds} to bound $\mathcal{B} (1; X, Y)$ and obtain the claim.
\end{proof}

\section{Hölder-Regularity Type Bounds for Euler-Maruyama}\label{sec:Hoelder_Bounds_Euler}

We introduce a continuous and adapted interpolation for the Euler-Maruyama sequence \eqref{def:EulerMaruyamaSequence} which is typically invoked for the analysis of the scheme, see e.g. \cite{KloedenPlaten};
\begin{equation}\label{def:EM_ContinuousInterpolation}
    \Xc_{t} \defined x + \int_{0}^{t} \mu \big( \X_{\phi(s)} \big) \de s + \int_{0}^{t} \sigma \big(\X_{\phi(s)}\big) \de W_s \quad t\geq 0
\end{equation}

\begin{proposition}\label{prop:EM_Rate_Gronwall}
    Assume \hyperref[sde:assumption]{(SDE)}.
    Let $p>1$ and $\Stopp$ be a stopping time.  
    If there is $\rho>0$ such that    
    $ \rho > 4p\ d (6\lipschitz{\mu} + 12p\lipschitz{\sigma}^2)$ and $\exp(\rho \Stopp) \in L^1 $,
    then for any $\varepsilon > 0$ there is $C > 0$ such that
    \begin{equation}
        \E \Big[ \sup_{\ell\in\unboundedTimeGrid} \sup_{t \in [\ell\wedge\Stopp, (\ell +h)\wedge\Stopp]} \big\| X_t - \Xc_t \big\|^p \Big] \leq C \big( h^\frac{p}{2} + h^{p(\frac{1}{2} - \varepsilon)} \big) 
        \quad\text{for all } h\in(0,1) \closeEqn
    \end{equation}
\end{proposition}
\begin{proof}
    Let $t\geq 0$. From the dynamics of $X$ and $\Xc$ we have that
    \begin{equation}
        X_{t\wedge\Stopp} - \Xc_{t\wedge\Stopp} = X_{\Stopp} - \Xc_{\Stopp} - \int_{t\wedge\Stopp}^{\Stopp} \mu (X_s) - \mu (\X_{\phi(s)}) \de s - \int_{t\wedge\Stopp}^{\Stopp} \sigma (X_s) - \sigma (\X_{\phi(s)}) \de W_s
    \end{equation}
    Since $\sigma$ has compact support the stochastic integral is a uniform integrable martingale.\footnote{
	   E.g. due to a conditional version of a Wald lemma for Brownian motion, see e.g. \cite[Lemma B.4]{Schlegel2024ProbabilisticShape}.
	} 
    Thus
    \begin{equation}
        X_{t\wedge\Stopp} - \Xc_{t\wedge\Stopp} = \E_{t} \Big[ X_{\Stopp} - \Xc_{\Stopp} - \int_{t\wedge\Stopp}^{\Stopp} \mu (X_s) - \mu (\X_{\phi(s)}) \de s \Big]
    \end{equation}
    and with the Lipschitz continuity of $\mu$ we further estimate (observe that $\Xc_s$ is squeezed in)
	\begin{equation}
		\|X_{t\wedge\Stopp}-\Xc_{t\wedge\Stopp}\| \leq \E_t \Big[ \|X_{\Stopp} - \Xc_{\Stopp}\| + \lipschitz{\mu} \Stopp \sup_{\ell\in\unboundedTimeGrid}\sup_{ \substack{r,s\in[\ell,\ell+h]\\r<s\leq\Stopp} } \| \Xc_s-\Xc_r \| + \lipschitz{\mu} \int_{t\wedge\Stopp}^{\Stopp} \|X_s-\Xc_s\| \de s\Big]
	\end{equation}
    We use a stochastic Gronwall lemma, see e.g. \cite[Lemma A.1]{Schlegel2024ProbabilisticShape}, to obtain
	\begin{equation}
		\|X_{t\wedge\Stopp}-\Xc_{t\wedge\Stopp}\| \leq \E_t \bigg[ \exp(\lipschitz{\mu}\Stopp) \Big( \|\Xc_{\Stopp}-X_{\Stopp}\| + \lipschitz{\mu} \Stopp \sup_{ \ell \in \unboundedTimeGrid }\sup_{ \substack{r,s\in[\ell,\ell+h]\\r<s\leq\Stopp} } \| \Xc_s-\Xc_r \| \Big)\bigg].
	\end{equation}
    Use the preceding bound together with Doob's $L^p$-inequality\footnote{
        There is a constant $C_p>0$ such that for any non-negative $A\in L^p$ it holds that
        \begin{equation}
            \E \Big[ \sup_{u\in\unboundedTimeGrid} \sup_{t\in[u,u+h]} \E_t[A]^p \Big] = \E \Big[ \sup_{t\geq 0} \E_t [A]^p \Big] \leq C_p \E \big[A^p\big]
        \end{equation}
    } 
    and an elementary bound;
    \begin{multline}
		\E \Big[ \sup_{u\in\unboundedTimeGrid} \sup_{t\in[ u,u+h]} \| X_{t\wedge\Stopp} -\Xc_{t\wedge\Stopp} \|^p \Big] \\
		\leq 2^{p-1}C_p \E \bigg[ \exp(p\lipschitz{\mu}\Stopp) \Big( \|\Xc_{\Stopp}-X_{\Stopp}\|^p + \lipschitz{\mu}^p \sup_{\ell\in\unboundedTimeGrid } \sup_{ \substack{r,s\in[\ell,\ell+h]\\r<s\leq\Stopp} } {\| \Stopp\Xc_s-\Stopp\Xc_r \|}^p \Big) \bigg] 
	\end{multline}
    With the Cauchy-Schwarz inequality and subadditivity of the square-root we conclude
    \begin{multline}\label{eqn:proof_Sec9_EQ1}
        \E \big[ \sup_{u\in\unboundedTimeGrid} \sup_{t\in[ u,u+h]} \| X_{t\wedge\Stopp} -\Xc_{t\wedge\Stopp} \|^p \big] 
        \leq 2^{p-1} C_p \E \big[ \exp(2p\lipschitz{\mu}\Stopp)\big]^\frac{1}{2} \bigg( \E \big[ \|\Xc_{\Stopp}-X_{\Stopp}\|^{2p} \big]^\frac{1}{2} \\
        + \lipschitz{\mu}^p \E \Big[ \sup_{\ell\in\unboundedTimeGrid } \sup_{ \substack{r,s\in[\ell,\ell+h]\\r<s\leq\Stopp} } {\| \Stopp\Xc_s-\Stopp\Xc_r \|}^{2p} \Big]^\frac{1}{2} \bigg)
    \end{multline}
    The exponential moment of $\Stopp$ exists by assumption. 
    Moreover the strong bound of the Euler-Maruyama scheme (see Lemma~\ref{lemmma:BGG_DiffusionEulerDeviation_ForStoppingTime}) yields $C_1 > 0$ such that
    $ \E [ \|\Xc_{\Stopp}-X_{\Stopp}\|^{2p} ]^\frac{1}{2} \leq C_1 h^\frac{p}{2} $
    and for given $\varepsilon >0$ we invoke Lemma~\ref{lemma:StoppScaledEuler_PathRegularity} to obtain a constant $C_2>0$ satisfying
    \begin{equation}
        \E \Big[ \sup_{\ell\in\unboundedTimeGrid } \sup_{ \substack{r,s\in[\ell,\ell+h]\\r<s\leq\Stopp} } {\| \Stopp\Xc_s-\Stopp\Xc_r \|}^{2p} \Big]^\frac{1}{2} \leq C_2 h^{p (\frac{1}{2}-\varepsilon)}
    \end{equation}
    Combine everything to further estimate \eqref{eqn:proof_Sec9_EQ1} and obtain the claim for an increased $C>0$.
\end{proof}

\begin{lemma}\label{lemma:StoppScaledEuler_PathRegularity} 
    Assume \hyperref[sde:assumption]{(SDE)}. Let $\Stopp$ be a finite, non-negative random variable and assume there is $m>0$ such that $\exp (m\Stopp) \in L^1 $. Then for $\varepsilon >0$ and $q > 2$ there is $C>0$ such that
    \begin{equation}
        \E \Big[ \sup_{\ell\in\unboundedTimeGrid} \sup_{s,t\in [\ell\wedge\Stopp,(\ell+h)\wedge\Stopp]} \| \Stopp \Xc_s - \Stopp \Xc_t \|^q \Big] \leq C h^{q(\frac{1}{2} - \varepsilon) } \quad \text{for all } h \in (0,1) . \closeEqn
    \end{equation}
\end{lemma}
\begin{proof}
    Following arguments of \cite{Seifried2024HoelderMoments} we consider the (non-adapted) process $(\Stopp \Xc_t)_{t\geq 0}$ and use the integrability of a local Hölder coefficient.\footnote{
        Similar to \cite[Proposition 2.1 and Example 2.2]{Seifried2024HoelderMoments} we verify Kolmogorov's tightness criterion is satisfied for all $q>2$ with constant ratio $\tfrac{1}{2}$. With this at hand the claim is a consequence of \cite[Corollary 1.3.c]{Seifried2024HoelderMoments}.
    }
    
    Let $p>2$ be fixed for now. Let $s,t\geq 0$. Using the Burkholder-Davis-Gundy and Jensen inequalities together with an elementary bound we compute
    \begin{equation}
        \E \big[ \|\Xc_{s} - \Xc_{t}\|^p \big]\! \leq\! 2^{p-1}\! \bigg(\! |s-t|^{p-1} \E \Big[ \int_{s}^{t}\! \|\mu (\X_{\phi(r)})\|^p \de r \Big] + |s-t|^{\frac{p}{2} - 1} C_p\ \E \Big[ \int_{s}^{t}\! \|\sigma(\X_{\phi(r)})\|^p \de r \Big]\! \bigg)
    \end{equation}
    By Assumption \hyperref[sde:assumption]{(SDE)} $\mu$ and $\sigma$ are compactly supported, hence there are $\supremum{\mu}, \supremum{\sigma} > 0$ such that
    \begin{equation}
        \E \big[ \|\Xc_{s} - \Xc_{t}\|^p \big] \leq 2^{p-1} \big( \supremum{\mu} |s-t|^p + C_p \supremum{\sigma} |s-t|^\frac{p}{2} \big)
    \end{equation}
    Particularly for $|s-t| \leq 1$ we obtain a constant $K_p >0$ satisfying
    \begin{equation}
        \E \big[ \|\Xc_{s} - \Xc_{t}\|^p \big] \leq 2^{p-1} (\supremum{\mu} + C_p \supremum{\sigma} ) |s-t|^\frac{p}{2} \defined K_p |s-t|^\frac{p}{2}
    \end{equation}
    Moreover using the Cauchy-Schwarz inequality we have
    \begin{equation}
        \E \big[ \|\Stopp \Xc_{s} - \Stopp \Xc_{t}\|^p \big] = \E \big[ \Stopp^q \|\Xc_{s} - \Xc_{t}\|^p \big] \leq \E \big[ \Stopp^{2p} \big]^\frac{1}{2} \E \big[ \|\Xc_{s} - \Xc_{t}\|^{2p} \big]^\frac{1}{2} 
    \end{equation}
    Since $p>2$ was arbitrarily fixed above we combine both observations to conclude that for each $q >2$ there is a constant\footnote{
        From our calculations above we know that $C(q) \defined K_{2q}^\frac{1}{2} \E \big[ \Stopp^{2q} \big]^\frac{1}{2}$ is an admissible choice. Particularly the constant is finite for each $q>2$ since by assumption $\Stopp$ has a positive exponential moment (see Lemma~\ref{lemma:A_exponentialMoment_ForAll_PowerMoments_SUPREMUM_VERSION}).
    }
    $C(q) \in (0,\infty)$ such that
    \begin{equation}
        \E \big[ \|\Stopp \Xc_{s} - \Stopp \Xc_{t}\|^q \big] \leq C(q) |s-t|^\frac{q}{2} \qquad \text{for all } |s-t| \leq 1
    \end{equation}
    Thus we verified that $(\Stopp \Xc_t)_{t\geq 0}$ satisfies Kolmogorov's criterion for all $q>2$ with constant ratio $\tfrac{1}{2}$. Let $\varepsilon >0$ and fix $q>2$ with \cite[Corollary 1.3]{Seifried2024HoelderMoments} we obtain a constant $C > 0$ such that
    \[
        \E \Big[ \sup_{\ell\in\unboundedTimeGrid} \sup_{ \substack{s,t\in [\ell,\ell+h] \\ s<t\leq\Stopp } } \| \Stopp \Xc_s - \Stopp \Xc_t \|^q \Big] \leq C h^{q(\frac{1}{2} - \varepsilon) } \qquad \text{for all } h\in (0,1). \qedhere
    \]
\end{proof}

\begin{lemma}\label{lemma:TheDelta_F_Bound}
    Assume \hyperref[sde:assumption]{(SDE)} and \hyperref[fbsde_bsde:assumption]{(BSDE)}. 
    Let $q\geq 1$ and $\Stopp$ denote a stopping time.  
    If there is $\rho>0$ such that    
    $ \rho > 32q\ d (6\lipschitz{\mu} + 96q\lipschitz{\sigma}^2) $ and $\exp(\rho \Stopp) \in L^1 $,
    then for every $\varepsilon > 0$ there is $C>0$ with
    \begin{equation}
        \E \bigg[ \Big(\int_{0}^{\Stopp} \big( f(X_s,Y_s) - f(\X_{\phi(s)},Y_s) \big)^2 \de s\Big)^{2q} \bigg]^\frac{1}{2q} \leq C h^{2 (\frac{1}{2}-\varepsilon)} 
        \quad\text{for all } h\in(0,1) \closeEqn
    \end{equation}
\end{lemma}
\begin{proof}
    Let $q\geq 1$ be fixed. With the Lipschitz continuity of $f$ we compute
    \begin{align}
        \int_{0}^{\Stopp} \big( f(X_s,Y_s) - f(\X_{\phi(s)},Y_s) \big)^2 \de s 
        & \leq  \lipschitz{f}^2 \sum_{t\in\unboundedTimeGrid} \int_{t\wedge\Stopp}^{(t+h)\wedge\Stopp} \| X_s - \X_{\phi(s)} \|^2 \de s \\
        & \leq \lipschitz{f}^2\ \Stopp \sup_{\ell\in\unboundedTimeGrid} \sup_{s\in [\ell\wedge\Stopp, (\ell+h)\wedge\Stopp]} \| X_s - \X_{\phi(s)} \|^2  
    \end{align}
    Thus with an elementary bound we obtain
    \begin{multline}
        \Big(\int_{0}^{\Stopp} \big( f(X_s,Y_s) - f(\X_{\phi(s)},Y_s) \big)^2 \de s\Big)^{2q} \\
        \leq 2^{4q} \lipschitz{f}^{4q}\ \Stopp^{2q}  \Big( \sup_{\ell\in\unboundedTimeGrid} \sup_{s\in [\ell\wedge\Stopp, (\ell+h)\wedge\Stopp]} \| X_s - \Xc_{s} \|^{4q}    
        + \sup_{\ell\in\unboundedTimeGrid} \sup_{s\in [\ell\wedge\Stopp, (\ell+h)\wedge\Stopp]} \| \Xc_{s} - \X_{\phi(s)} \|^{4q} \Big)
    \end{multline}
    Now fix $\varepsilon >0$. 
    We use the Cauchy-Schwarz inequality, Proposition~\ref{prop:EM_Rate_Gronwall} (with $p=8q$) and the assumption on the existence of exponential moments of $\Stopp$ to obtain $C_1>0$ such that
    \begin{align}\label{eqn:ProofTheDeltaF_EQ1}
        \E \Big[ \Stopp^{2q}\ \sup_{\ell\in\unboundedTimeGrid} \sup_{s\in [\ell\wedge\Stopp, (\ell+h)\wedge\Stopp]} \| X_s - \Xc_{s} \|^{4q}
        \Big] 
         & \leq \E \big[ \Stopp^{4q} \big]^\frac{1}{2} \E \Big[ \sup_{\ell\in\unboundedTimeGrid} \sup_{s\in [\ell\wedge\Stopp, (\ell+h)\wedge\Stopp]} \| X_s - \Xc_{s} \|^{8q} \Big]^\frac{1}{2} \notag \\
         & \leq C_1 \big( h^{4q} + h^{8q(\frac{1}{2}-\varepsilon)} \big)^\frac{1}{2} \leq C_1 h^{4q (\frac{1}{2} - \varepsilon)}
    \end{align}
    Moreover from Lemma~\ref{lemma:StoppScaledEuler_PathRegularity} we know that there is $C_2 >0$ such that
    \begin{align}\label{eqn:ProofTheDeltaF_EQ2}
        \E \Big[ \Stopp^{2q}\ \sup_{\ell\in\unboundedTimeGrid} \sup_{s\in [\ell\wedge\Stopp, (\ell+h)\wedge\Stopp]} \| \Xc_{s} - \X_{\phi(s)} \|^{4q} \Big]
        & \leq \E \Big[ \sup_{\ell\in\unboundedTimeGrid} \sup_{s\in [\ell\wedge\Stopp, (\ell+h)\wedge\Stopp]} \| \Stopp \Xc_{s} - \Stopp \X_{\phi(s)} \|^{4q} \Big] \notag \\
        & \leq C_2 h^{4q(\frac{1}{2} - \varepsilon)}
    \end{align}
    Combine \eqref{eqn:ProofTheDeltaF_EQ1} and \eqref{eqn:ProofTheDeltaF_EQ2} with the subadditivity of the $2q$-root to conclude the claim. 
\end{proof}

\begin{proposition}\label{prop:EM_Strong_DifferentStopp_Bound}
    Assume \hyperref[sde:assumption]{(SDE)} and let $h\in (0,1)$. Let $q\geq 4$ and let $\Stopp_1, \Stopp_2$ denote stopping times. If there is $\rho > 0$ such that
    $ \rho > 2q\ d (6\lipschitz{\mu} + 6q \lipschitz{\sigma}^2)$ and $\exp \big(\rho (\Stopp_1\vee\Stopp_2)\big) \in L^1 $
    then for every $\varepsilon > 0$ there is a constant $C>0$ such that
    \begin{equation}
        \E \big[ \| X_{\Stopp_1} - \Xc_{\Stopp_2} \|^q \big] \leq C h^{\frac{q}{2}} + C \E \big[\ |\Stopp_1-\Stopp_2|^{2q(1-\varepsilon)}\ \big]^\frac{1}{2} + C \E \big[\ |\Stopp_1-\Stopp_2|^{2q(\frac{1}{2}-\varepsilon)} \ \big]^\frac{1}{2} \closeEqn
    \end{equation}
\end{proposition}
\begin{proof}
    With an elementary bound we have
    \begin{equation}
        \| X_{\Stopp_1} - \Xc_{\Stopp_2} \|^q 
        \leq 2^{q-1} \| X_{\Stopp_1} - \Xc_{\Stopp_1} \|^q + 2^{q-1} \| \Xc_{\Stopp_1} - \Xc_{\Stopp_2} \|^q 
    \end{equation}
    From Lemma~\ref{lemmma:BGG_DiffusionEulerDeviation_ForStoppingTime} we know there is a constant $C_1 > 0$ such that
    \begin{equation}\label{eqn:Proof_Strong_DifferentStopp_Bound}
        \E \big[ \| X_{\Stopp_1} - \Xc_{\Stopp_1} \|^q \big] \leq C_1 h^\frac{q}{2}
    \end{equation}
    Moreover by the dynamics of $\Xc$ (see \eqref{def:EM_ContinuousInterpolation}) we have
    \begin{equation}
        \| \Xc_{\Stopp_1} - \Xc_{\Stopp_2} \|^q 
        \leq 2^{q-1} \Big\| \int_{\Stopp_1}^{\Stopp_2} \mu \big( \Xc_{\phi(r)} \big) \de r \Big\|^q + 2^{q-1} \Big\| \int_{\Stopp_1}^{\Stopp_2} \sigma \big( \Xc_{\phi(r)} \big) \de W_r \Big\|^q
    \end{equation}
    Subsequently we apply the global Hölder coefficient bound (see Corollary~\ref{corollary:HoelderMoments_TwoStopp}) to each summand in the bound. Following \cite[Example 2.2]{Seifried2024HoelderMoments} we verify that the processes
    $ B \defined \int_{0}^{\cdot} \mu \big( \Xc_{\phi(r)} \big) \de r$ and $M \defined \int_{0}^{\cdot} \sigma \big( \Xc_{\phi(r)} \big) \de W_r $
    satisfy Kolmogorov's criterion for the constant ratios $1$ resp. $\tfrac{1}{2}$ (in the sense of \eqref{def:KolmogorovCriterionRatio}). 
    Let $p\geq 2$. For the process $B$ we use Jensens's inequality together with the fact that $\mu$ is compactly supported by Assumption \hyperref[sde:assumption]{(SDE)} to conclude
    \begin{equation}
        \E \big[ |B_s - B_t|^p \big] \leq \supremum{\mu}^p\ |s-t|^p \quad\text{for all } s,t \geq 0
    \end{equation}
    Similarly since $\sigma$ has compact support we invoke the Burkholder-Davis-Gundy (with constant $C_p > 0$) and Jensen inequalities to obtain
    \begin{equation}
        \E \big[ |M_s - M_t|^p \big] \leq C_p \supremum{\sigma}^p\ |s-t|^{\frac{p}{2}} \quad\text{for all } s,t \geq 0
    \end{equation}
    and thus Kolmogorov's criterion holds as desired. 
    By assumption there is $m>0$ such that $\exp(m(\Stopp_1\vee\Stopp_2))\in L^1$. Whence we apply Corollary~\ref{corollary:HoelderMoments_TwoStopp} and obtain $C_2,C_3> 0$ such that
    \begin{equation}
        \E \big[ \|B_{\Stopp_1}-B_{\Stopp_2}\|^q \big] + \E \big[ \|M_{\Stopp_1}-M_{\Stopp_2}\|^q \big] \leq C_2 \E \big[\ |\Stopp_1-\Stopp_2|^{2q(1-\varepsilon)}\ \big]^\frac{1}{2} + C_3 \E \big[\ |\Stopp_1-\Stopp_2|^{2q(\frac{1}{2}-\varepsilon)} \ \big]^\frac{1}{2}
    \end{equation}
    Combine this with \eqref{eqn:Proof_Strong_DifferentStopp_Bound} to conclude that there is a constant $C>0$ such that
    \[
        \E \big[ \| X_{\Stopp_1} - \Xc_{\Stopp_2} \|^q \big] 
        \leq C h^{\frac{q}{2}} + C \E \big[\ |\Stopp_1-\Stopp_2|^{2q(1-\varepsilon)}\ \big]^\frac{1}{2} + C \E \big[\ |\Stopp_1-\Stopp_2|^{2q(\frac{1}{2}-\varepsilon)} \ \big]^\frac{1}{2} \qedhere
    \]
\end{proof}
%
\bibliographystyle{plain}
\bibliography{references}
\appendix
\section{A Global Hölder Coefficient Bound}\label{sec:HoelderCoeff}

We present a corollary of \cite[Theorem 1.1]{Seifried2024HoelderMoments}: If a process satisfies Kolmogorov's tightness criterion for increments of arbitrary magnitude with constant ratio, then differently stopped versions of said process can be bounded by the difference of the stopping times.
\begin{corollary}\label{corollary:HoelderMoments_TwoStopp}
    Let $\Process$ denote a process satisfying Kolmogorov's criterion with constant ratio $\alpha>0$, i.e.
    for all $p \geq 2$ there is a constant $C_p>0$ such that 
    \begin{equation}\label{def:KolmogorovCriterionRatio}
        \E \big[ \|\Process_t - \Process_s \|^{p} \big] \leq C_p |t-s|^{p\alpha}\quad \text{for all } s,t\geq 0 .
    \end{equation}
    Moreover let $\Stopp_1, \Stopp_2$ denote stopping times. Assume there is $m>0$ such that $\exp(m(\Stopp_1\vee\Stopp_2)) \in L^1$. Then for each $\varepsilon > 0$ there is a constant $C>0$ such that
    \begin{equation}
        \E \Big[ \sup_{s\geq 0} \| \Process_{s\wedge\Stopp_1} - \Process_{s\wedge\Stopp_2} \|^p \Big] \leq C \E \big[ \ |\Stopp_1 - \Stopp_2|^{2p (\alpha - \varepsilon)} \ \big]^\frac{1}{2} \closeEqn
    \end{equation}
\end{corollary}
\begin{proof}
    Observe 
    \begin{equation}
        \sup_{s\geq 0} \| \Process_{s\wedge\Stopp_1} - \Process_{s\wedge\Stopp_2} \| = \sup_{s\geq 0} \1_{ \{ s\geq \Stopp_1\wedge\Stopp_2 \} } \| \Process_{s\wedge\Stopp_1} - \Process_{s\wedge\Stopp_2} \|
        \leq \sup_{s,t \in [\Stopp_1\wedge\Stopp_2,\Stopp_1\vee\Stopp_2]} \| \Process_{s} - \Process_{t} \| 
    \end{equation}
    and thus for any $\gamma > 0$ we have
    \begin{equation}
        \sup_{s\geq 0} \| \Process_{s\wedge\Stopp_1} - \Process_{s\wedge\Stopp_2} \| \leq 
        |\Stopp_1 - \Stopp_2 |^\gamma \sup_{ \substack{s,t \in [\Stopp_1\wedge\Stopp_2,\Stopp_1\vee\Stopp_2] \\ s\neq t} } \frac{\| \Process_{s} - \Process_{t} \|}{|s-t|^\gamma} .
    \end{equation}
    Particularly for $p >1$ fixed we obtain 
    \begin{equation}
        \sup_{s\geq 0} \| \Process_{s\wedge\Stopp_1} - \Process_{s\wedge\Stopp_2} \|^p \leq |\Stopp_1 - \Stopp_2 |^{p\gamma} \ \bigg( \sup_{ \substack{s,t \in [0,\Stopp_1\vee\Stopp_2] \\ s\neq t} } \frac{\| \Process_{s} - \Process_{t} \|}{|s-t|^\gamma} \bigg)^p
    \end{equation}
    where the second factor contains the global Hölder coefficient (up to $\Stopp_1\vee\Stopp_2$). From \cite[Theorem 1.1]{Seifried2024HoelderMoments} we know a sufficient condition for the existence of the corresponding expectation. By assumption we know that Kolmogorov's criterion is satisfied for all $p\geq 2$ and from the existence of a exponential moment we particularly know that $\Stopp_1\vee\Stopp_2 \in L^q$ for all $q\geq 1$ (see Lemma~\ref{lemma:A_exponentialMoment_ForAll_PowerMoments_SUPREMUM_VERSION}). Whence (with the notation\footnote{
        See \cite[p.2]{Seifried2024HoelderMoments} for an explicit presentation for the dependency on $q$ of the constants $C_1(q), C_2(q)$ and $N(q)$. For the purpose in the present article we mention that $C_1(q),C_2(q)\to 1$ and $N(q)\to 0$ as $q\to\infty$.
    }
    from \cite{Seifried2024HoelderMoments}) for given $\varepsilon >0$ and $p \geq 2$ we choose $p_\varepsilon > 1$ such that
    \begin{equation}\label{eqn:HoelderMomentInterval_Eps}
        \tfrac{C_1(p_\varepsilon)}{p_\varepsilon} < \varepsilon < \tfrac{C_1(p_\varepsilon)+N(p_\varepsilon)}{p_\varepsilon} \qquad\text{and set } \gamma_\varepsilon \defined \alpha -\varepsilon
    \end{equation}
    and without restriction we further have $2p < C_2 (p_\varepsilon) p_\varepsilon$.\footnote{
    If this condition fails one may choose $p_\varepsilon < \Tilde{p}$ such that $2p < C_2 (\Tilde{p}) \Tilde{p}$. This might lead to $\varepsilon$ being greater than the interval \eqref{eqn:HoelderMomentInterval_Eps} of admissible Hölder exponents spanned by $\Tilde{p}$. In this case decrease $\varepsilon$. 
    }
    For this choice we invoke \cite[Theorem 1.1]{Seifried2024HoelderMoments} to obtain
    \begin{equation}
        \sup_{ \substack{s,t \in [0,\Stopp_1\vee\Stopp_2] \\ s\neq t} } \frac{\| \Process_{s} - \Process_{t} \|}{|s-t|^{\gamma_\varepsilon}} \in L^{2p}
    \end{equation}
    and now the claim follows from an application of the Cauchy-Schwarz inequality.
\end{proof}
\section{Strong Bounds of Euler-Maruyama Scheme}\label{sec:StrongEulerBounds_BGG}

We briefly recall well-known properties of the Euler-Maruyama scheme: Lemma~\ref{lemma:BGG_StrongEulerBound_LinearExponent} is an immediate corollary of \cite[Lemma A.2]{BGG_forward_exit} where we extract the dependency on $k$ in the bound\footnote{
    \cite{BGG_forward_exit} states the bound as product of a positive polynomial (in $k$) and the same (linear in $k$) exponential factor. See also \cite[Appendix A]{avikainen2007}. The proof uses a classical Gronwall argument; in contrast to the classic statement, see e.g. \cite[Theorem 10.2.2]{KloedenPlaten} (with quadratic $k$ in the exponent), the linear exponent is crucial for our analysis.
}
and Lemma~\ref{lemmma:BGG_DiffusionEulerDeviation_ForStoppingTime} is a ramification of \cite[Lemma 3.3]{BGG_forward_exit}.
\begin{lemma}\label{lemma:BGG_StrongEulerBound_LinearExponent}
    Let $p\geq 4$. For $h\in (0,1)$ let $\Xc = \Xc (h)$ denote the adapted interpolation \eqref{def:EM_ContinuousInterpolation} of the Euler-Maruyama sequence on the discrete-time grid $\unboundedTimeGrid$. Then it holds that
    \begin{equation}
        \E \Big[ \sup_{s\in[k,k+1]} \big| X_s - \Xc_s \big|^p \Big] \leq K (k+1)^\frac{p}{2} \exp \big( (k+1) pC_{\text{EM}}(p) \big)  h^\frac{p}{2} \quad\text{for all } k\in\nat_0
    \end{equation}
    where $C_{\text{EM}} (p) \defined d(6\lipschitz{\mu} + 3p \lipschitz{\sigma}^2)$ and $K>0$ is a constant (not depending on $k$ nor $h$). \close
\end{lemma}
\begin{proof}
    From the proof of \cite[Lemma A.2]{BGG_forward_exit} we know that
    \begin{equation}
        \E \Big[ \sup_{s\in[k,k+1]} \big| X_s - \Xc_s \big|^p \Big] \leq C\ Q(k,p)\ h^\frac{p}{2}
    \end{equation}
    where $C>0$ is a constant (not depending on $k$ nor $h$) and
    \[
        Q (k,p) \defined \big((k+1)^{\frac{p}{2}} + (k+1)^\frac{p}{4} \big) \exp \big( p (k+1) C_{\text{EM}} (p) \big) \leq 2 (k+1)^\frac{p}{2} \exp \big(  (k+1) p C_{\text{EM}} (p) \big). \qedhere
    \]
\end{proof}
\begin{lemma}\label{lemmma:BGG_DiffusionEulerDeviation_ForStoppingTime}
    Let $p\geq 2$ and let $\Stopp$ be a stopping time. If there is $\rho > 0$ such that
    $ \rho > 2p C_{\text{EM}} (2p)$ and $\exp (\rho\Stopp) \in L^1 $
    with $C_{\text{EM}}(q) \defined d (6\lipschitz{\mu} + 3q\lipschitz{\sigma}^2)$.
    Then there is $C>0$ such that
    \begin{equation}
        \E \big[ \|X_{\Stopp} -\Xc_{\Stopp}\|^p \big] \leq Ch^\frac{p}{2} . \closeEqn
    \end{equation}
\end{lemma}
\begin{proof}
    Let $h\in (0,1)$ be fixed. With the Markov inequality we have
    \begin{equation}
        \prob \big[ \Stopp \in [k,k+1) \big] \leq \prob [ k \leq \Stopp ] \leq \E \big[ \exp (\rho\Stopp) \big] \exp (-\rho k) \quad \text{for all } k\in\nat_0
    \end{equation}
    and by assumption $\beta \defined \E [ \exp (\rho\Stopp) ] < \infty$. With monotone convergence, the Cauchy-Schwarz inequality and Lemma~\ref{lemma:BGG_StrongEulerBound_LinearExponent} with constant $K>0$ we compute
    \begin{align}
        \E \big[ \|X_{\Stopp} - \Xc_{\Stopp}\|^p \big] & \leq \sum_{k\in\nat_0} \E \Big[ \1_{[k,k+1)} (\Stopp)\ \sup_{s\in[k,k+1]} \| X_s - \Xc_s \|^p \Big] \\
        & \leq \sum_{k\in\nat_0} \prob \big[ \Stopp \in [k,k+1) \big]^\frac{1}{2}\
        \E \Big[ \sup_{s\in[k,k+1]} \| X_s - \Xc_s \|^{2p} \Big]^\frac{1}{2} \\
        & \leq \beta^\frac{1}{2} K^\frac{1}{2} h^\frac{p}{2} \sum_{k\in\nat_0} \exp \big( -\tfrac{1}{2}\rho k \big) (k+1)^p \exp  \big( \tfrac{1}{2} (k+1) 2pC_{\text{EM}}(2p) \big)  \\
        & = \beta^\frac{1}{2} K^\frac{1}{2} \exp \big(\tfrac{\rho}{2} \big) h^\frac{p}{2} \sum_{k\in\nat} k^p \exp \Big( \tfrac{1}{2} k \big(-\rho + 2pC_{\text{EM}}(2p) \big) \Big)
    \end{align}
    The assumption on $\rho$ ensures the exponent is negative and thus the series is convergent.\footnote{
        Convergence is due to Cauchy's radical convergence test: $k^\frac{p}{k} \to 1$ as $k\to \infty$ and $\exp ( \tfrac{1}{2} (-\rho + 2p C_{\text{EM}} (2p) ) ) < 1$. 
    }
    To complete the proof we set
    \[
        \beta^\frac{1}{2} K^\frac{1}{2} \exp \big( \tfrac{\rho}{2} \big) \sum_{k\in\nat_0} k^p \exp \Big( \tfrac{1}{2} k \big(-\rho + 2pC_{\text{EM}}(2p) \big) \Big) \defined C < \infty . \qedhere
    \]
\end{proof}
\section{Supplements}\label{sec:generalSupplementsAndPathwiseGronwall}

\begin{lemma}[Pathwise Discrete Gronwall]\label{lemma:DiscreteGronwall}
	Let $(A_t)_{t\in\unboundedTimeGrid}, (B_t)_{t\in\unboundedTimeGrid}$ non-negative, $(\F_t)_{t\in\unboundedTimeGrid}$ adapted and let $\Stopp$ denote a $\unboundedTimeGrid$-valued finite stopping time. Assume $\sum_{\ell\in\unboundedTimeGrid \cap [0, \Stopp)} h B_\ell \in L^1$.
    
    Assume $A_t \in L^1$ for each $t\in\unboundedTimeGrid$ and that there is a non-negative $\xi \in L^1(\F_{\Stopp}) $ such that
    \begin{equation}\label{lemma:pathwiseDiscreteGronwall_AssumptionConstantAfterStopp}
        \1_{ \{ t \geq \Stopp \} } A_t = \1_{ \{ t \geq \Stopp \} } \xi \quad\text{for all } t\in\unboundedTimeGrid
    \end{equation}
    and moreover the following conditional convergence is satisfied
    \begin{equation}\label{eqn:AssumptionASConditionalConvergenceTerminal}
         \lim_{I\to\infty} C^{\frac I h}\E_t \big[ \1_{ \{ t+I<\Stopp \} } A_{t+I} \big] = 0 \quad\text{for all } t\in\unboundedTimeGrid .
    \end{equation}
    If there are constants $C, K\geq 0$ such that
	\begin{equation}\label{lemma:pathwiseDiscreteGronwall_AssumptionConditionalBound}
		\1_{\{t < \Stopp\}} A_t \leq C \E_t \big[ A_{t+h} \big] + Kh B_t \quad\text{for all }t\in\unboundedTimeGrid 
	\end{equation}
	then
    \begin{equation}
        \1_{\{t<\Stopp\}} A_t  \leq 
        \E_t \Big[ C^{\frac{\Stopp-t}{h}} \xi + Kh \sum_{\ell\in\unboundedTimeGrid [t+h, \Stopp)} C^{ \frac{\ell-t}{h}} B_\ell \Big] \quad\text{for all }t\in\unboundedTimeGrid . \closeEqn
    \end{equation}
\end{lemma}
\begin{proof}
    Throughout the proof let $t\in\unboundedTimeGrid$ be fixed.

    For $I\in\unboundedTimeGrid\cap[h,\infty)$ we consider the $\F_{t+I}$ measurable random variables 
    \begin{multline}\label{proof:DiscreteGronwall_EQ1}
        X_t (I) \defined \1_{ \{ t+I < \Stopp \} } C^{\frac{I}{h}} A_{t+I} \\ 
        + \sum_{\ell\in\unboundedTimeGrid\cap [t+h, t+I] } \1_{ \{ \ell = \Stopp \} } C^{\frac{\ell - t}{h}} \xi \ +\  Kh \sum_{\ell\in\unboundedTimeGrid\cap [t,t+I - h] } \1_{ \{ \ell < \Stopp \} } C^{\frac{\ell-t}{h}} B_{\ell} .
    \end{multline}
    
    \step{1} We show \eqref{lemma:pathwiseDiscreteGronwall_AssumptionConstantAfterStopp} and \eqref{lemma:pathwiseDiscreteGronwall_AssumptionConditionalBound} ensure that
    \begin{align}
        \1_{ \{ t < \Stopp \} } A_t \leq \E_t \big[ X_t (h) \big] 
        \qquad\text{and}\qquad
        X_t (J) \leq \E_{t+J} \big[ X_t (J+h) \big]\ \text{for all } J\in\unboundedTimeGrid\cap[h,\infty) .
    \end{align}
    Combine the assumptions \eqref{lemma:pathwiseDiscreteGronwall_AssumptionConstantAfterStopp} and \eqref{lemma:pathwiseDiscreteGronwall_AssumptionConditionalBound} to compute
    \begin{align}
        \1_{ \{ t < \Stopp \} } A_t & \leq \1_{ \{ t<\Stopp \} } \Big( C \E_{t} \big[ \1_{ \{ t+h < \Stopp \} } A_{t+h} \big] + C \E_{t} \big[ \1_{ \{ t+h = \Stopp \} } \xi \big] + Kh B_t \Big) \\
        & = \E_{t} \Big[ \1_{ \{ t+h < \Stopp \} } C A_{t+h} + \1_{ \{ t+h = \Stopp \} } C \xi + \1_{ \{ t<\Stopp \} } Kh B_t \Big] \\
        & = \E_t \big[ X_t (h) \big].
    \end{align}
    Observe that $ \1_{ \{\ell = \Stopp\} } \xi = \E [\1_{ \{\ell = \Stopp\} } \xi \ |\ \F_{\Stopp} ] = \E [\1_{ \{\ell = \Stopp\} } \xi \ |\ \F_{\ell} ] $ for any $\ell\in\unboundedTimeGrid$. Let $J\in\unboundedTimeGrid\cap[h,\infty)$ and use the preceding bound for $\ell = t+J$ to compute
    \begin{align}
        X_t (J) & = \1_{ \{ t+J < \Stopp \} } C^{\frac{J}{h}} A_{t+J}\ + \sum_{\ell\in\unboundedTimeGrid\cap [t+h, t+J] } \1_{ \{ \ell = \Stopp \} } C^{\frac{\ell - t}{h}} \xi \ +\  Kh \sum_{\ell\in\unboundedTimeGrid\cap [t,t+J - h] } C^{\frac{\ell-t}{h}} B_{\ell} \\
        & \leq \E_{t+J} \Big[ \1_{ \{ t+J+h < \Stopp \} } C^{\frac{J+h}{h}} A_{t+J+h} + \1_{ \{ t+J+h = \Stopp \} } C^{\frac{J+h}{h}} \xi + \1_{ \{ t+J<\Stopp \} } Kh C^{\frac{J}{h}} B_t \Big] \\
        & \hspace*{2.0cm} + \sum_{\ell\in\unboundedTimeGrid\cap [t+h, t+J] } \1_{ \{ \ell = \Stopp \} } C^{\frac{\ell - t}{h}} \xi \ +\  Kh \sum_{\ell\in\unboundedTimeGrid\cap [t,t+J - h] } C^{\frac{\ell-t}{h}} B_{\ell} \\
        & = \E_{t+J} \big[ X_t ( J+h) \big]. \closeEqn
    \end{align} 

    \step{2} We show 
    \begin{equation}
       \lim_{I\to\infty} \E_{t} [X_t (I)] = \E_t \Big[ C^{\frac{\Stopp-t}{h}} \xi + Kh \sum_{\ell\in\unboundedTimeGrid [t+h, \Stopp)} C^{ \frac{\ell-t}{h}} B_\ell \Big] .
    \end{equation}
    
    By definition (see \eqref{proof:DiscreteGronwall_EQ1}) we have
    \begin{align}
        \lim_{I\to\infty} X_t (I) & = \sum_{\ell\in\unboundedTimeGrid [t+h, \infty)} \1_{ \{ \ell=\Stopp \} } C^{\frac{\ell-t}{h}} \xi + Kh \sum_{\ell\in\unboundedTimeGrid [t+h, \infty)} \1_{ \{ \ell < \Stopp \} } C^{ \frac{\ell-t}{h}} B_\ell\\
        & = C^{\frac{\Stopp-t}{h}} \xi + Kh \sum_{\ell\in\unboundedTimeGrid [t+h, \Stopp)} C^{ \frac{\ell-t}{h}} B_\ell .
    \end{align}
    With monotone convergence we obtain 
    \begin{multline}
        \lim_{I\to\infty} \E_t \Big[ \sum_{\ell\in\unboundedTimeGrid\cap [t+h, t+I] } \1_{ \{ \ell = \Stopp \} } C^{\frac{\ell - t}{h}} \xi \ 
        +\  Kh \sum_{\ell\in\unboundedTimeGrid\cap [t,t+I - h] } \1_{ \{ \ell < \Stopp \} } C^{\frac{\ell-t}{h}} B_{\ell} \Big] \\
        = \E_t \Big[ C^{\frac{\Stopp-t}{h}} \xi + Kh \sum_{\ell\in\unboundedTimeGrid [t+h, \Stopp)} C^{ \frac{\ell-t}{h}} B_\ell \Big]
    \end{multline}
    and using \eqref{eqn:AssumptionASConditionalConvergenceTerminal} we have that the projection of the remainder of $X(I)$ vanishes, i.e. 
    \begin{equation}
        \E_t \big[\1_{ \{ t+I < \Stopp \} } C^{\frac{I}{h}} A_{t+I}\big] \to 0 \quad\text{as } I\to\infty. \closeEqn
    \end{equation}
    The overall statement now is an immediate consequence of the preceding tow steps;
    \begin{align}
        \1_{ \{ t<\Stopp \} } A_t & \leq \sup_{I\in\unboundedTimeGrid\cap[h,\infty)} \E_{t} \big[ X_t (I) \big] = \lim_{I\to\infty} \E_{t} \big[ X_t (I) \big]. \qedhere
    \end{align}
\end{proof}

\begin{lemma}\label{lemma:RepresentationConditionalExpectation}
    Let $A\in\mathcal{A}$ and $\mathcal{G}\subset\mathcal{A}$. Then for random variables $Y\geq X\geq 0$ it holds that
    \begin{equation}
        \int_{0}^{\infty} \prob \big[ A,\ Y > t \geq X \ \big|\ \mathcal{G} \big] \de t = \E \big[ \1_{A} (Y-X) \ \big|\ \mathcal{G} \big] = \int_{0}^{\infty} \prob \big[ A,\ Y-X\geq t\ \big|\ \mathcal{G} \big] \de t. \closeEqn 
    \end{equation}
\end{lemma}

\begin{lemma}\label{lemma:A_exponentialMoment_ForAll_PowerMoments_SUPREMUM_VERSION}
    Let $\ell>0$ and $(A_t)_{t\in(0,\ell]}$ a family of random variables. If there is $m>0$ such that
    \begin{equation}
        \sup_{h\in(0,\ell]} \exp(mA_h) \in L^1  
        \qquad\text{then}\qquad 
        \sup_{h\in(0,\ell]} A_h^p \in L^1 \quad\text{for all } p\geq 1. \closeEqn
    \end{equation}
\end{lemma}
\begin{proof}
    We show the claim for $p\in\nat$. For general $p>1$ one repeats the same argument for $\lceil p \rceil \in \nat$ and 
    concludes $\sup_{h\in(0,\ell]} A_h^p \in L^1$ due to inclusion $L^p (\prob) \subset L^{\lceil p \rceil} (\prob) $. 
    Let $p\in\nat$. By expansion of the exponential series we have 
    \[
        \tfrac{m^p}{p!} \E \big[ \sup_{h\in(0,\ell]} A_h^p \big] 
        \leq \E \Big[ \sup_{h\in(0,\ell]} \ \sum_{q\in\nat_0} \tfrac{m^q}{q!} A_h^q \Big] 
        = \E \Big[ \sup_{h\in(0,\ell]} \exp(mA_h) \Big] < \infty . \qedhere
    \]
\end{proof}

%
\end{document}